\documentclass[11pt]{amsproc}
\pdfoutput=1
\usepackage{amsmath,amssymb,amsthm,eucal, eufrak,amscd,tikz,tikz-cd,mathtools}
\usepackage{xcolor}
\usepackage{enumitem}
\usepackage[shortalphabetic]{amsrefs}
\usepackage{xypic}
\usepackage{graphicx}
\definecolor{allrefcolors}{rgb}{0,0.2,0.5}
\usepackage[linktocpage=true,colorlinks=true,allcolors=allrefcolors,bookmarksopen,bookmarksdepth=3]{hyperref}
\usepackage[margin=1.25in]{geometry}

\def\bd{\partial}

\def\Z{{\mathbb Z}}
\def\R{{\mathbb R}}
\def\C{{\mathbb C}}

\def\K{{\mathbf{k}}}
\def\w{\mathcal{W}(E)}

\def\e{\epsilon}
\def\cc{\mathcal{C}}
\def\dd{\mathcal{D}}
\def\rmod{\mathrm{mod\!-\!}}

\def\e{\epsilon}

\def\r#1{\mathrm{#1}}

\def\mc#1{\mathcal{#1}}

\def\w{\mathcal{W}}
\def\rw{\mathcal{RW}}

\def\d{\delta}
\def\D{\Delta}

\def\ainf{A_\infty}

\def\bd{\partial}
\def\z2{\Z / 2\Z}

\def\id{\mathrm{id}}

\def\p{\partial}

\def\ob{\mathrm{ob\ }}
\def\ch{\mathrm{Ch}}
\def\mod{\mathrm{Mod}}
\def\perf{\mathrm{Perf}}
\def\coh{\mathrm{Coh}}
\def\sing{\mathrm{Sing}}
\def\calk{\mathrm{Calk}}
\def\fun{\mathrm{Fun}}
\def\cone{\mathrm{Cone}}
\def\cinf{\widehat{\cc}_{\infty}}
\def\winf{\widehat{\w}_{\infty}}

\newtheorem{lem}{Lemma}[section]
\newtheorem{prop}[lem]{Proposition}
\newtheorem{thm}[lem]{Theorem}
\newtheorem{cor}[lem]{Corollary}

\newtheorem{defn}[lem]{Definition}

\newtheorem{ques}[lem]{Question}
\newtheorem{rem}[lem]{Remark}
\def\e{\epsilon}

\theoremstyle{remark}

\numberwithin{equation}{section}
\begin{document}
\begin{abstract}
This paper constructs and studies the Rabinowitz (wrapped) Fukaya category, a categorical invariant of exact cylindrical Lagrangians in a Liouville manifold whose cohomological morphisms, ``Rabinowitz wrapped Floer homology groups" measure the failure of wrapped Floer cohomology to satisfy Poincare duality (and in particular vanish for any pair with at least one compact Lagrangian).

    Our main result, answering a conjecture of Abouzaid, relates the Rabinowitz and usual wrapped Fukaya category by way of a general construction introduced by Efimov, the {\em categorical formal punctured neighborhood of infinity}. As an application, we show how Rabinowitz Fukaya categories can be fit into - and in particular often computed in terms of - mirror symmetry.
\end{abstract}

\title[Rabinowitz Fukaya categories]{Rabinowitz Fukaya categories and the categorical formal punctured neighborhood of infinity}
\author{Sheel Ganatra and Yuan Gao and Sara Venkatesh}

\maketitle

\section{Introduction}\label{sec:introduction}

This goal of this paper is to develop, and relate to existing structures in (wrapped)
Floer theory the {\em Rabinowitz (wrapped) Fukaya category}, a chain-level categorical
open-string (Lagrangian) analogue of the closed string {\em Rabinowitz Floer
homology}. 

Rabinowitz Floer homology, introduced in \cite{CF},
is a Floer-theoretic invariant of a contact-type $H$ hypersurface in a
Liouville manifold $X$, generated as a chain-complex by {\em all} (positive,
negative, constant) characteristics of $H$, and with differential counting
certain Floer-type cylinders in $X$. Our focus here will be on the case
$H$ is the boundary of a domain completing to $X$, in which case the result
only depends on $X$ and will be simply denoted $RFH^*(X)$.
The group $RFH^*(X)$ relates to the {\em
symplectic cohomology} $SH^*(X)$ and its dual theory, {\em symplectic homology} $SH_*(X)$ by way of a long-exact sequence \cite{CFO}. Although there are multiple direct definitions of $RFH^*(X)$ (e.g., using the Morse homology of Rabinowitz's action functional \cite{CF} or ``V-shaped Hamiltonians'' \cite{CFO}), 
following \cite{venkatesh} one can take \cite{CFO}'s long-exact sequence as a defining property of
$RFH^*(X)$: it is the homology group measuring the failure of Poincar\'{e}
duality to hold in symplectic cohomology (with respect to a canonical geometrically defined copairing that exists on the theory), i.e., as the homology of the cone of a (chain-level version) of a 
canonical connecting map $c: SH_*(X) \to SH^*(X)$ called the {\em continuation map}. In the hypothetical absence of closed characteristics, this reduces to the usual way in which $H^*(\bd_{\infty} X)$ is determined by the canonical (chain-level) map from compactly supported cohomology $H^*_c(X)$ to ordinary cohomology $H^*(X)$.
On the other hand, for a Liouville manifold $RFH^*(X)$ vanishes if and only if $SH^*(X)$ vanishes \cite{ritterTQFT}, which can be viewed as a measure of the degenerate nature of the map $c$ for Liouville manifolds (a fact which does not generalize to non-exact settings in which $SH^*(X)$ and $RFH^*(X)$ can be defined 
\cite{ritternegative,alberskang}). 

In the open string setting, one can analogously define the Rabinowitz wrapped Floer homology $RW^*(K,L)$ of a pair of (exact, non-compact, cylindrical at infinity) Lagrangians $K,L$ in a Liouville manifold as the theory measuring the failure of Poincar\'{e} duality in the usual wrapped Floer cohomology $HW^*(K,L)$, i.e., as the cohomology of the cone of the canonical chain-level map from wrapped Floer homology $HW_*(K,L)$ to wrapped Floer cohomology $HW^*(K,L)$. 
(a direct definition of $RW^*(K,L)$ in the spirit of \cite{CF}'s definition of $RFH^*(X)$ was also given in \cite{merry}; and the analogous long-exact sequence 
\cite{bounya, dahinden} compares this definition to the way we define these groups).
This theory can be computed as a chain complex by all characteristics between the Legendrian boundaries $\bd_{\infty} K$, $\bd_{\infty} L$. Notably, it vanishes if either $K$ or $L$ are compact (in which case $HW^*(K,L)$ coincides with the ordinary unwrapped Floer cohomology $HF^*(K,L)$, which is finite dimensional).

Importantly, the theory $RFH^*(X)$ admits an algebra structure, and its Lagrangian analogue admits composition maps $RW^*(K,L) \otimes RW^*(L,N) \to RW^*(K,N)$ \cite{CO}. From the perspective of viewing these ``boundary homology'' groups as the cohomology of mapping cones of (chain level maps) comparing homology and cohomology, this algebra and composition structure is not completely obvious, and has received a fair bit of recent attention \cite{seidel6, COcone, CHOduality}.
As a rather direct analogy, for a manifold-with-boundary $M$ one can correspondingly recover the cup product on $H^*(\partial M)$ from the perspective of viewing $H^*(\partial M) = H^*(\cone(C^*(M, \partial M) \to C^*(M)))$, from (chain-level versions of) the cup product on cohomology of the interior $M$, along with (chain-level versions of) the structure of $H^*(M, \partial M)$ as a module over $H^*(M)$.

There is a natural chain-level $\ainf$ structure underlying the above composition maps, which assemble into what we call
the {\em Rabinowitz (wrapped) Fukaya category}
\[
    \rw(X);
\]
by construction it comes with a natural functor from the usual wrapped Fukaya category $\w(X)$. A version of this category built using the ``popsicle'' moduli spaces of \cite{abouzaidseidel} is implicit in \cite{venkatesh, seidel6}; we spell out a definition of this category in \S \ref{sec:rabinowitzwrappedFukaya}.  An independent construction of a (presumably equivalent) version of $\rw(X)$, using symplectic field theory (SFT) techniques, also appears in \cite{legout}.

Just as the cohomology of the boundary of a manifold-with-boundary is not completely determined by the cohomology of its bounding manifold, one does not in general expect $RFH^*(X)$ to be determined by $SH^*(X)$: although the (symplectic or ordinary) homology of the interior is determined by (symplectic) cohomology, the connecting map between them is an additional --- if somewhat degenerate --- piece of information. Somewhat surprisingly then, our main result, answering a conjecture of Abouzaid, says in fact that the category $\rw(X)$ {\em is} in fact (in good cases) equivalent to a category determined by $\w(X)$ by a canonical formula:
\begin{thm}[Theorems \ref{thmfunctor} and \ref{thm:equivalence} below] \label{thm:main}
    Let $(X, \lambda)$ be any Liouville manifold. 
    \begin{enumerate}[label=(\roman*)]
        \item (Theorem \ref{thmfunctor})  There is an $\ainf$ functor
    \begin{equation}\label{mainfunctor}
    \Phi: \rw(X) \to \widehat{\w(X)}_{\infty}
\end{equation}
where for an $\ainf$ category $\cc$, $\widehat{\cc}_{\infty}$ denotes its {\em (algebraizable)} formal punctured neighborhood of infinity of $\cc$ \cite{efimov} described below. \\

\item (Theorem \ref{thm:equivalence}) If $X$ is Weinstein (or more generally non-degenerate) then $\Phi$ is a quasi-equivalence.
    \end{enumerate}
\end{thm}
Recall that a Liouville manifold is {\em non-degenerate} in the sense of \cite{ganatra} if it admits a collection of Lagrangians satisfying Abouzaid's generation criterion \cite{abouzaid1}. Every Weinstein manifold is known to be non-degenerate by \cite{CDGG, GPSstructural, ganatra, gao}.
The {\em (algebraizable) formal punctured neighborhood of $\infty$} $\widehat{\cc}_{\infty}$ of a (dg or $\ainf$) category $\cc$, introduced by Efimov in \cite{efimov} and reviewed in \S \ref{sec:infinity}, is a purely algebraic construction of a new category from $\cc$ which whose non-triviality measures in the failure of $\cc$ to be {\em proper}. Recall that a category $\cc$ over $\K$ is proper if $\hom_{\cc}(A,B) \in \perf(\K)$ for all objects $A$ and $B$, where, denoting by $\mod(\K)$ the dg category of chain complexes of $\K$ modules, $\perf(\K) \subset \mod(\K)$, the subcategory of perfect $\K$-linear chain complexes, consists (if $\K$ is a field) of chain complexes which have finite total cohomology. The failure of $\hom_{\cc}$ to always land in $\perf(\K) \subset \mod(\K)$ can be measured by the induced map to the (dg or $\ainf$) quotient:
\begin{equation}
    \overline{\hom}_{\cc}: \cc^{op} \times \cc \longrightarrow \mod(\K) \longrightarrow \frac{\mod(\K)}{\perf(\K)},
\end{equation}
whose non-triviality in turn is equivalent to the non-triviality of the essential image of the associated Yoneda functor $\overline{Y}$:
\begin{equation}
    \begin{split}
        \cc &\stackrel{\overline{Y}}{\longrightarrow} \fun(\cc^{op}, \frac{\mod(\K)}{\perf(\K)})\\
        K &\mapsto \overline{\hom}_{\cc}(-, K);
    \end{split}
\end{equation}
where $\fun(-,-)$ denotes a suitable category of ($A_{\infty}$) functors. The category $\widehat{\cc}_{\infty}$ is precisely defined to be this essential image $\widehat{\cc}_\infty:= im(\overline{Y})$; see \S \ref{sec:infinity} for more details and properties. Note, as reviewed there, that there is a computable formulae for morphism spaces in $\widehat{\cc}_\infty$ in in terms of certain Hochschild cochain complexes in $\cc$.

A few words about the proof of Theorem \ref{thm:main}, which appears in \S \ref{section: functor}, are in order. 
The construction of the functor $\Phi$ involves a number of new moduli spaces of discs analogous to the construction of $\rw(X)$, and makes use of the explicit Hochschild cochain models of morphisms in $\widehat{\w(X)}_{\infty}$ alluded to above. To reconcile the fact that $\widehat{\w(X)}_{\infty}$ only depends on $\w(X)$ with the fact that
the continuation map $c: HW_*(K,L) \to HW^*(K,L)$ (whose chain level cone determines
morphisms in $\rw(X)$) appears to be extra information beyond that in
$\w(X)$, we observe that $c$ can be recovered from the {\em weak smooth
Calabi-Yau structure} on $\w(X)$. Weak smooth Calabi-Yau structures are a type of duality structure which exists on
wrapped Fukaya categories of non-degenerate Liouville manifolds \cite{ganatra},
which also come from extra geometric information beyond $\w(X)$. (This is a way in which the relationship between $\w(X)$ and $\rw(X)$ is {\em not} analogous to the relationship between $C^*(M)$ and $C^*(\partial M)$ for a manifold-with-boundary $M$: even if $M$ is orientable $C^*(M)$ does not possess sufficient duality to recover $C^*(\partial M)$ by itself).
These additional geometric operations, coming
from discs with two outputs, constitute an operation $\mathcal{CY}$ (called in
{\em loc. cit.} the ``non-compact Calabi-Yau morphism'') between $\w(X)$ and
its ``inverse dualizing bimodule'', which {\em loc. cit.} showed induces a
Poincar\'{e}-duality-type isomorphism in the non-degenerate case. Via a series
of holomorphic curve degenerations, our main work relates the functor $\Phi$ we construct
to $\mc{CY}$ and therefore deduces $\Phi$ is an isomorphism if $\mc{CY}$ is. As a warm-up, we note that the homology level continuation map $c$, which can be thought of as an operation associated to holomorphic curves one input in homology and one output in cohomology, can be thought of dually as coming from holomorphic curves with two outputs in cohomology (via the rule typical in Floer theory that, at least cohomologically, inputs and outputs can be exchanged at the cost of replacing cohomology with homology), which can be recovered from a limiting instance of the operations constituting $\mathcal{CY}$.

In the next few subsections, we detail initial applications that follow
more or less immediately from our main theorem and the existing literature: to situating
Rabinowitz Fukaya categories within mirror symmetry (and hence deducing new
computations of these categories), to computing these categories as quotients
of wrapped and compact Fukaya categories in certain instances, and to the
open-closed string relationship for Rabinowitz Floer homology (which is about 
the question of recovering the closed-string Rabinowitz Floer homology from
open-string Rabinowitz and/or wrapped Fukaya categories).

\subsection{Rabinowitz Fukaya categories and mirror symmetry}
Homological mirror symmetry (HMS) is well known to extend to (and has often been verified in) the setting of mirror pairs $(X,Y)$ where $X$ is a Weinstein manifold and $Y$ some scheme, in which case it predicts there should be an equivalence $\w(X) = \coh(Y)$ (here $\coh(Y)$ denotes the dg category of coherent complexes, and $=$ means ``Morita equivalence''). 
In Corollary \ref{hmsrabinowitz} below, we show that our Theorem \ref{thm:main} and work of Efimov, which gives in many instances a purely algbro-geometric interpretation of the $\widehat{\coh(Y)}_{\infty}$, allows us to situate the Rabinowitz Fukaya categories $\rw(X)$ into HMS in a variety of cases, implying new computations. As reviewed in more detail in \S \ref{sec:rabinowitzhms}, Efimov \cite{efimov} proves:
\begin{enumerate}
    \item for a non-compact smooth variety $Y$ over $\C$,  $\widehat{\coh(Y)}_{\infty}$ coincides with a certain so-called category of {\em (algebraizable) perfect complexes on the formal punctured neighborhood at $\infty$} $\perf_{alg}(\hat{Y}_{\infty})$, which can be defined by smoothly compactifying $Y$ to $\bar{Y}$, and quotienting the perfect complexes along the formal completion $\hat{\bar{Y}}_D$ of $\bar{Y}$ along the compactification divisor $D = \bar{Y} \backslash Y$, by the image of the canonical inclusion functor $\perf(D) \to \perf(\hat{\bar{Y}}_D)$ (categorically   ``puncturing'' along $D$). 

    \item If in turn $Y$ is separated and finite type, and proper (but not necessarily smooth) over any perfect field, Efimov shows $\widehat{\coh(Y)}_{\infty}$ coincides with (the opposite of) the category of singularities $\sing(Y)^{op}$.
\end{enumerate}
From this and Theorem \ref{thm:main} we deduce:

\begin{cor}\label{hmsrabinowitz}
    Suppose there is an HMS equivalence $\w(X) \cong \coh(Y)$ for a non-degenerate manifold $X$. 
    Then, there is also an equivalence 
    \begin{equation}
        \rw(X) \cong \widehat{\coh(Y)}_{\infty},
    \end{equation}
    If $Y$ is a smooth variety over $\C$, then this latter category can be computed as
    $\perf_{alg}(\hat{Y}_{\infty})$. If $Y$ is proper (separated, finite type)
    over a perfect field, then this latter category can be computed as
    $\sing(Y)^{op}$.
\end{cor}
Recently there have been many computations of wrapped Fukaya categories of Weinstein manifolds, in large part through exploiting recent functoriality and local-to-global properties established for these categories \cite{GPSsectorsoc, GPSstructural}). These computations have in particular lead to verifications of HMS for (on the symplectic side) Weinstein manifolds in a number of instances, see e.g., \cite{gammage-shende,lekili-polishchuk-pants,hacking-keating,lekili-ueda,lee,aaeko}. Corollary \ref{hmsrabinowitz}, along with the general computability of the algebro-geometric categories appearing in Efimov's result, therefore gives a method of extracting from these instances of HMS a number of new computations of the Rabinowitz Fukaya category $\rw(X)$ in instances when the mirror is smooth or proper. We remark that $\rw(X)$ on its own is significantly less functorial than $\w(X)$, in parallel with the categorical fact that Efimov's construction $\widehat{\cc}_{\infty}$ is not completely functorial in $\cc$, and with the algebro-geometric fact that only {\em proper maps} of varieties induce maps between their formal punctured neighborhoods of infinity. Hence we do not expect recent functoriality/localization methods for Fukaya categories to be directly applicable to computing $\rw(X)$ without first passing through $\w(X)$ as an intermediary using our Theorem \ref{thm:main}.

Seeing as the mirror of a Weinstein manifold $X$ is sometimes a scheme $Y$ which is neither smooth nor proper, and sometimes not even a scheme but instead a Landau-Ginzburg model $(Y,w)$ (in which case to obtain an HMS equivalence one should replace $\coh(Y)$ by a suitable matrix-factorization category $MF(Y,w)$), we ask:
\begin{ques}
    What is the geometric interpretation of $\widehat{\coh(Y)}_{\infty}$ if $Y$ is neither smooth nor proper? And what is the geometric interpretation of $\widehat{MF(Y,w)}_{\infty}$?
\end{ques}
An answer to this question would provide homological mirror counterparts to $\rw(X)$ in a variety of additional cases.

\subsection{Koszul duality and computations of the Rabinowitz Fukaya category}

Generalizing the fact that $\widehat{\coh(Y)}_{\infty} = \sing(Y)^{op} = (\coh(Y) / \perf(Y))^{op}$ for a {\em proper} $Y$ (a separated, finite type scheme over a perfect field), Efimov \cite{efimov} articulated criteria for a general $\cc$ under which the categorical formal punctured neighborhood $\widehat{\cc}_{\infty}$  can be computed as a quotient $\cc/ \cc_{\r{prop}}$, of $\cc$ by its subcategory of {\em (right) proper} (or pseudoperfect) objects.\footnote{Our convention is that $\cc$ is implicitly pre-triangulated and split-closed, so $\cc = \perf(\cc)$ where $\perf(-)$ denotes perfect modules. In this case, if $\cc$ is smooth, then $\cc_{\r{prop}}$ coincides with the category of {\em proper modules over $\cc$}, those modules that are cohomologically finite rank.}  For $\cc = \coh(Y)$ when $Y$ is proper (even if non-smooth), these criteria hold and result in the computed equality $\widehat{\coh(Y)}_{\infty} = \sing(Y)^{op}$.

Efimov's criteria, which amount to a certain Koszul duality holding between $\cc$ and a subcategory $\mathcal{D}$ of its (right) proper objects, is applicable in symplectic geometry, given that such duality phenomena have frequently been observed between the wrapped Fukaya category $\cc = \w(X)$ and the Fukaya category $\mathcal{D} = \mc{F}(X)$ of compact exact Lagrangians (mirror to the {\em properness} of a given $Y$); see e.g., \cite{etgulekili, ekholmlekili, likoszul}. Combining this discussion with Theorem \ref{thm:main} we immediately deduce: 
\begin{cor}\label{corkoszul}
    Suppose Koszul duality holds between the wrapped Fukaya category $\w(X)$ and the compact Fukaya category $\mc{F}(X)$. Then there is an equivalence $\rw(X) \cong \w(X)/\mc{F}(X)$.
\end{cor}
We will review the relevant notions in \S \ref{sec:koszul} and give some examples of how this can be used as a computational device. 

\subsection{The open-closed string relationship}
For completeness, we describe some aspects of the open-closed string relationship in the Rabinowitz context. Recall that the wrapped Fukaya category of a non-degenerate Liouville manifold and the usual symplectic cohomology $SH^*(X)$ are related by the {\em open-closed isomorphism} \cite{ganatra}
\begin{equation}\label{ociso}
    \mathcal{OC}: \r{HH}_{*-n}(\w(X)) \stackrel{\cong}{\to} SH^*(X),
\end{equation}
where $\r{HH}_*(\w)$ denotes the {\em Hochschild homology} of the wrapped Fukaya category (there is also a {\em closed-open} isomorphism with the Hochschild cohomology which we will not discuss here for brevity; see {\em loc. cit.}).

One can ask the corresponding question for Rabinowitz theory: how can one recover the Rabinowitz Floer homology $RFH^*(X)$ categorically? 
Note that it  does {\em not} seem to always be the case that there is an isomorphism between $\r{HH}_{*-n}(\rw(X))$ and $RFH^*(X)$ even when \eqref{ociso} holds, see Remark \ref{ocfailure}. 
\begin{rem}\label{ocfailure} 
    When $X=T^*S^1$ with its usual $\Z$-grading, $\w(X)$ is Morita equivalent to the ordinary algebra $HW^0(T_q^*, T_q^*) = k[t,t^{-1}]$ and the open-closed isomorphism gives an equivalence between $SH^*(X)$ and $\r{HH}_{*-1}(k[t,t^{-1}])$.
    Passing to Rabinowitz Floer homology, one can again compute (directly or using Theorem \ref{thm:main}) that $RW^0(T_q^*, T_q^*) =  k((t)) \oplus k((t^{-1}))$ with nothing in other degrees.  Let us argue that there cannot be an isomorphism between $RFH^*(T^*S^1)$ and $\r{HH}_{*-1}(k((t)) \oplus k((t^{-1})))$.  By direct computation $RFH^0(X) = RW^0(T_q^*, T_q^*) = k((t)) \oplus k((t^{-1}))$, which is free rank 1 over $RW^0(T_q^*, T_q^*)$. In contrast, $HH_{-1}( k((t)) )$ (the case of $k( ( t^{-1}))$ is analogous) equals the module of K\"{a}hler differentials $\Omega^1_{k( ( t)) / k}$ (\cite{loday}*{Prop 1.1.10}, noting that we use a cohomological grading convention), which is {\em not} finitely generated over $k ( ( t ) ) $ in characteristic zero as the size of a differential basis concides with the transcendence degree of $k ( ( t) ) $ over $k$ \cite{matsumura}*{Thm. 6.5}, which is infinite (see e.g., \cite{maclaneschilling}*{Lemma 1}, compare also \cite{shaul}*{Remark 5.2} again noting the grading convention difference). It follows that $HH_{-1}(RW^0(T_q^*, T_q^*))$ cannot be isomorphic to $RFH^0(X)$ as $RFH^0(X)$-modules.  
\end{rem}
Despite the failure of the analogue of \eqref{ociso}, work-in-preparation of Rezchikov \cite{rezchikov}, combined with \cite{ganatra} and an algebraic folk Lemma asserted in \cite{efimov} implies that $RFH^*(X)$, like $\rw(X)$, {\em can} be directly algebraically recovered from $\w(X)$ when \eqref{ociso} holds:
\begin{thm}\label{aspirationalOC}
When the open-closed isomorphism holds, there is an isomorphism $\r{HH}_{*-n}(\w(X), \rw(X)) \stackrel{\cong}{\to} RFH^*(X)$, where for a functor $f: A \to B$, $\r{HH}_*(A, B)$ denotes the Hochschild homology of $A$ with coefficients in $B$ (thought of via $f$ as a bimodule over $A$).
\end{thm}
In principle, the above Theorem could be proved directly geometrically by trying to directly construct and analyze
open-closed maps compatible with the long-exact sequence between $SH_*(X)$,
$SH^*(X)$ and $RFH^*(X)$. Instead, we appeal to Theorem \ref{thm:main} to deduce
$\r{HH}_{*-n}(\w(X), \rw(X)) \cong \r{HH}_{*-n}(\w(X),\widehat{\w(X)}_{\infty})$, which can be algebraically computed as a cone of a map from $\r{HH}_{*}(\w(X))^{\vee}$ to $\r{HH}_*(\w(X))$, putting us in a position to appeal to the open-closed isomorphism \cite{ganatra} and its compatibility with this connecting map (recently verified in \cite{rezchikov}).

By general properties of Hochschild homology, if $\dd$ is a quotient of $\cc$ then $\r{HH}_*(\cc, \dd) = \r{HH}_*(\dd)$. Hence we conclude that in some instances, the Hochschild homology of the Rabinowitz category computes Rabinowitz Floer homology.

\begin{cor} 
    When Koszul duality holds as in Corollary \ref{corkoszul}, then there is an isomorphism $\r{HH}_*(\rw(X)) \cong RFH^*(X)$.
\end{cor}

As noted at the start of the introduction, a natural and interesting question
would be to understand the compatibility of Theorem \eqref{aspirationalOC} with
algebraic structures, in a manner often done for \eqref{ociso} (both in the
Floer-theoretic \cite{abouzaid1, ganatra, rittersmith, sheridanfano, fooo, afooo, BEEproduct} and formal topological field theoretic \cite{costello, kontsevichsoibelman} contexts). For
instance, we ask
\begin{ques}
    Is there a product on $\r{HH}_{*-n}(\w(X), \rw(X))$ (or the
equivalent $\r{HH}_*(\w(X), \widehat{\w(X)}_{\infty})$), using the data of a
Calabi-Yau structure on $\w(X)$, which recovers the product on $RFH^*(X)$ under
the open-closed isomorphism? 
\end{ques}

\subsection*{Overview of paper}
In \S \ref{sec:infinity}, we recall Efimov's definition of the categorical formal punctured neighborhood of $\infty$ and some basic algebraic results about these categories. In \S \ref{sec:modulispaces} we review geometric preliminaries, and collate all moduli spaces of domains and maps used throughout the paper, and study their degenerations. Using these various moduli spaces, in \S \ref{sec:rabinowitzwrappedFukaya} we define the category $\rw(X)$ and in \S \ref{section: functor} we prove Theorem \ref{thm:main} (statement (i) appears as Theorem \ref{thmfunctor} and (ii) appears as Theorem \ref{thm:equivalence}). We recommend the reader read these latter two sections first, referring back to \S \ref{sec:modulispaces} as necessary for various moduli spaces. (We have written things in this fashion in order to streamline the discussions in \S \ref{sec:rabinowitzwrappedFukaya}-\ref{section: functor}).
Finally, in \S \ref{sec:applicationsexamples} we establish the applications described above, as well as some examples. To give some geometric intuition for the category $\rw(X)$, we also sketch a direct computation of $\rw(X)$ for $X = T^* S^1$, giving an example of Theorem \ref{thm:main} in practice.

\subsection*{Acknowledgements}
We are grateful to Mohammed Abouzaid for sharing with us his conjecture
relating the Rabinowitz and usual wrapped Fukaya categories, and to Paolo Ghiggini and Will Merry for helpful correspondence. S.G. was supported
by NSF grant DMS--1907635.

\section{Homological algebra of categorical formal punctured neighborhoods of infinity}\label{sec:infinity}
The goal of this section is to recall Efimov's definition of the algebraizable categorical formal punctured neighborhood $\widehat{\cc}_{\infty}$ of an $\ainf$ category $\cc$, give computable formula for its chain-level morphism complexes and compositions in terms of certain Hochschild cochain complexes in $\cc$, and describe some initial properties of $\widehat{\cc}_{\infty}$.  We essentially follow the original reference for this material \cite{efimov}, with the small caveat that ({\em loc. cit.}  assumes $\cc$ is a dg category whereas we allow $\cc$ to be  more generally $\ainf$; the constructions generalize straightforwardly with the same definition. 

\subsection{Categorical preliminaries}
First, we review some notation and definitions concerning chain complexes,
$\ainf$ categories and their functor, module, and bimodule categories, all of
which are relatively standard (compare \cite{seidel_book, seidel_ainf, ganatra, tradler, sheridanformulae}). 
We start
by fixing some base field $\K$ ($\K$ can also be a ring such as the integers
with some additional requirements which we indicate in parantheticals) and an abelian group $G$ with a non-zero map from $\Z$ (for our purposes one can take $G$ to be $\Z$ or $\Z_2$), over which all the chain complexes we study are implicitly graded.
We will denote by 
\begin{equation}\label{chaink}
\ch_\K = \mod(\K)
\end{equation}
the associated category of (cohomologically $G$-graded) dg $\K$-modules, i.e.,
cochain complexes of $\K$ modules (which we require to be cofibrant if $\K$ is
a ring). Given an element $x \in M \in \ch_\K$, we denote the degree of $x$ by
$|x| = \deg(x)$. We denote morphism spaces in this category by either
$\ch_\K(M,N)$ or $\hom_{\K}(M,N)$. Inside this category one has the subcategory
of {\em perfect
chain complexes}
\begin{equation}
\perf(\K) \subset \mod(\K),
\end{equation}
which can equivalently be defined as the compact objects in \eqref{chaink} or
as those objects of \eqref{chaink} which are (quasi-)isomorphic to a retract of a finite
complex of free finitely generated $\K$-modules (i.e., if $\K$ is a field, those chain complexes
with total finite-dimensional cohomology).

Let $\cc$ be an $\ainf$ category over $\K$, which concretely consists of a collection of objects (denoted $\ob \cc$ or simply
$\cc$), a morphism space in $\mod(\K)$ denoted $\hom_{\cc}(X,Y)$ or $\cc(X,Y)$ for every $X,Y \in \cc$, and operations 
\[
\mu^d: \cc(X_{d-1}, X_d) \otimes_{\K} \cdots \otimes_{\K}
\cc(X_0, X_1) \to \cc(X_0, X_d)
\]
of degree $2-d$ for every $d+1$ tuple $X_0, \ldots, X_d \in \cc$ and every $d \geq 1$ (where $\mu^1: \cc(X_0, X_1) \to \cc(X_0, X_1)$ is the differential on $\cc(X,Y) \in \mod(\K)$) satisfying the $\ainf$ equations (c.f., \cite{seidel_book}*{(1.2)}).  
For a tuple of objects $(X_0, \ldots, X_d)$ we use the
shorthand $\cc(X_0, \ldots, X_d):= \cc(X_{d-1}, X_d) \otimes_{\K} \cdots \otimes_{\K}
\cc(X_0, X_1)$, so $\mu^d: \cc(X_0, \ldots, X_d) \to \cc(X_0, X_d)[2-d]$.
The opposite category $\cc^{op}$ has the same collection of objects, while the morphism spaces are
$\cc^{op}(X_{0}, X_{1}) = \cc(X_{1}, X_{0})$, and the $\ainf$-structure maps are
\[
\mu^{d}_{\cc^{op}}(x_{d}, \cdots, x_{1})  = (-1)^{*_{d}} \mu^{d}_{\cc}(x_{1}, \cdots, x_{d}),
\]
where
\begin{equation}\label{koszulsign}
*_{d} = |x_{1}| + \cdots + |x_{d}| - d.
\end{equation}
For a pair of $\ainf$ categories $\cc$ and $\dd$ there is an associated $\ainf$
category of $\ainf$ functors 
\[
\fun(\cc,\dd),
\]
(called $nu-fun(\cc, \dd)$ in \cite{seidel_book}).
Objects are $\ainf$ functors $F: \cc \to \dd$, which consist of 
\begin{itemize}
    \item a map on objects $F: \ob \cc \to \ob \dd$ and 
    \item for every
tuple $(X_0, \ldots, X_d) \in \ob \cc^{d+1}$, a map on morphism complexes
\[
F^{d}: \cc(X_0, \ldots, X_d) \to \dd(FX_0, FX_d).
\]
\end{itemize}
satisfying the $\ainf$ functor equations
\begin{equation}
\begin{split}
& \sum_{s \ge 1} \sum_{i_{1} + \cdots + i_{s} = d} \mu_{\dd}^{s}(F^{i_{s}}(x_{d}, \cdots, x_{i_{1} + \cdots + i_{s-1} + 1}), \cdots, F^{i_{1}}(x_{i_{1}}, \cdots, x_{1}), ) \\
= & \sum_{i, j} (-1)^{*_{i}} F^{d - j + 1}(x_{d}, \cdots, x_{i + j + 1}, \mu_{\cc}^{j}(x_{i + j}, \cdots, x_{i+1}), x_{i}, \cdots, x_{1}).
\end{split}
\end{equation}
The morphism space in $\fun(\cc, \dd)$ between two functors $F$ and $G$, the cochain complex of $\ainf$ natural transformations, is
\[
    \fun(\cc,\dd)^g(F,G):= \prod_{X_0, \ldots, X_d \in \cc} \hom_{\K}(\cc(X_0, \ldots, X_d), \dd(FX_0, GX_d)[g-d]),
\]
with differential given by, for $T \in \fun^{g}(\cc,\dd)(F,G)$ a morphism of degree $g = |T|$:
\begin{equation}
    \begin{split}
    \mu^1(T)^{k}(x_k, \ldots, x_1):= &\sum_{r, i} \sum_{\substack{s_{1}, \ldots, s_{r} \ge 0\\s_{1}+\cdots+s_{r}=k}} (-1)^{\dagger} \mu_{\dd}^{r}(G^{s_{r}}(x_{k}, \ldots, x_{s_{1}+\cdots+s_{r-1}+1}), \ldots, \\
    &G^{s_{i+1}}(x_{s_{1}+\cdots+s_{i+1}}, \ldots, x_{s_{1}+\cdots+s_{i}+1}), T^{s_{i}}(x_{s_{1}+\cdots+s_{i}}, \ldots, x_{s_{1}+\cdots+s_{i-1}+1}),\\
    & F^{s_{i-1}}(x_{s_{1}+\cdots+s_{i-1}}, \ldots, x_{s_{1}+\cdots+s_{i-2}+1}) , \ldots, F^{s_{1}}(x_{s_{1}}, \ldots, x_{1})) \\
    &-\sum_{m,n} (-1)^{*_{n}+|T|-1}T^{k-m+1}(x_{k},\ldots, x_{n+m+1}, \mu_{\cc}^{m}(x_{n+m},\ldots,x_{n+1}) x_{n}, \ldots, x_{1}) ,
    \end{split}
\end{equation}
where the signs are
\[
\dagger = (|T|-1)(|x_{1}|+\cdots+|x_{s_{1}+\cdots+s_{i-1}}|-s_{1}-\cdots-s_{i-1})
\]
and $*_{n}$ is defined as in \eqref{koszulsign}.

The higher compositions
\[
    \mu^d: \fun(\cc,\dd)(F_0, \ldots, F_d) \to \fun(\cc,\dd)(F_0, F_d)
    \]
    are given by
\begin{equation}
    \begin{split}
     \mu^d(T_d, \ldots, T_1)^{k}(x_k, \ldots, x_1) =&\sum_{r, i_{1},\ldots,i_{d}} \sum_{\substack{s_{1},\ldots,s_{r} \ge 0\\s_{1}+\cdots+s_{r}=k}} (-1)^{\circ} \mu^{r}_{\dd}( F_d^{s_{r}}(x_{k}, \ldots, x_{k-s_{r}+1}), \ldots, F_{d}^{s_{i_{d}+1}}(\ldots),\\
       &T_{d}^{s_{i_{d}}}(\ldots), \ldots, T_{2}^{s_{i_{2}}}(\ldots), F_1^{s_{i_{2}-1}}(\ldots), \ldots, F_{1}^{s_{i_{1}+1}}(\ldots), \\
    &T_{1}^{s_{i_{1}}}(x_{s_{1}+\cdots+s_{i_{1}}}, \ldots, x_{s_{1}+\cdots+s_{i_{1}-1}+1}), F_0^{s_{i_{1}-1}}(\ldots), \ldots, F_0^{s_{1}}(x_{s_{1}},\ldots,x_{1}))
\end{split}
\end{equation}
with
\begin{equation}\label{signforfunctorcomp}
\circ = \sum_{p=1}^{d}(|T_{p}|-1)(\sum_{j=1}^{s_{1}+\cdots+s_{i_{p}-1}} (|x_{j}|-1))
\end{equation}
and $1 \le i_{1} \le i_{2} \le \cdots \le i_{d} \le r$.
In the case that $F=G=\id$, the composition $\mu^2$ defined above recovers the usual formulaic definition of the Yoneda product of Hochschild cohomology cochains.

In the special case that $\dd$ is a dg category (in the sense that $\mu^{\geq
3}_{\dd} = 0$), then one can directly see that $\fun(\cc, \dd)$ is a dg category and the above formulae
greatly simplify to:
\begin{align}
    \label{diffnattransDG}\mu^1(T)^{k}(x_k, \ldots, x_1) &= \mu^1_{\dd}(T^{k}(x_k, \ldots, x_1)) -\sum_{m,n} (-1)^{*_{n}+|T|-1}T^{k-m+1}(x_{k},\ldots, \mu_{\cc}^{m}(x_{n+m},\ldots,x_{n+1}), \ldots, x_{1})\\
    \label{productnattransDG}\mu^2(S,T)^{k}(x_k, \ldots, x_1) &= \sum_{1 < \ell < k}\mu^2_{\dd}(S^{k-\ell}(x_k, \ldots, x_{\ell+1}), T^{\ell}(x_{\ell}, \ldots, x_1)).
\end{align}

The category of left, resp. right modules is by definition the category of
functors $\fun(\cc, \ch_\K)$, resp. $\fun(\cc^{op}, \ch_\K)$,
and a $\cc-\dd$ bimodule is an $\ainf$-bilinear functor $\cc \times \dd^{op} \to \ch_{\K}$ (in the sense of \cite{lyubashenkomulti}).
The first examples of left/right modules are the left/right Yoneda modules, sometimes
simply called the representable left/right modules
\[
    \begin{split}
        Y^{l}_K&:= \cc(K, -)\\
        Y^r_K&:=\cc(-, K)
    \end{split}
\]
Morphism spaces between and tensor products of Yoneda modules provide some interesting and important examples of bimodules,
for which we would like to provide formulae below.

One starts by recalling that the diagonal bimodule $\cc_{\D}$ is the bimodule with underlying cochain spaces
\[
\cc_{\D}(K, L) = \cc(L, K),
\]
such that the bimodule structure maps $\mu_{\cc_{\D}}^{k, l}$ are given by
\begin{equation}
\mu_{\cc_{\D}}^{k, l}(x_{k}, \ldots, x_{1}, c, x'_{1}, \ldots, x'_{l}) = (-1)^{\sum_{j=1}^{l} |x'_{j}| - l - 1} \mu^{k+1+l}_{\cc}(x_{k}, \ldots, x_{1}, c, x'_{1}, \ldots, x'_{l}).
\end{equation}

\begin{defn} \label{def: c dual}
    Define the linear dual diagonal bimodule 
\begin{equation}
(\cc)^{\vee}
\end{equation}
to be the $\ainf$-bimodule over $\cc$ such that
\begin{enumerate}[label=(\roman*)]

\item the underlying cochain complex is
\begin{equation}
(\cc)^{\vee}(K, L) = \hom_{\K}(\cc^{-*}(K, L), \K),
\end{equation}
with the opposite grading.

\item The bimodule structure maps
\begin{equation}
\begin{split}
	\mu^{k, l}_{(\cc)^{\vee}}: & \cc(K_{k-1}, K_{k}) \otimes \cdots \otimes \cc(K_{0}, K_{1})  \otimes (\cc)^{\vee}(K_{0}, L_{0}) \\
	& \otimes \cc(L_{1}, L_{0}) \otimes \cdots \otimes \cc(L_{l}, L_{l-1}) \to (\cc)^{\vee}(K_{k}, L_{l})
\end{split}
\end{equation}
are induced from $A_{\infty}$-structure maps of the $\ainf$-category $\cc$
\begin{equation}
\begin{split}
	&\mu^{k, l}_{(\cc)^{\vee}}(x_{k}, \ldots, x_{1}, f, x'_{1}, \ldots, x'_{l})(z) \\
= &(-1)^{\sum_{j=1}^{l} \deg(x'_{j}) - l} f(\mu^{k+l+1}_{\cc}(x_{k}, \ldots, x_{1}, z, x'_{1}, \ldots, x'_{l})).
\end{split}
\end{equation}

\end{enumerate}
\end{defn}

\begin{defn}\label{hom bimodule}
Let $\cc$ be an $\ainf$-category, and $(K, L) \in \ob \cc \times \ob \cc$ a pair of objects.
Define an $\ainf$ $\cc^{op} - \cc^{op}$-bimodule, or simply a bimodule over $\cc^{op}$,
\begin{equation}
	\hom_{\K}(Y_{K}^{r}, Y_{L}^{r})
\end{equation}
named as the space of $\K$-linear homomorphisms from the right Yoneda module of $K$ over $\cc$ to that of $L$ as follows.
\begin{itemize}

\item For each pair of objects $(K', L') \in \ob \cc^{op} \times \ob \cc^{op} = \ob \cc \times \ob \cc$,
\begin{equation}
	\hom_{\K}(Y_{K}^{r}, Y_{L}^{r}) = \hom_{\K}(Y_{K}^{r}(K'), Y_{L}^{r}(L')) = \hom_{\K}(\cc(K', K), \cc(L', L)).
\end{equation}

\item The bimodule structure maps, 
\begin{equation}
\begin{split}
	\mu^{k, l}_{\hom_{\K}(Y_{K}^{r}, Y_{L}^{r})}: & \cc^{op}(K'_{0}, \ldots, K'_{k}) \otimes \hom_{\K}(\cc(K'_{0}, K), \cc(L'_{0}, L)) \\
	& \otimes \cc(L'_{0}, \ldots, L'_{l}) \to \hom_{\K}(\cc(K'_{k}, K), \cc(L'_{l}, L))
\end{split}
\end{equation}
expressed in $\hom_{\cc}$ in view that $\hom_{\cc^{op}}(K, L) = \hom_{\cc}(L, K)$, are defined as
\begin{equation}
\begin{split}
	\mu^{k, l}_{\hom_{\K}(Y_{K}^{r}, Y_{L}^{r})}: & \cc(K'_{1}, K'_{0}) \otimes \cdots \otimes \cc(K'_{k}, K'_{k-1}) \otimes \hom_{\K}(\cc(K'_{0}, K), \cc(L'_{0}, L))) \\
	& \otimes \cc(L'_{l-1}, L'_{l}) \otimes \cdots \otimes \cc(L'_{0}, L'_{1}) \to \hom_{\K}(\cc(K'_{k}, K), \cc(L'_{l}, L))
\end{split}
\end{equation}
are determined by $\ainf$-module maps for $Y_{K}^{r}$ and $Y_{L}^{r}$ (the right Yoneda modules over $\cc$ are regarded as left $\ainf$-modules over $\cc^{op}$),
\begin{equation}
\begin{split}
	& \mu^{k, l}_{\hom_{\K}(Y_{K}^{r}, Y_{L}^{r})}(x_{1}, \ldots, x_{k}, \phi, y_{l}, \ldots, y_{1})(z)\\
= & \begin{cases}
0, & k>0, l>0\\
(-1)^{|\phi| + |y_{1}| + \cdots + |y_{l}| + |z| - l - 1} \phi( \mu_{Y_{K}^{r}}^{l}(z, y_{l}, \ldots, y_{1})) & k = 0, l > 0\\
(-1)^{|x_{1}| + \cdots + |x_{k}| + |z| - k - 1} \mu_{Y_{L}^{r}}^{k} (x_{1}, \ldots, x_{k}, \phi(z)) & k > 0, l = 0\\
(-1)^{|\phi| + |z| -1} \mu^{0}_{Y_{L}^{r}}(\phi(z)) + \phi(\mu^{0}_{Y_{K}^{r}}(z)), & k = l = 0
\end{cases}
\end{split}
\end{equation}
for elements $x_{i} \in \cc^{op}(K'_{i-1}, K'_{i}) = \cc(K'_{i}, K'_{i-1})$, 
$y_{j} \in \cc^{op}(L'_{j}, L'_{j-1}) = \cc(L'_{j-1}, L'_{j})$,
 $\phi \in \hom_{\K}(\cc(K'_{0}, K), \cc(L'_{0}, L)$ and $z \in \cc(K'_{k}, K)$

\end{itemize}
\end{defn}

\begin{defn}\label{tensor product bimodule}
Let $\cc$ be an $\ainf$-category, and $(K, L) \in \ob \cc \times \ob \cc$ a pair of objects.
Define another bimodule over $\cc^{op}$
\begin{equation}
	(Y_{K}^{r})^{\vee} \otimes_{\K} Y_{L}^{r}
\end{equation}
named as the space of finite-rank $\K$-linear homomorphisms from the right Yoneda modules of $K$ to that of $L$ as follows.
\begin{itemize}

\item For each pair of objects $(K', L') \in \ob \cc^{op} \times \ob \cc^{op} = \ob \cc \times \ob \cc$,
\begin{equation}
	((Y_{K}^{r})^{\vee} \otimes_{\K} Y_{L}^{r})(K', L') = (Y_{K}^{r})^{\vee}(K') \otimes_{\K} Y_{L}^{r}(L') = (\cc(K', K))^{\vee} \otimes_{\K} \cc(L', L).
\end{equation}

\item The bimodule structure maps
\begin{equation}
\begin{split}
\mu^{k, l}_{(Y_{K}^{r})^{\vee} \otimes_{\K} Y_{L}^{r}}: \mu^{k, l}_{(Y_{K}^{r})^{\vee} \otimes_{\K} Y_{L}^{r}}: & \cc^{op}(K'_{0}, \ldots, K'_{k}) \otimes (\cc(K'_{0}, K))^{\vee} \otimes_{\K} \cc(L'_{0}, L) \\
	& \otimes \cc(L'_{0}, \ldots, L'_{l})  \to (\cc(K'_{k}, K))^{\vee} \otimes_{\K} \cc(L'_{l}, L),
\end{split}
\end{equation}
are defined by
\begin{equation}
\begin{split}
	& \mu^{k, l}_{(Y_{K}^{r})^{\vee} \otimes_{\K} Y_{L}^{r}}(x_{1}, \ldots, x_{k}, f \otimes w, y_{l}, \ldots, y_{1})(z) \\
= & \begin{cases}
0, & k>0, l>0\\
(-1)^{|f| + |y_{1}| + \cdots + |y_{l}| + |z| - l - 1} f( \mu_{Y_{K}^{r}}^{l}(z, y_{l}, \ldots, y_{1})) \otimes w & k = 0, l > 0\\
(-1)^{|x_{1}| + \cdots + |x_{k}| + |z| - k - 1} \mu_{Y_{L}^{r}}^{k}(x_{1}, \ldots, x_{k}, f(z) w) & k > 0, l = 0\\
(-1)^{|f| + |z| - 1} \mu^{0}_{Y_{L}^{r}}(f(z)) + \phi(\mu^{0}_{Y_{K}^{r}(z)}), & k = l = 0.
\end{cases}
\end{split}
\end{equation}
\end{itemize}
\end{defn}

More generally, for any two right modules $\mathcal{M}, \mathcal{N}$,
we can define in a similar way to Definition \ref{hom bimodule} a bimodule
\begin{equation}
\hom_{\K}(\mathcal{M}, \mathcal{N}).
\end{equation}
While a general version of Definition \ref{tensor product bimodule} is the tensor product (over $\K$) of a left module with a right module, 
which is more standard.

\begin{rem}
	A $\cc^{op} - \cc^{op}$-bimodule naturally gives rise to a $\cc - \cc$-bimodule,
	except that covariance and contravariance are interchanged. 
More generally, a $\cc - \dd$ bimodule $\mathcal{B}$ is also a $\dd^{op} - \cc^{op}$ bimodule with the module structure maps reversed:
\[
	\mu^{k, l}_{\mathcal{B}, \cc-\rmod \dd}(x_{1}, \cdots, x_{k}, b, y_{l}, \cdots, y_{1}) = (-1)^{*} \mu^{l, k}_{\mathcal{B}, \dd^{op}-\rmod \cc^{op}}(y_{1}, \cdots, y_{l}, b, x_{k}, \cdots, x_{1})
\]
with a sign twist.
\end{rem}

\begin{rem}
	One can define an $\ainf$-quadmodule $\hom_{\K}(\cc, \cc)$, as a special case of $\ainf$-$n$-modules, following \cite{mau}.
For us, the most important structure is that the $\ainf$-quadmodule $\hom_{\K}(\cc, \cc)$ can be thought of as a 'bimodule-valued bimodule':
\[
	\hom_{\K}(\cc^{op}_{\D}, \cc^{op}_{\D})(K, -, L, -) = \hom_{\K}(Y_{K}^{r}, Y_{L}^{r}),
\]
similarly for $(\cc^{op})^{\vee} \otimes_{\K} \cc^{op}_{\D}$.
The Hoschschild cohomology cochain spaces
\[
	\r{CC}^{*}(\cc^{op}, \hom_{\K}(\cc^{op}_{\D}, \cc^{op}_{\D}))
\]
and
\[
	\r{CC}^{*}(\cc^{op}, (\cc^{op})^{\vee} \otimes_{\K} \cc^{op}_{\D})
\]
carry natural structures of $\ainf$-bimodules over $\cc^{op}$, induced from the $\ainf$-quadmodule structures of the above-mentioned quadmodules.
\end{rem}

Given a $\cc-\cc$ bimodule $\mathcal{P}$ (or simply called a bimodule over $\cc$), the {\it Hochschild cohomology cochain complex} is defined to be
\begin{equation}
    \r{CC}^k(\cc, \mathcal{P}) = \prod_{X_{0}, \ldots, X_{d} \in \cc} \hom_{\K}(\cc(X_{0}, \ldots, X_{d}), \mathcal{P}(X_{0}, X_{d})[k-d]),
\end{equation}
with the differential 
\begin{equation}
    \begin{split}
    \d(F)^{k}(x_{k}, \ldots, x_{1}) &= \sum (-1)^{\dagger} \mu^{r, s}_{\mathcal{P}}(x_{k}, \ldots, F^{k-r-s}(\ldots), x_{s}, \ldots, x_{1}) \\
    & - \sum (-1)^{*_{i}}  F^{k-j+1}(x_{k}, \ldots, \mu^{j}_{\cc}(\ldots), x_{i}, \ldots, x_{1}).
\end{split}
\end{equation}
where
\begin{equation}
    \dagger = |F| \cdot *_s
\end{equation}
Observe that the morphism spaces of the functor category $\fun(\cc, \dd)$ coincide with the (bar model of the) Hochschild cohomology co-chain complex of $\cc$ with values in the pulled-back bimodule $\dd(F(-), G(-)):= (F,G)^*(\dd_{\Delta})$, 
\begin{equation}\label{nattranshochschild}
        \fun(\cc, \dd)(F,G) = \r{CC}^*(\cc, \dd(F(-), G(-))).
    \end{equation}
    See for example \cite{seidel_book}*{\S (1f)} for a discussion.

A module $\mathcal{M}$ or bimodule $\mathcal{B}$ is {\em proper} if, when
thought of as a functor (or bilinear functor) to chain complexes, it lands in
{\em perfect chain complexes}; i.e., if $\mathcal{M}(X)$ (resp.
$\mathcal{B}(X,Y)$) is a perfect chain complex for all $X$ (resp. $X,Y$)
A module is {\em perfect} if it is isomorphic to a summand of a finite complex
of representable modules.  A category $\cc$ is {\em smooth} if its diagonal
bimodule is a perfect object in the category of bimodules, i.e., if it is a
retract of a finite complex of Yoneda bimodules. A category is {\em proper} if
its diagonal bimodule is proper or equivalently if $\cc(K,L) \in \perf(\K)$ for
all $K,L$.

Lastly, let us also mention a useful result regarding tensor product of bimodules.
Recall that the tensor product of a $(\cc, \dd)$-bimodule $\mathcal{P}$ with a $(\dd, \mathcal{E})$-bimodule $\mathcal{Q}$ over $\dd$ is defined to be the $(\cc, \mathcal{E})$-bimodule $\mathcal{P} \otimes_{\dd} \mathcal{Q}$ with underlying cochain space
\begin{equation}
\mathcal{P} \otimes_{\dd} \mathcal{Q}(X, Y) = \bigoplus_{Z_{0}, \ldots, Z_{s}} \mathcal{P}(X, Z_{0}) \otimes \dd(Z_{0}, \ldots, Z_{s}) \otimes \mathcal{Q}(Z_{s}, Y),
\end{equation}
the differential $\mu^{0, 0}_{\mathcal{P} \otimes_{\dd} \mathcal{Q}}$ given by the sum of operations $\mu^{0, l}_{\mathcal{P}} \otimes \id_{\dd}^{\otimes s-l} \otimes \id_{\mathcal{Q}}, \id_{\mathcal{P}} \otimes \id_{\dd}^{\otimes s-k} \otimes \mu^{k, 0}_{\mathcal{Q}}$ and $\id_{\mathcal{P}} \otimes \id_{\dd}^{\otimes s-i-j} \otimes \mu_{\dd}^{j} \otimes \id_{\dd}^{\otimes i}$,
and the higher structure maps $\mu^{r, s}_{\mathcal{P} \otimes_{\dd} \mathcal{Q}}$ being, up to signs
\[
\mu^{r, 0}_{\mathcal{P} \otimes_{\dd} \mathcal{Q}} = \sum_{l} \mu^{r, l}_{\mathcal{P}} \otimes \id_{\dd}^{\otimes s - l} \otimes \id_{\mathcal{Q}},
\]
\[
\mu^{0, s}_{\mathcal{P} \otimes_{\dd} \mathcal{Q}} = \sum_{k} \id_{\mathcal{P}} \otimes \id_{\dd}^{\otimes s-k} \otimes \mu^{k, s}_{\mathcal{Q}},
\]
and $\mu^{r, s}_{\mathcal{P} \otimes_{\dd} \mathcal{Q}} = 0$ if both $r, s > 0$.
(see for example \cite{ganatra} for more details about the definitions of tensor products with signs).
Taking the tensor product over $\cc$ with $\cc_{\D}$ either from the left or from the right induces a quasi-isomorphism of modules and bimodules, 
where for bimodules the statement is the following:

\begin{lem}[Proposition 2.2, \cite{ganatra}]\label{lemcollapse}
For any bimodule $\mathcal{B}$ over $\cc$, 
the {\it collapse map}
\begin{equation}\label{collapse map}
\mu_{\D}: \cc_{\D} \otimes_{\cc} \mathcal{B} \stackrel{\sim}\to \mathcal{B},
\end{equation}
given by the following data:
\begin{equation}
\cc(X_{0}, \ldots, X_{k}) \otimes \cc_{\D}(X_{0}, Z_{0}) \otimes \cc(Z_{0}, \ldots, Z_{s}) \otimes \mathcal{B}(Z_{s}, Y_{0}) \otimes \cc^{op}(Y_{0}, \ldots, Y_{l}) \to \mathcal{B}(X_{k}, Y_{l})
\end{equation}
\begin{equation}\label{collapseformula}
\mu_{\D}^{s; k, l}(x_{k}, \ldots, x_{1}, c, z_{1}, \ldots, z_{s}, b, y_{1}, \ldots, y_{l}) = (-1)^{\circ^{s}_{-l}} \mu_{\mathcal{B}}^{k+s+1, l}(x_{k}, \ldots, x_{1}, c, z_{1}, \ldots, z_{s}, b, y_{1}, \ldots, y_{l})
\end{equation}
is a quasi-isomorphism.
Here the sign is
\begin{equation}
\circ^{s}_{-l} = \sum_{i=1}^{s} |z_{i}| - s + |b| - 1 + \sum_{j=1}^{l} |y_{j}| - l. 
\end{equation}
\end{lem}

\subsection{Calkin complexes}

The dg category of {\em Calkin complexes} over $\K$ is by definition the dg
quotient of (co)chain complexes by perfect complexes: 
\begin{equation}\label{calkK}
    \calk_\K:= \mod(\K) / \perf(\K) \cong \mod(\K) / \K.
\end{equation}
(note that the quotient by any subcategory ---  the one object category consisting of $\K$ in this instance --- agrees with the quotient by its split-closed pre-triangulated closure --- $\perf(\K)$ in this instance).

By definition of the dg quotient \cite{drinfeld}, morphism spaces in \eqref{calkK} can be calculated by:
\begin{align}
    \calk_\K(M,N) &:= \cone( \hom_{\K}(M, \K) \otimes_{\K}^{\mathbb L} \hom_{\K}(\K, N) \to \hom_{\K}(M,N)) \\
    & \cong \cone( M^* \otimes_{\K} N \stackrel{ev}{\to} \hom_{\K}(M,N))\\
    \label{matrixdecomp}& :=  (M^* \otimes_{\K} N)[1] \oplus \hom_{\K}(M,N), d = \left(\begin{array}{cc} d & ev \\ 0 & d \end{array} \right)
\end{align}
The second equality can be realized by either observing that the derived
tensor product is  --- in this case --- the ordinary tensor product. Or more
formulaically, if one uses the  reduced bar resolution for the derived
tensor product over $\K$, the higher length terms are just empty (as the kernel
of the augmentation $\K \to \K$ is zero).

With respect to the direct sum decomposition \eqref{matrixdecomp}, we refer to
the left factor of the direct sum as the ``$-$'' factor and the right factor as
the ``$+$'' factor. Note that the $-$ factor can be identified with the
sub-dg $\K$-module of $\hom_{\K}(M,N)$ consisting of finite-rank homomorphisms
from $M$ to $N$.  The
composition of elements in $\calk_k$ is given by, for $(\phi \otimes n, \psi) \in
\calk_k(M,N)$ and $(\phi' \otimes p, \psi') \in \calk_k(N,P)$ by
\begin{equation}\label{calkincomp}
    (\phi \otimes \psi'(n) + \psi^*(\phi')  \otimes p, \psi' \circ \psi).
\end{equation}
In terms of the decomposition \eqref{matrixdecomp}, there are three interesting
components of composition: The $++ \to +$ component is simply the
composition of homomorphisms, and the $-+ \to -$ and $+-\to -$ components are
the (right resp. left) module action of homomorphisms over finitely supported
homomorphisms. The remaining components are zero.  

One way to deduce the above
formula for composition is to note that, while in the usual bar complex
formula for localization (c.f., \cite{lyubashenkoquotient}) there would be a $\mu^2$ of length two tensors,
and hence a $-- \to -$ component of the product, this $\mu^2$ produces a length 3 tensor, and we are in a quotient complex where the length 3
tensors are zero. 

There is a natural quotient functor $j: \mod(\K) \to \calk_\K$. Essentially by definition, a chain
complex $C^*$ is perfect if and only if $j(C^*) = 0$,
where $=0$ means isomorphic to zero in the cohomology category in this subsection and next. 

\subsection{The formal punctured neighborhood of infinity}

Following Efimov \cite{efimov}, we shall define the categorical formal punctured neighborhood at infinity of an arbitrary small 
$\ainf$ category $\cc$.  The (algebraizable) formal punctured neighborhood of infinity
of an $\ainf$ category $\cc$ can be thought of as the  obstruction to $\cc$ being
proper. 
Recall that any $\ainf$ category $\cc$ has a tautological bilinear functor given by the hom pairing
\[
    \hom: \cc^{op} \times \cc \to \mod(\K),
\]
and the non-triviality of the composition of this functor with the quotient map $\mod(\K) \to \calk_\K$,
\[
    \overline{\hom}: \cc^{op} \times \cc \to \mod(\K) \to \calk_{\K}
\]
measures the failure of $\hom$ to take values in perfect complexes. In turn,
the non-triviality of the latter functor can be detected by looking at the
essential image of the induced partial adjoint map
\begin{equation} \label{calkinyoneda}
\begin{split}
y: \cc &\to \fun(\cc^{op}, \calk_{\K}) \\
K &\mapsto \overline{\hom}(-,K)
\end{split}
\end{equation}
which is the composition of the standard yoneda functor $y: \cc \to \fun(\cc^{op}, \ch_{\K})$ with the natural projection $\ch_{\K} \to \calk_{\K}$,
also denoted by $y$ by abuse of notation.

\begin{defn}[\cite{efimov}]
    The category of {\em algebrizeable perfect complexes on the formal punctured neighborhood at $\infty$} of $\cc$, denoted
    $\cinf$ is by definition the (Karoubi completed
    pre-triangulated closure of the) essential image of $y$ in \eqref{calkinyoneda} above.
\end{defn}

There is, by construction a functor $y: \cc \to \cinf$. Immediately by definition, we see that:

\begin{lem}\label{rightproperzero}
 If $K \in \cc$ is any right proper object, meaning $\hom(-,K) \in \perf(\K)$ for all $-$, then $y(K) = 0$. \qed
 \end{lem}
 
\begin{lem}
    $\cinf = 0$ if and only if $\cc$ is proper. 
\end{lem}
\begin{proof}
    If $\cc$ is proper, then every object is right proper so Lemma \ref{rightproperzero} implies $y(\cc) = 0$.  
   Conversely, if $\cinf = 0$, then every Yoneda functor $\overline{\hom}(-,X): \cc \to \mod (\K) \to \calk_{\K}$ is the zero functor, which means that $\overline{\hom}(X,Y) = j(\hom(X,Y)) =0 \in \calk_{\K}$ for all $Y$, i.e., $\hom(X,Y) \in \perf(\K)$ for all $X$ and $Y$.
\end{proof}

The general fact that morphism spaces in the functor category are Hochschild cohomology cochain complexes \eqref{nattranshochschild} suggests that morphism spaces in $\cinf$ are also Hochschild cohomology cochain complexes of some sort.
It turns out that the coefficients of these Hochschild cohomology cochain complexes have appeared in Definitions \ref{hom bimodule} and \ref{tensor product bimodule}.

For any $\ainf$ category $\cc$, the category $\cinf$ is a dg category, i.e. an $\ainf$ category with $\mu^{\geq 3}_{\cinf} = 0$.
The objects are the same as those of $\cc$. Let us compute the morphism spaces explicitly. 
For $K, L \in  \ob \cc$ and their images in $\cinf$ as $\bar{y}(K), \bar{y}(L)$, we have
\begin{align}
    \label{cinfmorphism}   &\cinf(y(K), y(L)) := \fun(\cc^{op}, \calk_{\K})(Y^r_K, Y^r_L) \\
    &= \prod_{X_0, \ldots, X_k \in \cc} \hom_{\K}(\cc^{op}(X_0, \ldots, X_k), \calk_{\K}(\cc(X_0, K), \cc(X_k, L))) \\
    &=  \hom_{\K}( \prod_{X_0, \ldots, X_k \in \ob \cc^{op}}\cc^{op}(X_0, \ldots, X_k), \cone(\cc(X_0, K)^* \otimes_{\K} \cc(X_k, L) 
        \stackrel{ev}{\to} \hom_{\K}(\cc(X_0, K), \cc(X_k, L)))\\
        \label{cinfmorphismspelledout} &= \cone(\r{CC}^*(\cc^{op}, (Y^r_K)^* \otimes_{\K} Y^r_L) \stackrel{ev \circ}{\to}  \r{CC}^*(\cc^{op}, \hom_{\K}(Y^r_K, Y^r_L)))\\
        &= \r{CC}^*(\cc^{op}, (Y^r_K)^* \otimes_{\K} Y^r_L)[1] \oplus  \r{CC}^*(\cc^{op}, \hom_{\K}(Y^r_K, Y^r_L)), d = \left(\begin{array}{cc} \delta_- & ev \circ
        \\\ 0 & \delta_+ \end{array} \right)
 )
 \end{align}
 where $\delta_-$ and $\delta_+$ are the usual Hochschild cohomology differentials. 

Now let us compute the differential and the composition.
  Following \eqref{diffnattransDG}, the differential acts on an element $\phi
 \in \eqref{cinfmorphism}$, with respect to the decomposition $\phi = \phi_-
 \oplus \phi_+$ of \eqref{cinfmorphismspelledout}, by $\delta_-(\phi_-) \oplus
 (\r{ev} \circ \phi_- + \delta_+(\phi_+))$.
Let us write down the differential and product concretely. 
Let $\mu_{\hom_{\K}(Y^r_K, Y^r_L)}^{k, l}$ denote the $\ainf$ $\w^{op}$-bimodule structure maps on $\hom_{\K}(Y^r_K, Y^r_L)$, 
 If $\phi_{+} = \{\phi_{+}^{d}\}_{d = 0}^{\infty}$, and $\phi_{-} = \{\phi_{-}^{d}\}_{d = 0}^{\infty}$, then
\begin{equation}\label{formula for Hochschild differential +}
\begin{split}
& (\delta_{+} (\phi_{+}))^{d}(x_{1}, \cdots, x_{d}) \\
= & \sum (-1)^{|\phi_{+}| + |x_{1}| + \cdots + |x_{i}| - i} \phi_{+}^{d - j + 1}(x_{1}, \ldots, \mu_{\cc}^{j}(x_{i+1}, \ldots, x_{i+j}), \ldots, x_{d}) \\
& + \sum (-1)^{(|\phi_{+}| -1)(|x_{1}| + \cdots + |x_{i}| - i)} \mu_{\hom_{\K}(Y^r_K, Y^r_L)}^{k - i - j, i}(x_{1}, \ldots, \phi_{+}^{j}(x_{i+1}, \ldots, x_{i+j}), \ldots, x_{d}) \\
= & \sum (-1)^{|\phi_{+}| + |x_{1}| + \cdots + |x_{i}| - i} \phi_{+}^{d - j + 1}(x_{1}, \ldots, \mu_{\cc}^{j}(x_{i+1}, \ldots, x_{i+j}), \ldots, x_{d}) \\
& + \sum \mu_{\hom_{\K}(Y^r_K, Y^r_L)}^{d - j, 0}(x_{1}, \ldots, x_{d-j}, \phi_{+}^{j}(x_{d-j+1}, \ldots, x_{d})) \\
& + \sum \mu_{\hom_{\K}(Y^r_K, Y^r_L)}^{0, d - j}(\phi_{+}^{j}(x_{1}, \ldots, x_{j}), x_{j+1}, \ldots, x_{d})
\end{split}
\end{equation}
 and
\begin{equation}\label{formula for Hochschild differential -}
\begin{split}
 & (\delta_{-} (\phi_{-}))^{d}(x_{1}, \cdots, x_{d}) \\
 = & \sum (-1)^{|\phi_{-}| - 1 + |x_{1}| + \cdots + |x_{i}| - i}  \phi_{-}^{d - j + 1}(x_{1}, \ldots, \mu_{\cc}^{j}(x_{i+1}, \ldots, x_{i+j}), \ldots, x_{d}) \\
 & + \sum (-1)^{(|\phi_{-}| )(|x_{1}| + \cdots + |x_{i}| - i)} \mu_{(Y^r_K)^* \otimes_{\K} Y^r_L}^{k - i - j, i}(x_{1}, \ldots, \phi_{-}^{j}(x_{i+1}, \ldots, x_{i+j}), \ldots, x_{d}) \\
 = & \sum (-1)^{|\phi_{-}| - 1 + |x_{1}| + \cdots + |x_{i}| - i}  \phi_{-}^{d - j + 1}(x_{1}, \ldots, \mu_{\cc}^{j}(x_{i+1}, \ldots, x_{i+j}), \ldots, x_{d}) \\
 & + \sum  \mu_{(Y^r_K)^* \otimes_{\K} Y^r_L}^{d - j, 0}(x_{1}, \ldots, x_{d-j}, \phi_{+}^{j}(x_{d-j+1}, \ldots, x_{d}))  \\
 & + \sum  \mu_{(Y^r_K)^* \otimes_{\K} Y^r_L}^{0, d - j}(\phi_{+}^{j}(x_{1}, \ldots, x_{j}), x_{j+1}, \ldots, x_{d})
\end{split}
\end{equation}
We can also write out the $\mu^{2}$-product explicitly as the composition
\begin{equation}\label{mu2 in cinf}
    \mu^{2}_{\cinf}(\phi_2, \phi_1)^{d}(x_{1}, \ldots, x_{d}) := \sum_{\ell} (-1)^{|\phi_{1}|} \phi_2(x_1, \ldots, x_{\ell}) \circ \phi_1(x_{\ell+1}, \ldots, x_d)
\end{equation}
where composition is in the category $\calk_{\K}$, i.e., is given by
\eqref{calkincomp} with respect to the decomposition
\eqref{cinfmorphismspelledout}.

\subsection{Smoothness and Koszul duality}\label{section:koszuldual} 

In case of a smooth $\ainf$-category $\cc$, there are some useful results and formulas which will be convenient for computational purposes.
In this subsection, we include three such: the first one is a formula for taking tensor products, the second one is a formula for the categorical formal punctured neighborhood of infinity as a bimodule over the original category, and the third one has to do with understanding the categorical formal punctured neighborhood of infinity in the setting of {\it Koszul duality}.

Let $\cc$ be an $\ainf$-category, $\mathcal{B}$ a bimodule over $\cc$, and $\mathcal{N}$ a right $\cc$-module, or a bimodule over $\cc$.
The $\K$-linear tensor product 
\begin{equation}
\mathcal{B} \otimes_{\K} \mathcal{N}
\end{equation}
is a quadmodule (if $\mathcal{N}$ is a bimodule), or a trimodule (if $\mathcal{N}$ is a right module).
Consider the following Hochschild cohomology cochain complex
\begin{equation}\label{hochschild tensor product coeff}
\r{CC}^{*}(\cc, \mathcal{B} \otimes_{\K} \mathcal{N}),
\end{equation}
which is either a bimodule or a right module.
Let $\mathcal{P}$ be another bimodule over $\cc$, or a left module.
An interesting question is: when taking tensor product of the above Hochschild cohomology cochain complex over $\cc$ with $\mathcal{P}$ (with respect to the left action on $\mathcal{P}$ if $\mathcal{P}$ is bimodule),
\[
\r{CC}^{*}(\cc, \mathcal{B} \otimes_{\K} \mathcal{N}) \otimes_{\cc} \mathcal{P},
\]
can we bring in $\mathcal{P}$ into the coefficients of the Hochschild cohomology cochain complex?

First let us consider the case where $\mathcal{N}$ is a bimodule.
Note that the bimodule structure on $\r{CC}^{*}(\cc, \mathcal{M} \otimes_{\K} \mathcal{N})$ comes from the quadmodule structure on $\mathcal{B} \otimes_{\K} \mathcal{N}$,
such that the right action on $\mathcal{B}$ and the left action on $\mathcal{N}$ survive after taking Hochschild cochains. 
Therefore, taking the tensor product with $\mathcal{P}$ over $\cc$ induces a canonical map
\begin{equation}\label{bring in tensor product}
\r{CC}^{*}(\cc, \mathcal{B} \otimes_{\K} \mathcal{N}) \otimes_{\cc} \mathcal{P} \to \r{CC}^{*}(\cc, (\mathcal{B} \otimes_{\cc} \mathcal{P}) \otimes_{\K} \mathcal{N}).
\end{equation}
Now suppose $\mathcal{N}$ is a right module.
Then the Hochschild complex \eqref{hochschild tensor product coeff} is a right module with the right action induced by that on $\mathcal{B}$. 
For the same reasoning, we again have the map \eqref{bring in tensor product} and a variant of the arguments in \cite{ganatra}*{Lemma 2.1, Proposition 2.16} shows:

\begin{prop}\label{prop: bring in tensor product}
Suppose $\cc$ is smooth. Then the map \eqref{bring in tensor product} is a quasi-isomorphism, in all circumstances listed above.\qed
\end{prop}

For the second result, let us recall that for a bimodule $\mathcal{P}$ over $\cc$, its {\it bimodule dual} is defined to be
\begin{equation}
\mathcal{P}^{!} := \hom_{\cc-\cc}(\mathcal{P}, \cc_{\D} \otimes_{\K} \cc_{\D}).
\end{equation}
In particular, for the diagonal bimodule, its bimodule dual is called the {\it inverse dualizing bimodule} of $\cc$,
\begin{equation}\label{inverse dualizing bimodule}
\cc^{!} := \hom_{\cc-\cc}(\cc_{\D}, \cc_{\D} \otimes_{\K} \cc_{\D}).
\end{equation}
Using a canonical quasi-equivalence from the ordinary Hochschild cohomology cochain complex to the two-pointed Hochschild cohomology cochain complex as in \cite{ganatra}, we obtain a bimodule quasi-isomorphism
\begin{equation}\label{cc as inverse dualizing bimodule}
\r{CC}^{*}(\cc, \cc_{\D} \otimes_{\K} \cc_{\D}) \stackrel{\sim}\to \hom_{\cc-\cc}(\cc_{\D}, \cc_{\D} \otimes_{\K} \cc_{\D}) = \cc^{!}.
\end{equation}
By definition, for a smooth $\ainf$-category $\cc$, there is a canonical bimodule map given as the following composition 
\begin{equation}
\cc^{!} \otimes_{\cc} \cc^{\vee} \stackrel{\sim}\to \r{CC}^{*}(\cc, \cc^{\vee} \otimes_{\K} \cc_{\D}) \to \r{CC}^{*}(\cc, \hom_{\K}(\cc, \cc)).
\end{equation}
By choosing an inverse of the canonical Yoneda quasi-isomorphism
\[
\cc_{\D} \stackrel{\sim}\to \r{CC}^{*}(\cc, \hom_{\K}(\cc, \cc)),
\]
we get the following bimodule map
\begin{equation}
\cc^{!} \otimes_{\cc} \cc^{\vee} \to \cc_{\D}
\end{equation}
The following proposition relates this map with the formal punctured neighborhood of infinity, $\cinf$.

\begin{prop}\label{prop: cinf as bimodule}
Suppose $\cc$ is smooth. Then, we have a natural quasi-isomorphism of $\cc^{op}$-bimodules: 
\begin{equation}
\cone(\cc_{op}^{!} \otimes_{\cc^{op}} (\cc^{op})^{\vee} \to \cc_{\D}^{op}) \to \cinf.
\end{equation}
\end{prop}
\begin{proof}
The formula \eqref{cinfmorphismspelledout} for morphisms in $\cinf$ implies that $\cinf$ as a bimodule over $\cc^{op}$ can be written as
\begin{equation}
\cinf = \cone(\r{CC}^{*}(\cc^{op}, (\cc^{op})^{\vee} \otimes_{\K} \cc^{op}_{\D}) \to \r{CC}^{*}(\cc^{op}, \hom_{\K}(\cc^{op}_{\D}, \cc^{op}_{\D}))).
\end{equation}
By Yoneda lemma, we have a natural quasi-isomorphism of bimodules
\[
\cc^{op}_{\D} \to \r{CC}^{*}(\cc^{op}, \hom_{\K}(\cc^{op}_{\D}, \cc^{op}_{\D})).
\]
By \eqref{cc as inverse dualizing bimodule}, we have a quasi-isomorphism
\[
\r{CC}^{*}(\cc^{op}, \cc^{op}_{\D} \otimes_{\K} \cc^{op}_{\D}) \otimes_{\cc^{op}} (\cc^{op})^{\vee} \to \cc_{op}^{!} \otimes_{\cc^{op}} (\cc^{op})^{\vee}.
\]
By Proposition \ref{prop: bring in tensor product}, smoothness of $\cc$ implies that there is also a quasi-isomorphism
\[
\r{CC}^{*}(\cc^{op}, \cc^{op}_{\D} \otimes_{\K} \cc^{op}_{\D}) \otimes_{\cc^{op}} (\cc^{op})^{\vee} \to \r{CC}^{*}(\cc^{op}, (\cc^{op}_{\D} \otimes_{\cc^{op}} (\cc^{op})^{\vee}) \otimes_{\K} \cc^{op}_{\D}) \stackrel{\sim}\to \r{CC}^{*}(\cc^{op}, (\cc^{op})^{\vee} \otimes_{\K} \cc^{op}_{\D}).
\]
The proposition follows by inverting the quasi-isomorphism.
\end{proof}

The third result about smoothness is related to Koszul duality. 
To state it, we need to recall a few more definitions.
An object $K$ in $\cc$ is {\em (right, resp. left) proper} if its (right, resp. left) Yoneda module is a proper module, 
meaning that hom into (respectively from) $K$ is a perfect complex for any other object.
Denote by $\cc_{\r{prop}}$ the subcategory of $\cc$ consisting of its right proper objects.
Another notation, $\r{Prop} \cc$, will be reserved for the subcategory of $\rmod \cc$ consisting of proper modules.
Observe that the functor \eqref{calkinyoneda} always sends right proper objects to zero, 
and hence factors through the quotient of $\cc$ by its right proper objects, inducing a functor
\begin{equation}\label{quotienttoinfinity}
	\bar{y}: \cc / \cc_{\r{prop}} \to \cinf.
\end{equation}

The natural question is when this functor is an equivalence.
To recall (and adapt to the $\ainf$ setting) a sufficient criterion for
equivalence established by Efimov, we introduce the following notion of Koszul
duality.

\begin{defn}\label{defn: koszul dualizing subcategory}
Let $\mathcal{C}$ be an $A_{\infty}$-category over $\K$. 
A full subcategory $j: \mathcal{A} \xhookrightarrow{} \mathcal{C}$ is called a {\it Koszul dualizing subcategory}, 
if the following conditions are satisfied:
\begin{enumerate}[label=(\roman*)]

\item $\mathcal{A}$ is a subcategory of $\cc_{\r{prop}}$, 

\item For any two proper right modules $M, N$ over $\mathcal{A}^{op}$, or equivalently proper left modules over $\mathcal{A}$,
the natural morphism 
\[
M^{*} \otimes_{\mathcal{A}} N = N \otimes_{\mathcal{A}^{op}} M^{*} \to \hom_{\rmod \mathcal{A}^{op}}(N, M)^{*}
\]
is a quasi-isomorphism,
where the linear dual $M^{*}$ carries the induced left module structure over $\mathcal{A}^{op}$, or the right module structure over $\mathcal{A}$.

\item The natural functor $\perf \cc^{op} \to \r{Prop} \mathcal{A}^{op}$ on right modules over $\cc^{op}$ and $\mathcal{A}^{op}$
induced by pullback and computing $ \hom_{\rmod \mathcal{A}^{op}}$ is cohomologically fully faithful.

\end{enumerate}

\end{defn}

If $\cc$ is a smooth $\ainf$-category, there is an embedding
\[
	\fun(\cc^{op}, \calk_{\K}) \to \perf(\cc \otimes \calk_{\K}),
\]
by noticing that functors embed in the 'right perfect' bimodules.
Thus, the kernel of $\bar{y}$ is the subcategory of $\cc / \cc_{\r{prop}}$ which represent zero modules over $\cc \otimes \calk_{\K}$.
In fact, with the extra conditions in Definition \ref{defn: koszul dualizing subcategory}, 
one can show moreover that the functor $\bar{y}$ is indeed cohomologically fully faithful.
The following Proposition \ref{koszul implies quotient} is essentially in \cite{efimov}, in the case of dg categories. 
For the convenience of the reader, we include an analogous proof in the case of $\ainf$-categories.
For the proof of Proposition \ref{koszul implies quotient}, 
we will need a certain $\ainf$-module of the form
\begin{equation}
\hom_{\rmod \cc}(\mathcal{P}, \mathcal{N}),
\end{equation}
defined by taking the right $\cc$-module $\hom$ from the right module structure on a given bimodule $\mathcal{P}$ over $\cc$ to a right $\cc$-module.
The underlying cochain space, after specializing a testing object $Y$ for the right module $\hom_{\rmod \cc}(\mathcal{P}, \mathcal{N})$, is
\begin{equation}
\hom_{\rmod \cc}(\mathcal{P}, \mathcal{N})(Y) = \hom_{\rmod \cc}(\mathcal{P}(Y, \cdot), \mathcal{N}).
\end{equation}
The differential $\mu^{0}_{\hom_{\rmod \cc}(\mathcal{P}, \mathcal{N})}$ is given by the differential in the dg category $\rmod \cc$ taking all the right-sided bimodule structure maps $\bigoplus_{l} \mu^{0, l}_{\mathcal{P}}$ of $\mathcal{P}$.
The higher right module structure maps
\begin{equation}
\mu^{k}: \hom_{\rmod \cc}(\mathcal{P}, \mathcal{N})(Y_{0}) \otimes \cc(Y_{1}, Y_{0}) \otimes \cdots \otimes \cc(Y_{k}, Y_{k-1}) \to \hom_{\rmod \cc}(\mathcal{P}, \mathcal{N})(Y_{k})
\end{equation}
are defined by
\begin{equation}
\mu^{k}(f, c_{1}, \ldots, c_{k}) = f_{c_{1}, \ldots, c_{k}} \in \hom_{\rmod \cc}(\mathcal{P}(Y_{k}, \cdot), \mathcal{N})
\end{equation}
where $f_{c_{1}, \ldots, c_{k}}$ is the module morphism consisting of the following sequence of maps
\begin{equation}
f_{c_{1}, \ldots, c_{k}}^{l}(p, a_{1}, \ldots, a_{l}) = \sum_{i} (-1)^{*} f^{l-i}( \mu^{k, i}_{\mathcal{P}}(c_{1}, \ldots, c_{k}, p, a_{1}, \ldots, a_{i}), a_{i+1}, \ldots, a_{l}),
\end{equation}
with sign
\begin{equation}
* = |f| (|a_{i+1}| + \cdots + |a_{l}| - l+j-1).
\end{equation}

\begin{prop}[\cite{efimov} Proposition 9.3]\label{koszul implies quotient}
	If $\cc$ is smooth and has a Koszul dualizing subcategory $\mathcal{A}$, then the functor $\bar{y}$ \eqref{quotienttoinfinity} is a quasi-equivalence.
\end{prop}
\begin{proof}
	It suffices to prove that the natural functor $\cc / \mathcal{A} \to \cinf$ is a quasi-equivalence.
By the construction of $\cinf$, this functor is essentially surjective on objects.
It remains to show that it is cohomologically fully faithful.

Let $K, L \in \ob \cc$,
thought of as objects of $\cc / \mathcal{A}$, and of $\cinf$ via the functor $y$.
We compute their $\hom$ in $\cc / \mathcal{A}$ to be
\begin{equation}
(\cc / \mathcal{A})(K, L) \cong \cone(j^{*} Y^{r}_{L} \otimes_{\mathcal{A}} j^{*} Y^{l}_{K} \to \cc(K, L)).
\end{equation}
We also compute their $\hom$ in $\cinf$ to be
\begin{equation}
\begin{split}
\cinf(y(K), y(L)) \cong & \cone(\r{CC}^{*}(\cc^{op}, (Y^{r}_{K})^{\vee} \otimes_{\K} Y^{r}_{L}) \to \r{CC}^{*}(\cc^{op}, \hom_{\K}(Y^{r}_{K}, Y^{r}_{L}))) \\
\cong & \cone(\r{CC}^{*}(\cc^{op}, (Y^{r}_{K})^{\vee} \otimes_{\K} Y^{r}_{L}) \to \cc(K, L))
\end{split}
\end{equation}
where the second quasi-isomorphism is the statement of Yoneda lemma.
Let $\cc_{\D}$ denote the diagonal bimodule over $\cc$, 
whose underlying complex is
\[
\cc_{\D}(K, L) = \cc(L, K).
\]
Thus, it suffices to find a quasi-isomorphism 
\begin{equation}\label{cc- as tensor}
j^{*} Y^{r}_{L} \otimes_{\mathcal{A}} j^{*} Y^{l}_{K}  \to \r{CC}^{*}(\cc^{op}, (Y^{r}_{K})^{\vee} \otimes_{\K} Y^{r}_{L}),
\end{equation}
that is compatible with maps to $\cc(K, L)$.

Consider the following composition of quasi-isomorphisms:
\begin{equation}\label{composing qis}
j^{*} Y^{r}_{L} \otimes_{\mathcal{A}} j^{*} Y^{l}_{K} \stackrel{\sim}\to j^{*} \hom_{\rmod \cc}(\cc_{\D}, Y^{r}_{L}) \otimes_{\mathcal{A}} j^{*} Y^{l}_{K} \stackrel{\sim}\to \r{CC}^{*}(\cc^{op}, ((\id \times j)^{*} \cc_{\D})^{\vee} \otimes_{\K} Y^{r}_{L}) \otimes_{\mathcal{A}} j^{*} Y^{l}_{K}.
\end{equation}
Here $j^{*} \hom_{\rmod \cc}(\cc_{\D}, Y^{r}_{L})$ is the right module over $\mathcal{A}$ defined as
\begin{equation}
j^{*} \hom_{\rmod \cc}(\cc_{\D}, Y^{r}_{L})(-) = \hom_{\rmod \cc}(\cc_{\D}(j(-), \cdot), Y^{r}_{L}).
\end{equation}
The second quasi-isomorphism in \eqref{composing qis} follows from the assumption that the diagonal bimodule $\cc_{\D}$ is perfect,
as well as that $\mathcal{A}$ is a proper subcategory.

Again because $\cc$ is smooth, we can `bring in the tensor factor' by Proposition \ref{prop: bring in tensor product}, and obtain a quasi-isomoprhism
\begin{equation}\label{qi 1}
\r{CC}^{*}(\cc^{op}, ((\id \times j)^{*} \cc_{\D})^{\vee} \otimes_{\K} Y^{r}_{L}) \otimes_{\mathcal{A}} j^{*} Y^{l}_{K} \stackrel{\sim}\to \r{CC}^{*}(\cc^{op}, ((\id \times j)^{*} \cc_{\D})^{\vee} \otimes_{\mathcal{A}} j^{*} Y^{l}_{K} \otimes_{\K} Y^{r}_{L}).
\end{equation}
Then, by the condition (ii) of Definition \ref{defn: koszul dualizing subcategory},
we have a quasi-isomoprhism
\begin{equation}\label{qi 2}
\begin{split}
& \r{CC}^{*}(\cc^{op}, ((\id \times j)^{*} \cc_{\D})^{\vee} \otimes_{\mathcal{A}} j^{*} Y^{l}_{K} \otimes_{\K} Y^{r}_{L}) \\
= &  \r{CC}^{*}(\cc^{op}, j^{*} Y^{l}_{K} \otimes_{\mathcal{A}^{op}} ((\id \times j)^{*} \cc_{\D})^{\vee} \otimes_{\K} Y^{r}_{L})\\
\stackrel{\sim}\to & \r{CC}^{*}(\cc^{op}, \hom_{\rmod \mathcal{A}^{op}}(j^{*}Y^{l}_{K}, (\id \times j)^{*} \cc_{\D})^{\vee} \otimes_{\K} Y^{r}_{L}).
\end{split}
\end{equation}
Finally, by the condition (iii) of Definition \ref{defn: koszul dualizing subcategory},
we obtain a quasi-isomorphism
\begin{equation}\label{qi 3}
\begin{split}
& \r{CC}^{*}(\cc^{op}, \hom_{\rmod \mathcal{A}^{op}}(j^{*}Y^{l}_{K}, (\id \times j)^{*} \cc_{\D})^{\vee} \otimes_{\K} Y^{r}_{L})  \\
\stackrel{\sim}\to & \r{CC}^{*}(\cc^{op}, \hom_{\rmod \cc^{op}}(Y^{l}_{K}, \cc_{\D})^{\vee} \otimes_{\K} Y^{r}_{L}) \\
= & \r{CC}^{*}(\cc^{op}, \hom_{\rmod \cc}(\cc_{\D}, Y^{r}_{K})^{\vee} \otimes_{\K} Y^{r}_{L}) \\
\stackrel{\sim}\to & \r{CC}^{*}(\cc^{op}, (Y^{r}_{K})^{\vee} \otimes_{\K} Y^{r}_{L}).
\end{split}
\end{equation}
Combining \eqref{composing qis}, \eqref{qi 1}, \eqref{qi 2}, \eqref{qi 3}, we get the desired quasi-isomorphism \eqref{cc- as tensor},
which completes the proof of cohomological fully faithfulness.
\end{proof}

\section{Moduli spaces of domains and maps}\label{sec:modulispaces}

In this section, we collect the various moduli spaces of disks and maps (and their compactifications and boundary strata) that are required for   the main constructions of our paper. Our construction begins by fusing together the disks and Floer equations of \cite{abouzaid1} with the ``popsicle framework'' introduced in \cite{abouzaidseidel}, which we use in a manner somewhat different from {\em loc. cit.}, as a convenient mechanism for tracking ``higher homotopies'' (as also used in e.g., \cite{seidel6}); see Remark \ref{remonweights} for more elaboration on the technical differences.
In particular, unlike the construction in \cite{abouzaidseidel} or \cite{seidel6}, it will turn out (as in \cite{abouzaid1}) that the associated
definition of wrapped Floer cohomology/wrapped Fukaya categories given in \S \ref{section:wrappedfukayacategory} using such moduli spaces will not require any non-trivial ``popsicles''; the variants of popsicles described here will be mostly used in our framework of Rabinowitz Fukaya categories (discussed in Sections \ref{section:rc} and \ref{section:A-infinity on rc}) as well as in the proof of Theorem \ref{thm:main}.
The other main difference is that the inputs and outputs will carry an addition decoration of being ``positive'' or ``negative'', 
which means that when it comes to defining maps from popsicles, we reverse negative inputs into positive outputs and take the dual.

\subsection{Popsicle domains}

Popsicles are introduced in \cite{abouzaidseidel} to define the wrapped Fukaya category using linear Hamiltonians.
To keep the review short, we closely follow the exposition sketched in \cite{rittersmith}.  
For an integer $d \ge 1$, let $F$ be a finite set, and $\mathbf{p}: F \to \{1, \ldots d\}$ a map that defines labels, where we indicate $p_f:=\mathbf{p}(f)$ for $f \in F$. (In general there is no need for the $p_f$ to be distinct; however when we define the $A_{\infty}$-structure maps for the wrapped Fukaya category, we will restrict to the case where $F \subseteq \{1, \ldots, d\}$ and $\mathbf{p}$ is the inclusion.)

A {\it popsicle} of flavor $\mathbf{p}$ consists of a pair $(S, \sigma)$, where
\begin{enumerate}[label=(\roman*)]

\item $S$ is a Riemann surface isomorphic to a disk with $d+1$ boundary punctures
\[
S := D^2\setminus\{z_0, \dots, z_d\},
\]
where $z_0, ..., z_d$ are points of $\partial D^2$, ordered cyclically counterclockwise along the boundary of $D^{2}$.

\item $\sigma$ is a collection of holomorphisms
\[
\left\{\sigma_f: S\rightarrow D^2\right\}_{f\in F}
\]
that each extend to a biholomorphism
\[
    \overline{\sigma}_f: S\cup\{z_1, \dots, \hat{z}_{p_f}, \dots, z_d\} \rightarrow \R\times [0, 1]
\]
taking the puncture $z_0$ to $-\infty$ and the puncture $z_{p_f}$ to $+\infty$.
Equivalently, we can think of $\sigma_{f}$ as a choice of a preferred point on the unique hyperbolic geodesic $C_{p_{f}}$ connecting the points at infinity $z_{0}$ and $z_{p_{f}}$. This point is $\sigma_{f}(0, \frac{1}{2})$, 
which we interchangeably denote by the same symbol $\sigma_{f}$ as well.
Call the geodesics $C_{k}$ {\it popsicle sticks} and the points $\sigma_{f} \in C_{p_{f}}$ {\it sprinkles}. Observe that the size $|\mathbf{p}^{-1}(i)|$ is the number of sprinkles on the $i$th geodesic $C_i$.

\end{enumerate}

A popsicle $(S, \sigma)$ of flavor $\mathbf{p}$ is stable if $d + |F| \ge 2$. The stable popsicles form a moduli space $\mathcal{R}^{d+1, \mathbf{p}}$, which is a smooth manifold of dimension $d + |F| - 2$, diffeomorphic to $\mathbb{R}^{d + |F| - 2}$.
This  moduli space can be compactified by adding broken popsicles.
Broken popsicles are modelled over a ribbon tree $T$ with $d$ leaves and one root, just as in the case of disks. 
Suppose for each vertex $v$, the valency $|v| \ge 2$.
Let $\mathbf{F} = \{F_{v}\}_{v \in T}$ be a decomposition of $F$ into subsets,
and let $j_{v}: F_{v} \to F$ be the inclusion map.
Let $\mathbf{p}_{v} : F_{v} \to \{1, \cdots, |v|-1\}$ be the map induced by $\mathbf{p}$.
The decomposition and the induced maps should satisfy the following condition: for each $f \in F_{v}$, the $\mathbf{p}_{v}(f)$-th edge adjacent to $v$ either is the $p_f$-th leaf of the tree, or lies on the path from $v$ to that leaf.

A broken popsicle of type $T$ and flavors $\vec{\mathbf{p}} = \{\mathbf{p}_{v}\}_{v \in T}$ is a collection of popsicles $\{(S_{v}, \sigma_{v})\}_{v \in T}$ where $(S_v, \sigma_v)$ is a popsicle of flavor $\mathbf{p}_v$.
It is stable if for every vertex $v$, $|v| + |F_{v}| \ge 3$.
These form a moduli space, the relevant product of moduli spaces of stable popsicles:
\begin{equation}
\mathcal{R}^{T, \vec{\mathbf{p}}} \cong \prod_{v \in T} \mathcal{R}^{|v|, \mathbf{p}_{v}}.
\end{equation}
The stable map compactification of the moduli space $\mathcal{R}^{d+1, \mathbf{p}}$ is the disjoint of all these moduli spaces of popsicles (over all such ribbon trees $T$ with $d$ leaves and one root)
\begin{equation}
\bar{\mathcal{R}}^{d+1, \mathbf{p}} = \coprod_{T, \mathbf{F}} \mathcal{R}^{T, \vec{\mathbf{p}}},
\end{equation}
which is a compact manifold with corners.

There is a group $Aut(\mathbf{p})$ acting on the moduli space $\mathcal{R}^{d+1, \mathbf{p}}$, the group of automorphisms of the index set $F$ preserving $\mathbf{p}$, i.e $g \in Aut(\mathbf{p}) \subset Sym(F)$ if $p_{f} = p_{g(f)}$.
Note that if $\mathbf{p}$ is injective, then $Aut(\mathbf{p})$ is trivial.

\subsection{Auxiliary data}\label{subsec:weights}

Fix a $\mathbf{p}$ as above.
Let $\mathbf{w} = (w_{0}, \ldots, w_{d})$ be a collection of non-positive integers, $w_{i} \in \mathbb{Z}_{\le 0}$, called weights,
satisfying the following conditions
\begin{equation}\label{weight-sprinkle}
w_{0} = w_{1} + \cdots + w_{d} + |F|, 
\end{equation}
\begin{equation}\label{maximum number of sprinkles}
|\mathbf{p}^{-1}(i)| \le - w_{i} \text{ for all } i = 1, \ldots, d.
\end{equation}
We think of each weight $w_{i}$ as being assigned to the $i$th puncture $z_{i}$ of a popsicle of flavor $\mathbf{p}$ (e.g., $w_0$ is assigned to the outgoing puncture $z_0$). 
The moduli space $\mathcal{R}^{d+1, \mathbf{p}}$ can be in particular be decorated with associated weights $\mathbf{w}$, with result denoted $\mathcal{R}^{d+1, \mathbf{p}, \mathbf{w}}$, though there is nothing changed in the space itself.

A collection of weights for a popsicle uniquely determines collections of weights $(w_{v, 0}, \ldots, w_{v, |v|-1})$ for each component of a broken popsicle in the following way.
On each component $(S_{v}, \sigma_{v})$, if the $ith$ incoming puncture $z_{v, i}$ is one of the original points $z_j \in \{z_{0}, \cdots, z_{d}\}$, then we associate weight $w_{v,i}$ equal to the original weight $w_j$ assigned to that point. 
In the other cases, whenever two components are connected by an edge, the weights on these two ends must agree.
In addition, the weights for each component satisfy the condition
\begin{equation}
w_{v, 0} = w_{v, 1} + \cdots + w_{v, |v|-1} + |F_{v}|.
\end{equation}
(which determines the outgoing weight in terms of the incoming weights and $F_v$).
Also, the analogue of \eqref{maximum number of sprinkles} carries over to each component.

We are mainly interested in the case 
\begin{equation}\label{weights -1 and 0}
w_{0}, \ldots, w_{d} \in \{-1, 0\}.
\end{equation}
We say the point $z_i$ is a {\em positive marked point} if $w_i = 0$ and a {\em negative marked point} if $w_i =-1$.
If $w_{j} = 0$ for some $j > 0$, then $\mathbf{p}^{-1}(j) = \varnothing$, which means that there is no sprinkle on the geodesic $C_{j}$.
If $w_{j} = -1$, then $ |\mathbf{p}^{-1}(j)| \le 1$.
Note that \eqref{weights -1 and 0} might not be satisfied by broken popsicles which appear in the boundary strata of the moduli space of popsicles.
However, we have the following result:

\begin{lem}[\cite{seidel6}, Lemma 4.2]\label{undesired popsicles}
Let $\mathbf{w}$ be a collection of weights in $\{-1, 0\}$.
If in a boundary stratum of $\bar{\mathcal{R}}^{d+1, \mathbf{p}, \mathbf{w}}$ of codimension $1$ of type $T = \{v_2 \to v_1\}$ the induced weights for the components of broken popsicles do not lie in $\{-1, 0\}$, i.e. for some $i$,
\begin{equation}
w_{v_1, i} = w_{v_2, 0} < -1,
\end{equation}
then either of the following is true:
\begin{enumerate}[label=(\roman*)]

\item $|\mathbf{p}_{v_1}^{-1}(i)| \ge 2$, or

\item $w_{v_1, 0} = w_{0} = -1, w_{v_1, i} = w_{v_2, 0} = -2$, and $\mathbf{p}_{v_1}^{-1}(i)$ consists of one element $f$.
    In this case, there are exactly two possible values for $k = \mathbf{p}
    (j_{v_1}(f))$ 
    (i.e., for indices $k$ of geodesics $C_k$ where the sprinkle $\sigma_{j_{v_1}(f)}$ can lie before breaking),
satisfying $w_{k} = w_{v_2, k - i + 1} = -1$ and $\mathbf{p}_{v_2}^{-1}(k-i+1) = \varnothing$.

\end{enumerate}
\end{lem}

An example of a popsicle in $\mathcal{R}^{5, \mathbf{p}, \mathbf{w}}$ is sketched in Figure \ref{fig:popsicle},
where all the weights need to be either $0$ or $-1$.
One can figure out the weights by the following rule: 
the dotted geodesics are only drawn from inputs $z_{j}$ with $w_{j}=-1$;
the weight $w_{0}$ associated to the output $z_{0}$ is the sum of the input weights $w_{j}$'s minus the number of sprinkle points.

Return to a popsicle with a smooth domain $S$, carrying weights $w_{j} \le 0$.
 Equip $S$ with {\it strip-like ends} $\{\epsilon_j\}_{j\in\{0, \dots, d\}}$ in the following way.
 Let $Z^{\pm} = \{s: \pm s \ge 0\} \times [0, 1]$ denote the positive/negative half-infinite strips,
 with standard coordinates $(s, t)$.
 Define symbols 
\begin{equation}\label{output symbol}
\d_{j} = 
\begin{cases}
0, & \text{ if } j = 0, w_{j} < 0 \text{ or } j > 0, w_{j} = 0, \\
-1, & \text{ if } j = 0, w_{j} = 0 \text{ or } j > 0, w_{j} > 0.
\end{cases}
\end{equation}
Further, introduce a notation for signs:
\begin{equation}\label{sign notation}
s_{j} = 
\begin{cases}
+, & \text{ if } j = 0, w_{j} < 0 \text{ or } j > 0, w_{j} = 0, \\
-, & \text{ if } j = 0, w_{j} = 0 \text{ or } j > 0, w_{j} > 0.
\end{cases}
\end{equation}
The $j$-th strip-like end near $z_{j}$ is a holomorphic embedding
\begin{equation}
\epsilon_{j}: Z^{s_{j}} \to S,
\end{equation}
such that $\epsilon_{j}^{-1}(\partial S) = \{t = 0, 1\}$
and $\lim\limits_{s \to s_{j} \infty} \epsilon_{j}(s,t) = z_j$.
We require that the images of these $\epsilon_{j}$'s be disjoint. Note that for positive punctures $z_i$, the sign of the strip-like end is positive if $z_i$ is an input ($i>0$) and negative if $z_i$ is an output ($i<0$), and that for negative punctures $z_i$ the situation is reversed.

\begin{figure}
	\centering
	\def\svgwidth{\columnwidth}
	\resizebox{0.5\textwidth}{!}{\includegraphics{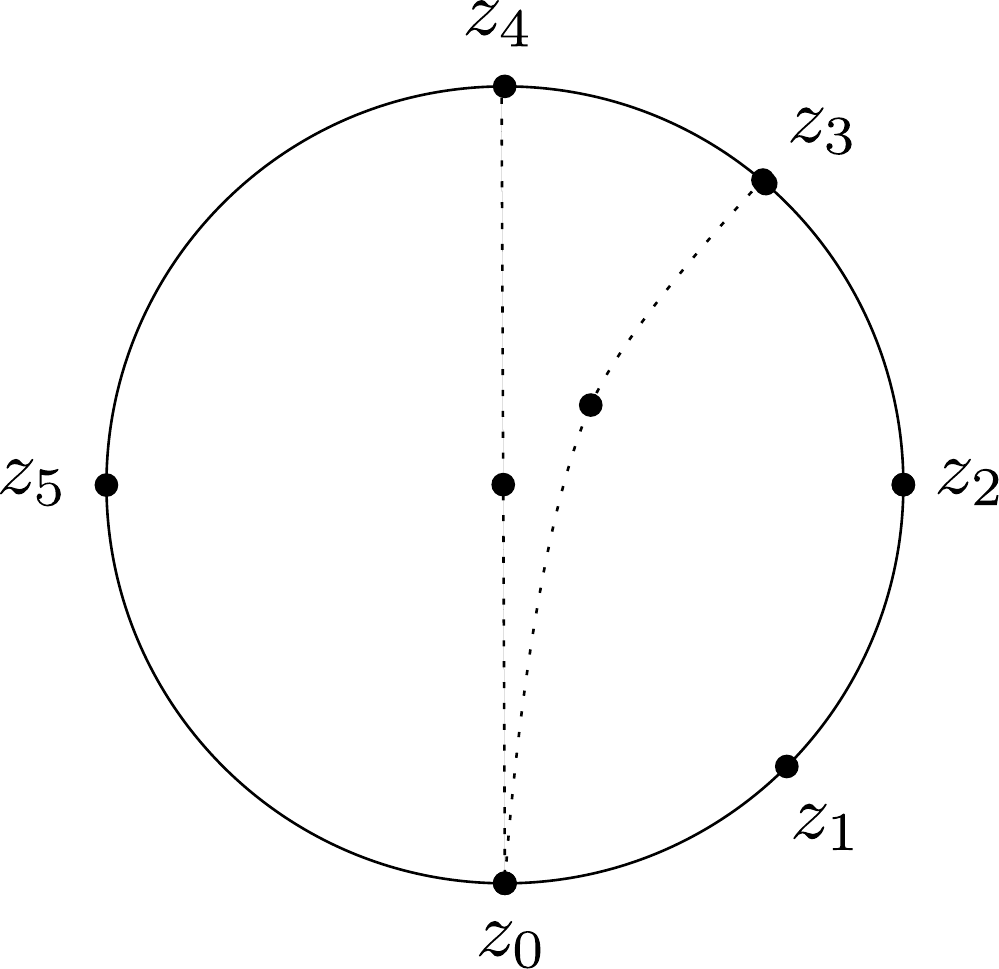}}
	\caption{A popsicle with two sprinkles}
	\label{fig:popsicle}
\end{figure}

Notably, there is a sub-collection of popsicles satisfying
\begin{equation}
w_{0} = w_{1} = \cdots = w_{d} = 0.
\end{equation}
This condition is automatically inherited by components of a broken popsicle.
This sub-collection is precisely the usual space of disks with boundary punctures, i.e. popsicles without sprinkles,
whose moduli space is denoted
\[
\mathcal{R}^{d+1} = \mathcal{R}^{d+1, \varnothing, \mathbf{0}}.
\]

\subsection{Geometric preliminaries}

Having described the moduli spaces of domains in the previous two subsections,
we now describe the target space geometry.
We are mainly interested in a class of open symplectic manifolds, called {\it Liouville manifolds}.
Recall that a Liouville manifold is a symplectic manifold $(X, \omega)$ with an exact symplectic form $\omega = d\lambda$,
where $\lambda$ is called the {\it Liouville form}, satisfying a convexity condition ``near infinity'' expressed in terms of the
{\it Liouville vector field} $Z$,
defined to be the symplectic dual of $\lambda$:
\begin{equation}
    i_{Z} \omega = \lambda.
\end{equation}
Our convention for $Z$ is such that the time-$\log(\rho)$-flow of $Z$, $\psi^{\rho}$, conformally expands the symplectic form:
\begin{equation}
(\psi^{\rho})^{*} \omega = \rho \omega.
\end{equation}
The convexity condition then requires that outside of a compact set in $X$, the vector field $Z$ can be modeled by the expanding vector field $r \partial_r$ on the positive symplectization $([1, \infty)_r\times\Sigma, r\alpha)$ of a contact manifold $(\Sigma, \alpha)$.  
Abusing notation, we identify $X$  outside of this compact set with $[1, \infty)\times\Sigma$, thus identifying $\lambda$ with $r\alpha$.

Let $L$ be a Lagrangian submanifold of $X$.  We say that $L$ is {\it admissible} if
\begin{enumerate}[label=(\roman*)]
\item $L$ is {\it exact}, that is, there exists a function $f_L: L\rightarrow\R$ such that $\theta\big|_{L} = df_L$,
\item $L$ is {\it cylindrical at infinity}, that is, there exists a Legendrian submanifold $\Lambda$ of $\Sigma$ such that $L$ takes the form  $[1, \infty)\times\Lambda$ on the infinite end $[1, \infty)\times\Sigma$, and
\item the relative first Chern class $2c_1^{X, L} \in H^{2}(X, L)$ and second Stiefel-Whitney class $w_2(L) \in H^2(L; \Z_2)$ vanish.
\end{enumerate}
Condition (iii) implies that $L$ can be equipped with a {\it Spin structure} and a {\it grading}.  
Abusing notation, we implicitly assume that an admissible Lagrangian also comes with the data of a choice of a Spin structure (and therefore orientation) and grading.

We say a Hamiltonian $H: X \to \R$ is {\it admissible},
if away from compact set and in the specified cylindrical coordinates,
\begin{equation}
H(r, y) = r^{2}.
\end{equation}
The set of all admissible Hamiltonians (which implicitly depends on a choice of cylindrical coordinate) is denoted by
\begin{equation}
\mathcal{H}(X).
\end{equation}
  Let $X_H$ be the Hamiltonian vector field associated to $H$, uniquely defined by 
\[
dH(-) = \omega(-, X_H).
\]
Denote the flow of $X_H$ by $\phi_H^{t}$.  

Let $c > 0$ be a positive real number.
An $\omega$-compatible almost-complex structure $J$ is said to be {\it $c$-rescaled contact type} on the cylindrical end,
if it satisfies 
\begin{equation}
\label{eq:cylindrical}
\frac{c}{r} \lambda \circ J = dr
\end{equation}
on $[1, \infty)\times\Sigma$.
The set of all $c$-rescaled contact type almost complex structures is denoted by
\begin{equation}
\mathcal{J}_{c}(X).
\end{equation}

We assume that the Liouville form $\lambda$ is chosen generically such that
\begin{equation}\label{non-degenerate Reeb dynamics}
\begin{split}
& \text{ all Reeb orbits of } \alpha \text{ are non-degenerate, and } \\
 &\text{all Reeb chords between Legendrian boundaries of admissible Lagrangians are non-degenerate}.
\end{split}
\end{equation}

\subsection{Inhomogeneous pseudoholomorphic maps}\label{subsec:floerdatum}

Now let us consider maps from popsicles to our target Liouville manifold $X$.
Consider a popsicle $(S, \sigma)$, equipped with weights $w_{j} \le 0$.
To write down the appropriate inhomogeneous Cauchy-Riemann equation on $S$, we need to choose Floer data.
Let us further equip each strip-like end $\epsilon_{j}$ with a choice of positive real number $\nu_{j} \ge 1$,
which we will call a {\it rescaling factor} for the strip-like end.
These rescaling factors are called and in fact identified with the popsicle weights in \cite{abouzaidseidel},
but they are different in our setup, and are not to be confused with the weights $w_{j}$ chosen for the popsicle.
The rescaling factors $\nu_{j}$ for strip-like ends should satisfy the following inequality:
\begin{equation}\label{stokes}
\sum_{j=0}^{d} (-1)^{\d_{j}} \nu_{j} \le 0,
\end{equation}
where $\d_{j}$ are as in \eqref{output symbol}.
In addition, fix a time-dependent almost complex structure $J_{t}: [0, 1] \to \mathcal{J}_{1}(X)$.

\begin{defn}\label{def: Floer datum for popsicles}
A Floer datum $\mathbf{D}_{S, \sigma, \mathbf{w}}$ on a popsicle $(S, \sigma)$ with weights $\mathbf{w}$ in $\mathbb{Z}_{\le 0}$ consists of the following choices:
\begin{enumerate}[label=(\roman*)]
\item A collection of strip-like ends $\epsilon_{j}, j = 0, \ldots, d$;

\item Rescaling factors $\nu_{j} \ge 1$ satisfying \eqref{stokes};

\item A sub-closed one-form $\alpha_{S} \in \Omega^{1}(S)$ such that $\alpha_{S} |_{\partial S} = 0$, $d \alpha_{S} \le 0$ everywhere, and
\[
\epsilon_{j}^{*} \alpha_{S} = \nu_{j} dt.
\]

\item A rescaling function $\rho_{S}: S \to [1, +\infty)$ such that $\rho_{S}$ is constant and equal to $\nu_{j}$ over the $j$-th strip-like end $\epsilon_{j}$.

\item A domain-dependent family of Hamiltonians $H_{S}: S \to \mathcal{H}(X)$ which is compatible with weighted strip-like ends in the following sense:
\begin{equation}
\epsilon_{j}^{*} H_{S} = \frac{H \circ \psi^{\nu_{j}}}{\nu_{j}^{2}}, j = 0, \ldots, d.
\end{equation}

\item A domain-dependent family of almost complex structures $J_{S}$ such that $J_{z} \in \mathcal{J}_{\rho_{S}(z)}(X)$ for every $z \in S$, 
and 
\begin{equation}
\epsilon_{j}^{*} J_{S} = (\psi^{\nu_{j}})^{*} J_{t}, j = 0, \ldots, d.
\end{equation}

\end{enumerate}

\end{defn}
Liouville flow induces an $\R_{>0}$ action on the space of Floer data for any $(S,\sigma)$, preserving the collection of strip-like ends and sending, for $\lambda \in \R_{>0}$, 
\begin{equation}\label{rescalingaction}
    (\{\nu_j\}, \alpha_S, \rho_S, H_S, J_S) \mapsto (\{\lambda \nu_j\}, \lambda \alpha_S, \lambda \rho_S, \frac{H_S \circ \psi^{\lambda}}{\lambda^2}, (\psi^{\lambda})^* J_S).
\end{equation}
There is a further $\R$-action on the space of Floer data by adding a constant to the Hamiltonian function
\begin{equation}\label{constantaction}
    H \mapsto H + C.
\end{equation}

We also extend the definition to the case of an unstable domain $Z$, 
where the popsicle structure is trivial.
In that case, a Floer datum is greatly simplified: 
the rescaling factors are both $1$,
the one-form is $dt$,
the rescaling function is the constant function $1$,
the Hamiltonian is domain-independent $H$,
and the almost complex structure is $J_{t}$.

\begin{rem}\label{remonweights}
The integer weights $w_{j}$ in the definition of our popsicles have no effects on the sub-closed one-forms as part of Floer data chosen on the domain, but only the rescaling factors $\nu_{j}$ do,
whereas in \cite{abouzaidseidel}, \cite{seidel6} these two collections of numbers coincide, and together with the popsicle structures, they determine the asymptotics of the sub-closed one-form.
\end{rem}

A {\it Lagrangian label} is a collection of admissible Lagrangians $L_0, \ldots L_d$, each of which is assigned to a boundary component of $S$.  
We denote the component between $z_j$ and $z_{j+1}$ by $\partial_{j} S$ for $0\leq j < d$, 
and the component between $z_0$ and $z_d$ by $\partial_{d} S$.  
Then $L_{j}$ is assigned to the boundary component $\p_{j} S$.

Having chosen a Floer datum and a Lagrangian label for the $\mathbf{w}$-weighted popsicle $(S, \sigma)$,
we can write down the {\it inhomogeneous Cauchy-Riemann equation} for a map $u: S \to X$
\begin{equation}\label{eq:genfloer}
(du - X_{H_{S}} \otimes \alpha_{S})^{0, 1} = \frac{1}{2}[(du - X_{H_{S}} \otimes \alpha_{S}) + J_{S} \circ (du - X_{H} \otimes \alpha_{S}) \circ j ] =0, 
\end{equation}
such that the following boundary and asymptotic conditions are satisfied:
\begin{equation}\label{boundary and asymptotics}
\begin{cases}
u(z) \in (\psi^{\rho_{S}(z)})^{*} L_{j} \text{ for } z \in \p_{j} S \\
\lim\limits_{s \to s_{j} \infty} u \circ \epsilon_{j}(s, \cdot) = (\psi^{\nu_{j}})^{*} x_{j},
\end{cases}
\end{equation}
where 
\begin{itemize}
    \item For a Lagrangian $L$, $(\psi^\rho)^*(L):=(\psi^{\rho})^{-1}(L)$;
    \item for $j = 1, \ldots, d$, $x_{j}$ is a time-one chord for $H$,
which runs from $L_{j-1}$ to $L_{j}$ if $w_{j} = 0$,
or from $L_{j}$ to $L_{j-1}$ if $w_{j} = -1$; and
\item $x_{0}$ is a time-one chord for $H$
which runs from $L_{0}$ to $L_{d}$ if $w_{0} = 0$
or from $L_{d}$ to $L_{0}$ if $w_{0} = - 1$.
\end{itemize}

In the translation-invariant case where $S = Z = \mathbb{R} \times [0, 1]$ is a strip and the weights are $w_{0} = w_{1} = 0$, 
choose the one-form to be translation-invariant on the whole strip, $\alpha =  dt$,
and the rescaling function $\rho_{Z}$ to be the constant function $1$.
Then \eqref{eq:genfloer} becomes the familiar Floer's equation for a map $u: Z \to X$:
\begin{equation}\label{eq:floer}
\frac{\partial u}{\partial s} + J_{t}\left(\frac{\partial u}{\partial t} - X_H\right) = 0
\end{equation}
and \eqref{boundary and asymptotics} becomes
\begin{equation}\label{boundary and asymptotics 2}
\begin{cases}
u(s, 0) \in L_{0}, u(s, 1) \in L_{1} \\
\lim\limits_{s \to -\infty} u(s, \cdot) = x_{0}, \lim\limits_{s \to +\infty} u(s, \cdot) = x_{1}.
\end{cases}
\end{equation}
In the other case where $w_{0} = w_{1} = -1$,
the equation is the same as \eqref{eq:floer},
but we interchange the roles of the input and the output:
\begin{equation}
\lim\limits_{s \to +\infty} u(s, \cdot) = x_{0}, \lim\limits_{s \to -\infty} u(s, \cdot) = x_{1},
\end{equation}
where $x_0$ and $x_1$ are chords from from $L_{1}$ to $L_{0}$ instead (i.e., $u(s, 0) \in L_{1}$ and $u(s, 1) \in L_{0}$)

\subsection{Moduli space of popsicle maps}\label{section:moduli space of popsicle maps}

The moduli space of popsicle maps is the space of equivalence classes of triple $(S, \sigma, u)$, with $(S,\sigma)$ an arbitrary popsicle, such that $u$ satisfies \eqref{eq:genfloer} and \eqref{boundary and asymptotics}.
In order for such moduli spaces to have nice compactifications, we need to make a {\it universal and conformally consistent} choice of Floer data in the sense of \cite{abouzaid1}. This is by definition a smoothly varying choice of Floer data for every $(S,\sigma)$ in $\bar{\mathcal{R}}^{d+1, \mathbf{p}, \mathbf{w}}$ for each $d, \mathbf{p}, \mathbf{w}$ which agrees with previously made choices along boundary strata up to the $\R_{>0} \times \R$ (Liouville rescaling and constant addition) actions described in \eqref{rescalingaction}-\eqref{constantaction}.
In addition, the space of Floer data for popsicles of type $(d+1, \mathbf{p}, \mathbf{w})$ fibers over $\mathcal{R}^{d+1, \mathbf{p}, \mathbf{w}}$;
and we assume that the action of $Aut(\mathbf{p})$ lifts, meaning that we always choose Floer data to be invariant under $Aut(\mathbf{p})$-action.
This is possible because 
\begin{itemize}

\item the average of one-forms or rescaling functions still satisfies the conditions of a Floer datum;

\item we can take $Aut(\mathbf{p})$-invariant sections of the pullback bundle $\pi^{*}End(TX) \to \mathcal{S}^{d+1, \mathbf{p}, \mathbf{w}} \times X$ and exponentiate them to perturb a fixed almost complex structure.

\end{itemize}

All the choices involved in the definition of a Floer datum form a contractible space, 
so that a universal and conformally consistent choice of Floer data exists by a standard induction argument.
we make such a choice as well as a choice of consistent Lagrangian labels.
Given a collection of Hamiltonian chords $\mathbf{x} = (x_0, \ldots, x_d)$, define 
\begin{equation}\label{popsicle moduli space}
\mathcal{R}^{d+1, \mathbf{p}, \mathbf{w}}(\mathbf{x})
\end{equation}
to be the moduli space of triples $(S, \sigma, u)$,
such that $u$ satisfies \eqref{eq:genfloer} and \eqref{boundary and asymptotics}.
This is called the moduli space of popsicle maps of type $(d+1, \mathbf{p}, \mathbf{w})$.
In the case where $d=1$ and the domain $F$ of $\mathbf{p}$ is empty,
 we use the same notation for the moduli space of non-constant solutions modulo translation action by $\mathbb{R}$, 
 \[
     \mathcal{R}^{2}(x_{0}, x_{1}) = \mathcal{R}^{2, \varnothing, \mathbf{0}}(x_{0}, x_{1}).
 \]
 More generally, when $d$ is arbitrary and the domain of $F$ of $\mathbf{p}$ is empty, \eqref{popsicle moduli space} specializes to the usual moduli spaces of maps 
 \begin{equation}\label{wrapped moduli space}
     \mathcal{R}^{d+1}(\mathbf{x}):= \mathcal{R}^{d+1, \varnothing, \mathbf{0}}(\mathbf{x})
\end{equation}
appearing in the construction of wrapped Fukaya categories as in \cite{abouzaid1}.
 By Sard-Smale and Baire category theorem, the set of Floer data that make all moduli spaces $\mathcal{R}^{d+1, \mathbf{p}, \mathbf{w}}(\mathbf{x})$ smooth is non-empty, and is in fact dense in a Banach manifold. 
Such Floer data are called {\it regular}, and from now on we shall always assume we choose regular Floer data.

\begin{rem}
Our moduli space of popsicle maps is different from the ones considered in \cite{abouzaidseidel} or in \cite{seidel6},
as the Floer data are chosen differently.
Here we turn every input with a negative weight $w<0$ into an output, and vice versa.
And we have chosen additional rescaling factors $\nu_{j}$ to define the one-form which have no relation to the popsicle weights $w_{j}$,
except that they must satisfy \eqref{stokes}.
\end{rem}

\begin{lem}\label{lem:vdim of popsicle moduli space}
The moduli space $\mathcal{R}^{d+1, \mathbf{p}, \mathbf{w}}(\mathbf{x})$ is a smooth manifold of dimension
\begin{equation}\label{virtual dimension formula for moduli space of popsicles}
\dim \mathcal{R}^{d+1, {\bf p}, {\bf w}}({\bf x}) = d - 2 + |F| + n(1 + \sum_{j=0}^{d} \d_{j}) - \sum_{j=0}^{d} (-1)^{\d_{j}} \deg(x_{j}),
\end{equation}
where $\d_{j}$ are as in \eqref{output symbol}.
\end{lem}
\begin{proof}
The dimension follows from an index computation for the associated linearized Fredholm operator, following Proposition 11.13 of \cite{seidel_book}.
\end{proof}

Although the way we choose Floer data is different from the usual way for popsicles used in the linear Hamiltonian framework,
it turns out that the moduli space \eqref{popsicle moduli space} has a natural Gromov compactification
\begin{equation}\label{Gromov compactification}
\bar{\mathcal{R}}^{d+1, \mathbf{p}, \mathbf{w}}(\mathbf{x}),
\end{equation}
which behaves exactly the same as one would expect for moduli spaces of broken popsicles.
The boundary strata are obtained by adding broken popsicles with non-positive weights,
plus additionally inhomogeneous pseudoholomorphic strips \eqref{eq:floer}.
These can be described as follows.
Take $(T, \vec{\mathbf{p}})$ an underlying model for a broken popsicle, but without imposing the stability condition $|v| + |F_{v}| \ge 3$.
For each vertex $v$, we have a popsicle map
\begin{equation}
(S_{v}, \sigma_{v}, u_{v}) \in \mathcal{R}^{|v|, \mathbf{p}_{v}, \mathbf{w}_{v}}(\mathbf{x}_{v}).
\end{equation}
The only unstable case is where $|v| = 2$ and $F_{v} = \varnothing$.
In that case we have a popsicle map in $\mathcal{R}^{2, (w_{v, 0}, w_{v, 1})}(x_{v, 0}, x_{v, 1})$ with $w_{v, 0} = w_{v, 1}$, which is a translation-invariant inhomogeneous pseudoholomorphic strip.
The asymptotic chords $\mathbf{x}_{v} = (x_{v, 0}, \ldots, x_{v, |v|-1})$ should satisfy the following condition:
\begin{itemize}

\item If the $k$-th edge of $v$ corresponds to the $j$-th leaf, then $x_{v, k} = x_{j}$ is one of the given chords for a smooth popsicle $u$;

\item Otherwise, if $v_{-}$ and $v_{+}$ are adjacent vertices connected by a finite edge which is the $0$-th edge for $v_{-}$ and $k$-th edge for $v_{+}$, we require that $x_{v_{-}, 0} = x_{v_{+}, k}$.

\end{itemize}
Under the assumption that our choices of Floer data are regular, we have 
\begin{prop}\label{master prop for smoothness and compactness}
When the virtual dimension \eqref{virtual dimension formula for moduli space of popsicles} is zero, 
then $\mathcal{R}^{d+1, \mathbf{p}, \mathbf{w}}(\mathbf{x})$ is a compact smooth manifold of dimension zero.
When the virtual dimension \eqref{virtual dimension formula for moduli space of popsicles} is one,
the compactified moduli space $\bar{\mathcal{R}}^{d+1, \mathbf{p}, \mathbf{w}}(\mathbf{x})$ is a compact smooth manifold-with-boundary,
with boundary strata given by moduli spaces of stable maps from broken popsicles that are modeled on trees $T$ with two vertices,
such that each component of a broken popsicle map belongs to a zero-dimensional moduli space of popsicle maps with smooth domains.
\end{prop}
\begin{proof}
    Standard Gromov compactness and gluing arguments apply once it is shown, via some version of a maximum principle, that image of elements of the moduli spaces $\mathcal{R}^{d+1, \mathbf{p}, \mathbf{w}}(\mathbf{x})$ lie in a fixed compact subset (possibly depending on $\mathbf{x}$ and the system of Floer data chosen) of $X$. For this needed fact see e.g., \cite{abouzaidseidel}*{Lemma 7.2} or \cite{abouzaid1}*{\S B}.
\end{proof}

\subsection{Disks with no output}\label{section: pairing disks}

The purpose of this subsection is to introduce certain moduli spaces of disks that are needed for establishing a version of bimodule Poincar\'e duality for wrapped Floer theory, whose algebraic formulation will be spelt out in \S\ref{section: Poincare duality}.
From the algebraic perspective of Poincar\'e duality, these disks should be thought of as disks with two main inputs, arbitrarily many auxiliary inputs and no output,
where the two main inputs are equipped with weights $-1$ and $0$ as in the framework of weighted popsicles.
However, analytically it will be more convenient to reverse the negatively weighted input to a positively weighted output,
and thus to consider disks with one main input and one main output.

Let $S_{p}$ be a Riemann surface isomorphic to a disk with two boundary punctures $z_{-}, z_{+}$, 
with one additional interior marked point $\zeta$ such that the domain is stable.
The moduli space of such disks is a one-dimensional smooth connected manifold, and one explicit diffeomorphism with $\mathbb{R}$ can be seen as follows.
Each such disk admits a unique biholomorphism to the unit disk in the complex plane mapping
one of the boundary punctures to $-1$ and the interior marked point $\zeta$ to $0$.
The other boundary marked point is then mapped to some point $z \in \p D^{2} \setminus \{-1\}$, and the space of possible such positions is diffeomorphic to $\mathbb{R}$.

We could think of this as a variant of a popsicle with two input punctures, with first input $z_-$ negative and second input $z_+$ positive i.e., with associated weights $w_{-} = -1$, $w_{+} = 0$ in the sense of \S \ref{subsec:weights}, but we will not emphasize that perspective.
Choose strip-like ends
\begin{equation}
\epsilon_{-}: Z^{-} \to S_{p}
\end{equation}
around $z_{-}$, and
\begin{equation}
\epsilon_{+}: Z^{+} \to S_{p}
\end{equation}
around $z_+$
If we forget the interior marked point $\zeta$, 
$S_{p}$ is just a copy of a strip $Z$. 
However, with the extra interior marked point $\zeta$, there is no translation symmetry on $S_{p}$. 

Choose a Lagrangian label $K, L$ for the two boundary components divided by the two punctures $z_{-}, z_{+}$.
A Floer datum for $S_{p}$ can be defined in a similar way to the one for a strip $Z$, 
except that we do not impose on translation invariance on the family of almost complex structures.

Consider smooth maps
\[
	u: S_{p} \to X
\]
which satisfy the inhomogeneous Cauchy-Riemann equation with respect to the chosen Floer datum, 
satisfy the boundary conditions given by the Lagrangian label, 
and asymptotic conditions given by some Hamiltonian chords $x_{-}, x_{+} \in \mathcal{X}(L, K; H)$.
The moduli space of pairs $(S_p, u)$ where $S_p$ is an arbitrary domain as above, is denoted by
\begin{equation}\label{moduli space of pairing disks without constraint}
\tilde{\mathcal{R}}_{p}(x_{-}, x_{+}).
\end{equation}
A standard index computation similar to Lemma \ref{lem:vdim of popsicle moduli space}, together with the extra one dimension contributed from the moduli space of domains, implies

\begin{lem}
This moduli space has virtual dimension
\begin{equation}
\dim \tilde{\mathcal{R}}_{p}(x_{-}, x_{+}) = \deg(x_{-}) - \deg(x_{+}) + 1.
\end{equation}
\end{lem}\qed

Inside \eqref{moduli space of pairing disks without constraint}, there is a codimension-one subspace, 
obtained by cutting down the moduli space of domains by the constraint that $\zeta$ should only vary along the geodesic between $z_{-}$ and $z_{+}$; this fixes a unique $S_p$.
The resulting subspace of maps from this $S_p$ is denoted by
\begin{equation}\label{moduli space of pairing disks}
\mathcal{R}_{p}(x_{-}, x_{+}),
\end{equation}
which has virtual dimension
\begin{equation}
\dim \mathcal{R}_{p}(x_{-}, x_{+}) = \deg(x_{-}) - \deg(x_{+}).
\end{equation}

There is a natural Gromov compactification 
\[
\bar{\mathcal{R}}_{p}(x_{+, 1}, x_{+, 2})
\]
by adding broken inhomogeneous pseudoholomorphic strips with boundary conditions $(L, K)$ over both ends.
The boundary strata of one-dimensional moduli spaces are covered by the following product moduli spaces of dimension zero:
\begin{align}\label{boundary strata of the moduli space of pairing disks}
& \mathcal{R}_{p}(y_{-}, x_{+}) \times \mathcal{R}^{2}(x_{-}, y_{-}) \\
& \mathcal{R}_{p}(x_{-}, y_{+}) \times \mathcal{R}^{2}(y_{+}, x_{+}). 
\end{align}

\begin{lem}\label{the only pairing disk}
The moduli spaces $\mathcal{R}_{p}(x_{-}, x_{+})$ with $\deg(x_{-}) = \deg(x_{+})$ are all empty unless $x_{-} = x_{+}$,
in which case the only element in the moduli space is the constant map.
\end{lem}
\begin{proof}
Choose Floer data to be translation-invariant. 
If $\deg(x_{-}) = \deg(x_{+})$ but $x_{-} \neq x_{+}$, and $\mathcal{R}_{p}(x_{-}, x_{+}) \neq \varnothing$, 
then there must be a non-constant solution $u: S_{p} \to X$ with asymptotics $x_{-}, x_{+}$.
Then, for any $s \in \R$, $u_{s} = u(\cdot+s, \cdot)$ is another solution, 
which implies that $\dim \mathcal{R}_{p}(x_{-}, x_{+}) \ge 1$, contradicting the degree condition.
\end{proof}

Now we add more boundary punctures.
For $k, l \ge 0$, consider a Riemann surface $S^{k, l}_{p}$ a disk with $k+l+2$ boundary punctures,
\[
z_{-}, z'_{1}, \ldots, z'_{l}, z_{+}, z_{k}, \ldots, z_{1},
\]
ordered cyclically counterclockwise, and an additional interior marked point $\zeta$. 
Such a Riemann surface is obtained from $S_{p}$ by adding $k$ punctures on one boundary component between $z_{-}$ and $z_{+}$, 
and another $l$ punctures on the other boundary component.
Denote by 
\begin{equation}
\mathcal{R}_{p}^{k, l}
\end{equation}
 the moduli space of such disks. See Figure \ref{fig:pairing disk} for a picture of such a disk.
The compactification $\bar{\mathcal{R}}_{p}^{k, l}$ is obtained by adding broken disks,
such that the interior marked point belongs to one disk component of each broken disk,
and the other components belong to $\mathcal{R}^{|v|}$,
provided that the broken disk is modeled by a ribbon tree $T$ with vertices $v$.

\begin{figure}
	\centering
	\def\svgwidth{\columnwidth}
	\resizebox{0.5\textwidth}{!}{\includegraphics{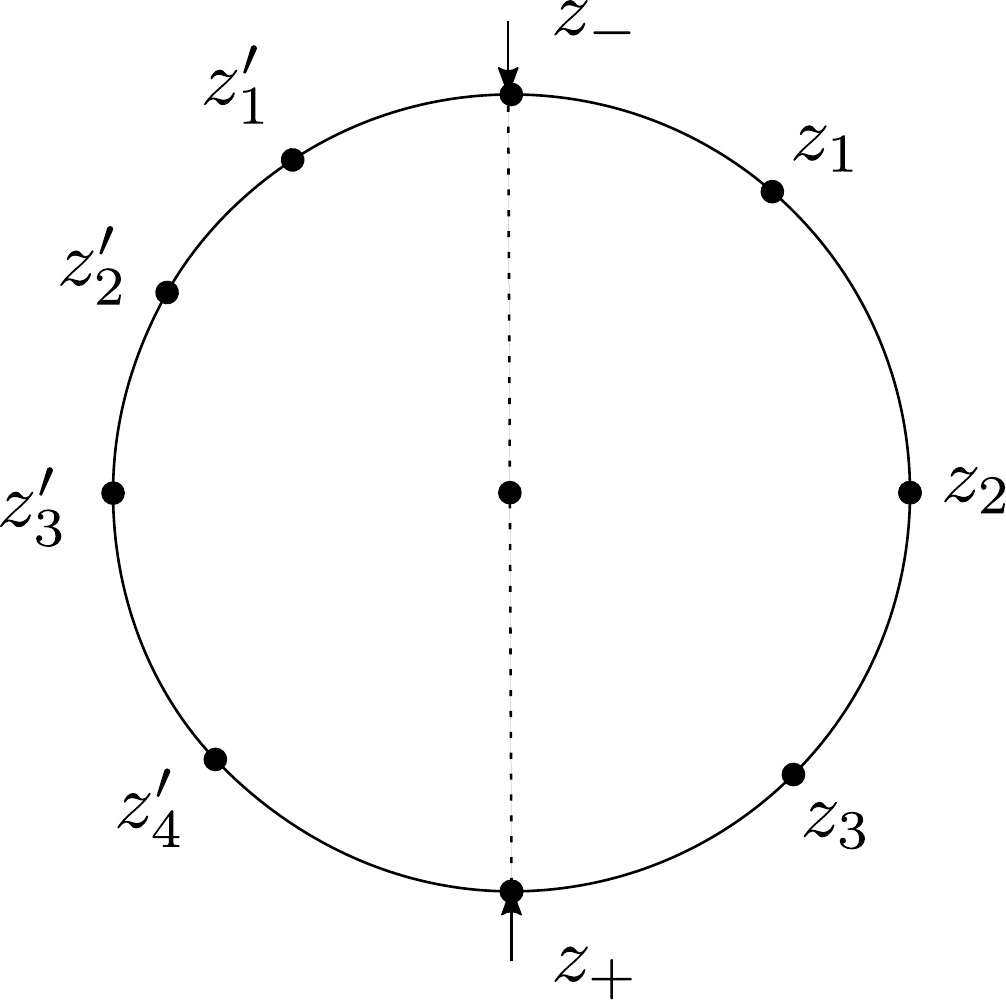}}
	\caption{A disk with two main inputs, $k, l$ auxiliary inputs and one interior marked point}
	\label{fig:pairing disk}
\end{figure}

Choose the strip-like ends near $z_{-}, z_{+}$ as before,
and positive strip-like ends near the other punctures
\[
\epsilon_{i}: Z^{+} \to S_{p}^{k, l}, i = 1, \ldots, k,
\]
and
\[
\epsilon'_{j}: Z^{+} \to S_{p}^{k, l}, j = 1, \ldots, l.
\]
We want to equip these strip-like ends with weights $\nu_{-}, \nu_{+}, \nu_{i}, \nu'_{j}$, which are positive real numbers $\ge 1$, 
such that
\begin{equation}\label{stokes for pairing disks}
\nu_{-} \ge \nu_{+} + \sum_{i=1}^{k} \nu_{i} + \sum_{j=1}^{l} \nu'_{j}.
\end{equation}
Following Definition \ref{def: Floer datum for popsicles}, we can define the notion of a Floer datum $\mathbf{D}_{S_{p}^{k, l}}$ for a fixed $S_{p}^{k, l}$.
Since the boundary strata of $\bar{\mathcal{R}}_{p}^{k, l}$ consist of product moduli spaces of the same type together with moduli spaces of popsicles,
the definition of a universal and conformally consistent choice of Floer data is naturally extended to the moduli spaces of such domains $\bar{\mathcal{R}}_{p}^{k, l}$, which is further required to be compatible with that choice for moduli spaces of popsicles.

Choose a Lagrangian label
\[
L_{0}, \ldots, L_{l}, K_{k}, \ldots, K_{0}, 
\]
such that $L_{0}$ is assigned to the boundary component between $z_{-}$ and $z'_{1}$,
$L_{j}$ is assigned to the boundary component between $z'_{j}$ and $z'_{j+1}$, for $j = 1, \ldots, l-1$,
$L_{l}$ is assigned to the boundary component between $z'_{l}$ and $z_{+}$,
$K_{k}$ is assigned to the boundary component between $z_{+}$ and $z_{k}$,
$K_{i}$ is assigned to the boundary component between $z_{i+1}$ and $z_{i}$, for $i = 1, \ldots, k-1$,
and $K_{0}$ is assigned to the boundary component between $z_{1}$ and $z_{-}$.

Suppose we have made a universal and conformally consistent choice of Floer data for all $\bar{\mathcal{R}}_{p}^{k, l}$, all $k, l \ge 0$.
Consider smooth maps
\[
u: S^{k, l}_{p} \to X
\]
that satisfy the inhomogeneous Cauchy-Riemann equation
\begin{equation}
(du - X_{H_{S_{p}^{k, l}}} \otimes \alpha_{S_{p}^{k, l}})^{0, 1} = 0,
\end{equation}
with boundary conditions given by the chosen Lagrangian labels,
\begin{equation}
u(z) \in (\psi^{\rho_{S_{p}^{k, l}(z)}})^{*} K_{i} \text{ or } L_{j},
\end{equation}
and asymptotic conditions 
\begin{align}
\lim\limits_{s \to + \infty} u \circ \epsilon_{i}(s, \cdot) = (\psi^{\nu_{i}})^{*} x_{i}, \\
\lim\limits_{s \to + \infty} u \circ \epsilon'_{j}(s, \cdot) = (\psi^{\nu'_{j}})^{*} x'_{j}, \\
\lim\limits_{s \to - \infty} u \circ \epsilon_{-}(s, \cdot) = (\psi^{\nu_{-}})^{*} x_{-}, \\
\lim\limits_{s \to + \infty} u \circ \epsilon_{+}(s, \cdot) = (\psi^{\nu_{+}})^{*} x_{+},
\end{align}
for some time-one $H$-chords $x_{i}$ from $K_{i}$ to $K_{i-1}$,
$x'_{j}$ from $L_{j-1}$ to $L_{j}$,
$x_{-}$ from $K_{0}$ to $L_{0}$,
and $x_{+}$ from $L_{l}$ to $K_{k}$.
Let
\begin{equation}
\tilde{\mathcal{R}}^{k, l}_{p}(\mathbf{x}, \mathbf{x}'; x_{-}, x_{+})
\end{equation}
denote the moduli space of such maps,
generalizing $\mathcal{R}_{p}(x_{-}, x_{+})$, where $k = l = 0$.
Here the notations for the chords are
\begin{align}
\mathbf{x} & = (x_{k}, \ldots, x_{1}), \\
\mathbf{x}' & = (x_{1}, \ldots, x'_{l}).
\end{align}
As in the case where $k=l=0$, there is a codimension-one subspace obtained by constraining the interior marked point $\zeta$ on the geodesic between $z_{-}$ and $z_{+}$:
\begin{equation}\label{moduli space of pairing disks with marked points}
\mathcal{R}^{k, l}_{p}(\mathbf{x}, \mathbf{x}'; x_{-}, x_{+}).
\end{equation}
The latter space has virtual dimension
\begin{equation}
\dim \mathcal{R}^{k, l}_{p}(\mathbf{x}, \mathbf{x}'; x_{-}, x_{+}) = \deg(x_{-}) - \deg(x_{+}) - \sum_{i=1}^{k} \deg(x_{i}) - \sum_{j=1}^{l} \deg(x'_{j}) + k + l.
\end{equation}
A standard transversality and gluing argument along the lines of Proposition \ref{master prop for smoothness and compactness} implies

\begin{lem}\label{boundary of moduli space of pairing disks with marked points}
There exist universal and conformally consistent choices of Floer data such that the following hold:
\begin{enumerate}[label=(\roman*)]
\item If
\[
\deg(x_{-}) = \deg(x_{+}) + \sum_{i=1}^{k} \deg(x_{i}) + \sum_{j}^{l} \deg(x'_{j}) - k - l,
\]
the moduli space $\mathcal{R}^{k, l}_{p}(\mathbf{x}, \mathbf{x}'; x_{-}, x_{+})$ is a compact smooth manifold of dimension zero.

\item If
\[
\deg(x_{-}) = \deg(x_{+}) + \sum_{i=1}^{k} \deg(x_{i}) + \sum_{j=1}^{l} \deg(x'_{j}) + 1 - k - l,
\]
 the moduli space $\mathcal{R}^{k, l}_{p}(\mathbf{x}, \mathbf{x}'; x_{-}, x_{+})$ is a smooth manifold, 
and its Gromov compactification $\bar{\mathcal{R}}^{k, l}_{p}(\mathbf{x}, \mathbf{x}'; x_{-}, x_{+})$ is a compact smooth manifold of dimension one,
whose codimension-one boundary strata are covered by the following products of moduli spaces of dimension zero:
\begin{equation}\label{boundary strata of moduli space of pairing disks with marked points}
\begin{split}
& \p \bar{\mathcal{R}}^{k, l}_{p}(\mathbf{x}, \mathbf{x}'; x_{-}, x_{+})\\
= & \coprod \mathcal{R}^{k_{1} + l_{1} + 2, \mathbf{p}_{1}, \mathbf{w}_{1}} (y_{-}, \mathbf{x}_{1}, x_{-}, \mathbf{x}'_{1})
 \times_{0, -} \mathcal{R}^{k_{2}, l_{2}}_{p}(\mathbf{x}_{2}, \mathbf{x}'_{2}; y_{-}, x_{+}) \\
\cup & \coprod \mathcal{R}^{k_{1}, l_{1}}_{p}(\mathbf{x}_{1}, \mathbf{x}'_{1}; x_{-}, y_{+}) 
\times_{+, 0} \mathcal{R}^{k_{2} + l_{2} + 2}(y_{+}, \mathbf{x}'_{2}, x_{+}, \mathbf{x}_{2}) \\
\cup & \coprod \mathcal{R}^{k-j+1, l}_{p}(\mathbf{x}_{i, j, 1}, \mathbf{x}'; x_{-}, x_{+})
 \times_{i, 0} \mathcal{R}^{j+1}(\mathbf{x}_{i, j, 2}) \\
\cup & \coprod \mathcal{R}^{k, l-j'+1}_{p}(\mathbf{x}, \mathbf{x}'_{i', j', 1}; x_{-}, x_{+}) 
\times_{i', 0} \mathcal{R}^{j'+1}(\mathbf{x}'_{i', j', 2})
\end{split}
\end{equation}
Here the popsicle structure $\sigma_{1}$ of flavor $\mathbf{p}_{1}$ is trivial, i.e. has no sprinkle, 
and the weights $\mathbf{w}_{1}$ are $-1$ at the $0$-th and $(l_{1}+1)$-th punctures, and $0$ at other punctures.
The notations for the chords are as follows:
\begin{align}
\mathbf{x}'_{1} &= (x'_{1}, \ldots, x'_{l_{1}}), \\
\mathbf{x}'_{2} & = (x'_{l_{1}+1}, \ldots, x'_{l}), \\
\mathbf{x}_{1} &= (x_{k_{1}}, \ldots, x_{1}), \\
\mathbf{x}_{2} & = (x_{k}, \ldots, x_{k_{1}+1}), \\
\mathbf{x}'_{i', j', 1} & = (x'_{1}, \ldots, x'_{i'}, x'_{new}, x'_{i'+j'+1}, \ldots, x'_{l})\\
\mathbf{x}'_{i', j', 2} & = (x'_{new}, x'_{i'+1}, \ldots, x'_{i'+j'})  \\
\mathbf{x}_{i, j, 1} & = (x_{1}, \ldots, x_{i}, x_{new}, x_{i+j+1}, \ldots, x_{k}) \\
\mathbf{x}_{i, j, 2} & = (x_{new}, x_{i+j}, \ldots, x_{i+1}).
\end{align}
The notation $\times_{0, -}$ means that the $0$-th puncture on the first disk component is attached to the $z_{-}$ puncture on the second disk component, where the interior marked point is located,
similarly for other products.
\end{enumerate}
\end{lem}

\subsection{Disks with one input and two outputs}\label{section: disks with two outputs}

In this subsection, we are going to introduce certain moduli spaces of disks
with one distinguished input, two distinguished outputs 
and arbitrary number of auxiliary inputs.
Similar moduli spaces are used in \cite{ganatra} to construct the non-compact Calabi-Yau structure \eqref{calabiyaustr} for the wrapped Fukaya category, 
which will be revisited in section \ref{section: inverse dualizing bimodule}.
More precisely, in \cite{ganatra} disks with two distinguished negative punctures and two distinguished positive punctures together with arbitrary other positive punctures are considered,
 for the purposing of constructing $\ainf$-bimodule morphisms.
For us, we shall replace the bimodule $\w^{!}$ \eqref{calabiyaustr} with a different but quasi-isomorphic bimodule,
so that the disks to be introduced below will have different numbers of boundary punctures.
Another difference is that we allow the main input and one of the output to carry weights as popsicles, 
such that the input turns to an output and vice versa when the weights are negative.
In particular, disks with negative weights are used to construct a different Calabi-Yau-type map in \S\ref{section: cy-}.
Moreover, the description of the boundary strata of these moduli spaces will be used to compare the two Calabi-Yau-type maps,
whose precise statement will be given in Proposition \ref{prop: comparing cy to cy-}.

Consider a Riemann surface $S_{2}^{1, l; r, s}$ isomorphic to a disk with $l+3 + r + s$ boundary punctures
\begin{equation}
z_{-, 1}, z^{+, 1}_{r}, \ldots, z^{+, 1}_{1}, z, z^{+, 2}_{1}, \ldots, z^{+, 2}_{s}, z_{-, 2}, z'_{l}, \ldots, z'_{1},
\end{equation}
ordered cyclically counterclockwise along the boundary of the disk.
Here we fix the cross-ratio of $z_{-, 1}, z, z_{-, 2}$,
e.g. place them at $i, -1, -i$ on the unit circle,
and allow other punctures to vary freely between these three special punctures.
The puncture $z$ is considered as an primary input, 
and $z_{-, 1}, z_{-, 2}$ as two primary outputs,
and the other punctures $z^{+, 1}_{r}, \ldots, z^{+, 1}_{1}, z^{+, 2}_{1}, \ldots, z^{+, 2}_{s}$ and $z'_{l}, \ldots, z'_{1}$ as auxiliary inputs.
The moduli space of such disks is denoted by
\begin{equation}
\mathcal{R}^{1, l; r, s}_{2}.
\end{equation}
See Figure \ref{fig:cy disk} for a picture of a disk in the moduli space $\mathcal{R}^{1, l; r, s}_{2}$.
The Deligne-Mumford compactification 
\begin{equation}
\bar{\mathcal{R}}^{1, l; r, s}_{2}
\end{equation}
is obtained by adding broken disks,
such that only one disk component has the same type, 
while all the other disk components are the usual disks with boundary punctures.
$\bar{\mathcal{R}}^{1, l; r, s}_{2}$ can be inductively defined as the space covered by products of moduli spaces of the following types:
\begin{align}
\label{boundary type 1} \bar{\mathcal{R}}^{r_{2}+s_{2}+2} & \times_{0} \bar{\mathcal{R}}^{1, l; r_{1}, s_{1}}_{2} \\
\label{boundary type 2} \bar{\mathcal{R}}^{j+1} & \times_{0, i}^{+, 1} \bar{\mathcal{R}}^{1, l; r - j + 1, s}_{2}, \\
\label{boundary type 3} \bar{\mathcal{R}}^{j'+1} & \times_{0, i'}^{+, 2} \bar{\mathcal{R}}^{1, l; r, s - j' + 1}_{2} \\
\label{boundary type 4} \bar{\mathcal{R}}^{j''+1} & \times_{0, i''} \bar{\mathcal{R}}^{1, l - j''+1; r, s}_{2} \\
\label{boundary type 5} \bar{\mathcal{R}}^{1, l_{1}; r_{1}, s}_{2} & \times_{-, 1, r_{1}+1} \bar{\mathcal{R}}^{r_{2}+l_{2}+2}  \\
\label{boundary type 6} \bar{\mathcal{R}}^{1, l_{1}; r, s_{1}}_{2} & \times_{-, 2, l_{1}+1} \bar{\mathcal{R}}^{s_{2}+l_{2}+2} .
\end{align}
Here the notation $\times_{0}$ means that the $0$-th puncture on the first disk component is matched with the $z$ puncture on the second component,
$\times_{0, i}^{+, 1}$ means that the $0$-th puncture on the first component is matched with the $z^{+, 1}_{i}$ puncture on the second component,
$\times_{0, i'}^{+, 2}$ means that the $0$-th puncture on the first component is matched with the $z^{+, 2}_{i'}$ puncture on the second component,
$\times_{0, i''}$ means that the $0$-th puncture on the first component is matched with the $z'_{i''}$ puncture on the second component,
$\times_{-, 1, r_{1}+1}$ means that the $z_{-, 1}$ puncture on the first component is matched with the $(r_{1}+1)$-th puncture on the second component,
and $\times_{-, 2, l_{1}+1}$ means that the $z_{-, 2}$ puncture on the first component is matched with the $(l_{1}+1)$-th puncture on the second component.

\begin{figure}
	\centering
	\def\svgwidth{\columnwidth}
	\resizebox{0.6\textwidth}{!}{\includegraphics{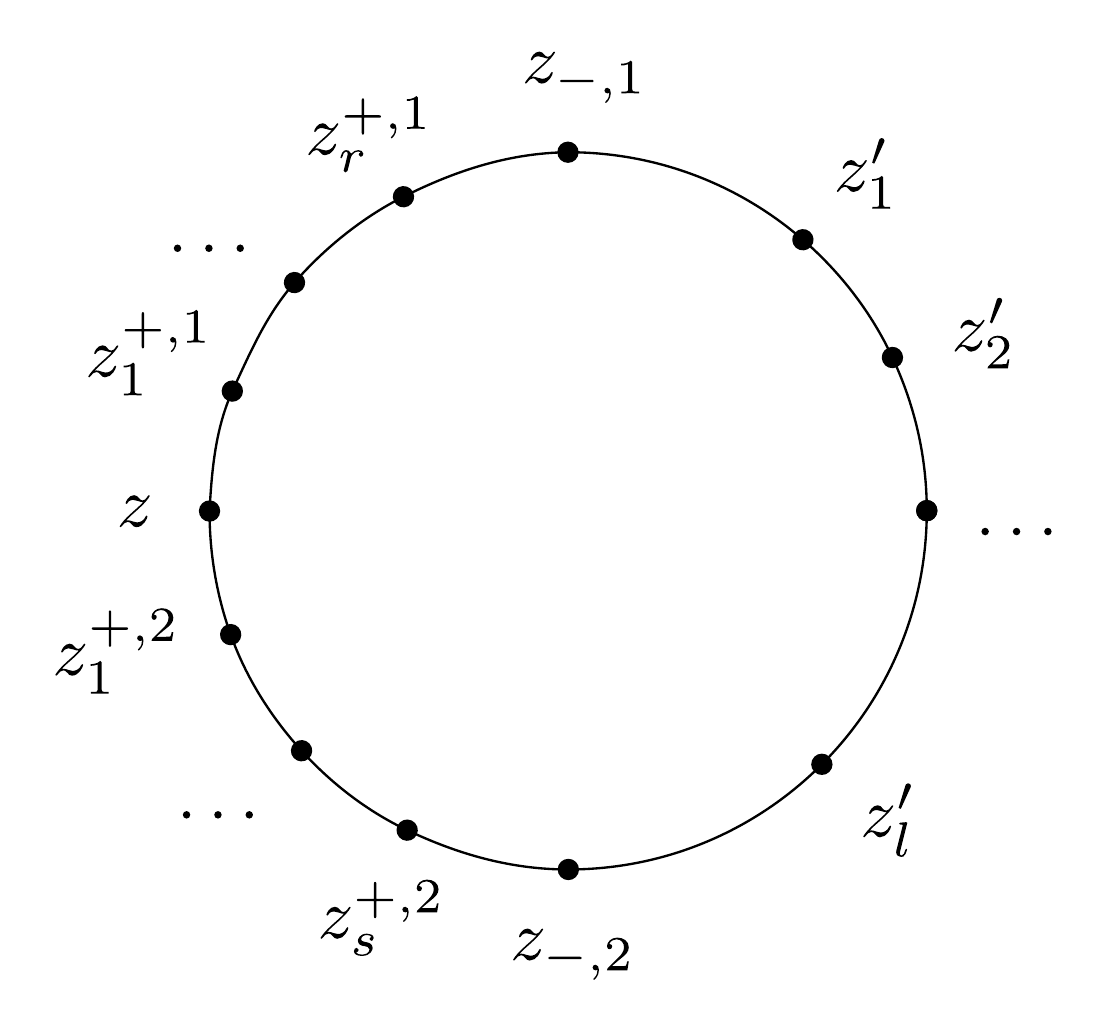}}
	\caption{An element of $\mathcal{R}^{1, l; r, s}_{2}$}
	\label{fig:cy disk}
\end{figure}

We can treat such disks as popsicles in a generalized sense with no sprinkle, i.e. $\mathbf{p} = \varnothing$ in the following way.
Assign the popsicle weights in the sense of \S \ref{subsec:weights} by
\begin{equation}
w^{+, 1}_{i} = w^{+, 2}_{j} = w_{-, 2} = w'_{h} = 0,
\end{equation}
at the punctures $z^{+, 1}_{i}, z^{+, 2}_{j}, z_{-, 2}$ and $z'_{h}$ (so that these are all positive punctures),
where $i = 1, \ldots, r, j = 1, \ldots, s, h = 1, \ldots, l$.
At $z_{-, 1}$ and $z$, we assign the weights $w_{-, 1}, w$ so that they are either both $0$ or both $-1$ (so that these are either both positive or both negative).
In either case, the condition \eqref{weight-sprinkle} implies that the number of sprinkles $|F| = 0$.
Although, strictly speaking, they are not popsicles because we have not defined popsicles structures on disks with two outputs.

Choose strip-like ends near the punctures in the following way:
near $z_{-, 1}$,
\[
\epsilon_{-, 1}: Z^{-1 - w_{-, 1}} \to S_{2}^{1, l; r, s},
\]
near $z$,
\[
\epsilon: Z^{w} \to S_{2}^{1, l; r, s},
\]
near $z^{+, 1}_{i}$,
\[
\epsilon^{+, 1}_{i}: Z \to S_{2}^{1, l; r, s}, i = 1, \ldots, r,
\]
near $z^{+, 2}_{j}$,
\[
\epsilon^{+, 2}_{j}: Z \to S_{2}^{1, l; r, s}, j = 1, \ldots, s,
\]
near $z_{-, 2}$,
\[
\epsilon_{-, 2}: Z^{-} \to S_{2}^{1, l; r, s},
\]
and near $z'_{h}$,
\[
\epsilon'_{h}: Z \to Z \to S_{2}^{1, l; r, s}, h = 1, \ldots, l.
\]
Equip these strip-like ends with rescaling factors $\nu_{-, 1}, \nu, \nu^{+, 1}_{i}, \nu^{+, 2}_{j}, \nu_{-, 2}, \nu'_{h}$ in the sense of \S \ref{subsec:floerdatum}, which are positive real numbers $\ge 1$,
which should satisfy
\begin{equation}\label{stokes for cy disks}
(-1)^{w} \nu + \sum_{i=1}^{r} \nu^{+, 1}_{i} + \sum_{j=1}^{s} + \sum_{h=1}^{l} \nu'_{h} \le (-1)^{w_{-, 1}} \nu_{-, 1} + \nu_{-, 2} 
\end{equation}

Define a Floer datum on $S_{2}^{1, l; r, s}$ in a way similar to Definition \ref{def: Floer datum for popsicles}.
Given the description of the boundary strata as in \eqref{boundary type 1} - \eqref{boundary type 6},
the notion of a universal and conformally consistent choice of Floer data can also be extended to $\bar{\mathcal{R}}^{1, l; r, s}_{2}$, where the conformal consistency condition involves choices previously made on the moduli spaces of boundary punctured disks in \S \ref{subsec:floerdatum}.

A Lagrangian label is a collection of Lagrangians, ordered according to the counterclockwise order on the circle as
\begin{equation}
K_{r}, \ldots, K_{0}, L_{0}, \ldots, L_{r}, L'_{0}, \ldots, L'_{l},
\end{equation}
each of which is assigned to a boundary component on $S^{1, l; r, s}$,
such that $K_{r}$ is assigned to the boundary component between $z_{-, 1}$ and $z^{+, 1}_{r}$.

Suppose we have made a universal and conformally consistent choice of Floer data for all $\bar{\mathcal{R}}^{1, l; r, s}$,
which are conformally consistent with choice of Floer data for moduli spaces of boundary-punctured disks $\bar{\mathcal{R}}^{d+1}$.
Then we can construct moduli spaces of inhomogeneous pseudoholomorphic maps from such domains $S_{2}^{1, l; r, s}$.
As noted before, there are two cases depending on the popsicle weights assigned to the punctures $z_{-, 1}$ and $z$ (or equivalently the signs associated to $z_{-,1}$ and $z$; recall that from \S \ref{subsec:weights} that weight $-1$ is assigned to negative punctures and weight $0$ is assigned to positive punctures).

The first case is 
\begin{equation}
w_{-, 1} = w = 0.
\end{equation}
In this case, $z$ carries a positive strip-like end and $z_{-, 1}$ carries a negative strip-like end.
The moduli space of inhomogeneous pseuholomorphic maps with respect to such Floer data is denoted by
\begin{equation}\label{moduli space of cy disks with marked points}
\mathcal{R}_{2; (0, 0)}^{1, l; r, s}(\mathbf{x}^{+, 1}, x, \mathbf{x}^{+, 2}, x_{-, 2}, \mathbf{x}', x_{-, 1})
\end{equation}
where $x$ is a time-one $H$-chord from $K_{0}$ to $L_{0}$,
and $x_{-, 1}$ is a time-one $H$-chord from $K_{r}$ to $L'_{l}$.

The second case is 
\begin{equation}
w_{-, 1} = w = -1.
\end{equation}
Now $z$ carries a negative strip-like end $z_{-, 1}$ carries a positive strip-like end.
The moduli space of inhomogeneous pseuholomorphic maps with respect to such Floer data is denoted by
\begin{equation}\label{moduli space of cy- disks with marked points}
\mathcal{R}_{2; (-1, -1)}^{1, l; r, s}(\mathbf{x}^{+, 1}, x, \mathbf{x}^{+, 2}, x_{-, 2}, \mathbf{x}', x_{-, 1})
\end{equation}
where $x$ now represents a time-one $H$-chord from $L_{0}$ to $K_{0}$,
and $x_{-, 1}$ a time-one $H$-chord from $L'_{l}$ to $K_{r}$.

The basic ones, where $r = s = 0$, are
\begin{equation}\label{moduli space of cy disks}
\mathcal{R}_{2; (0, 0)}^{1, l}(x, x_{-, 2}, \mathbf{x}', x_{-, 1})
\end{equation}
and
\begin{equation}\label{moduli space of cy- disks}
\mathcal{R}_{2; (-1, -1)}^{1, l}(x, x_{-, 2}, \mathbf{x}', x_{-, 1}).
\end{equation}

By a standard index calculation, we have

\begin{lem}
The virtual dimensions of \eqref{moduli space of cy disks with marked points} and  \eqref{moduli space of cy- disks with marked points}
are given by the common formula
\begin{equation}\label{dimension of moduli space of cy and cy- disks}
\begin{split}
 &r + s + l -n + (-1)^{-1 - w_{-, 1}} \deg(x_{-, 1}) + \deg(x_{-, 2}) \\
 &- (-1)^{w} \deg(x) - \sum_{i=1}^{r} \deg(x^{+, 1}_{i}) - \sum_{j=1}^{s} \deg(x^{+, 2}_{j}) - \sum_{h=1}^{l} \deg(x'_{h}).
\end{split}
\end{equation}
\qed
\end{lem}

The Gromov compactifications
\begin{equation}
\bar{\mathcal{R}}_{2; (0, 0)}^{1, l; r, s}(\mathbf{x}^{+, 1}, x, \mathbf{x}^{+, 2}, x_{-, 2}, \mathbf{x}', x_{-, 1})
\end{equation}
and
\begin{equation}
\bar{\mathcal{R}}_{2; (-1, -1)}^{1, l; r, s}(\mathbf{x}^{+, 1}, x, \mathbf{x}^{+, 2}, x_{-, 2}, \mathbf{x}', x_{-, 1})
\end{equation}
have boundary strata consisting of product moduli spaces of maps from broken domains listed in \eqref{boundary type 1} - \eqref{boundary type 6}.
In addition, there are extra boundary strata coming from strip breaking.
In particular, when the virtual dimension equals one, 
these boundary strata are given by product moduli spaces of maps whose domains are smooth.
In the second case where $w_{-, 1}=w=-1$, there are some boundary strata of $\bar{\mathcal{R}}_{2; (-1, -1)}^{1, l; r, s}(\mathbf{x}^{+, 1}, x, \mathbf{x}^{+, 2}, x_{-, 2}, \mathbf{x}', x_{-, 1})$
consisting of moduli spaces of broken popsicles with special weight conditions, to be described below.

Consider a boundary-punctured Riemann surface isomorphic to a disk with $k+l+2$ boundary punctures.
Choose a Lagrangian label
\[
K_{k}, \ldots, K_{0}, L_{0}, \ldots, L_{l},
\]
ordered counterclockwise along the boundary of a disk,
where we think of the Lagrangian $K_{k}$ as being assigned to the boundary component between the $0$-th and the first punctures.
Let $x_{i}$ be a time-one $H$-chord from $K_{i}$ to $K_{i-1}$,
$x'_{j}$ a time-one $H$-chord from $L_{j-1}$ to $L_{j}$,
 $y$ a time-one $H$-chord from $L_{0}$ to $K_{0}$,
 $z$ a time-one $H$-chord from $L_{l}$ to $K_{k}$.
Consider the moduli space of popsicles
\begin{equation}\label{w- disks}
\mathcal{R}^{k+l+2, \varnothing, \mathbf{w}_{-, k, l}}(z, \mathbf{x}, y, \mathbf{x}'),
\end{equation}
where the popsicle structure has no sprinkles, i.e. the empty map $\varnothing$,
and the weights are
\begin{equation}
\mathbf{w}_{-, k, l} = (-1, \underbrace{0, \ldots, 0}_{k \text{ times }}, -1, \underbrace{0, \ldots, 0}_{l \text{ times }}).
\end{equation}
This moduli space has virtual dimension
\begin{equation}
\dim \mathcal{R}^{k+l+2, \varnothing, \mathbf{w}_{-, k, l}}(z, \mathbf{x}, y, \mathbf{x}') = \deg(y) - \deg(z) - \sum_{i=1}^{k} \deg(x_{i}) - \sum_{j=1}^{l} \deg(x'_{j}) + k + l - 1.
\end{equation}

Now we can give a description of the boundary strata of $\bar{\mathcal{R}}_{2; (-1, -1)}^{1, l; r, s}(\mathbf{x}^{+, 1}, x, \mathbf{x}^{+, 2}, x_{-, 2}, \mathbf{x}', x_{-, 1})$,
while leaving the description of the boundary strata of $\bar{\mathcal{R}}_{2; (0, 0)}^{1, l; r, s}(\mathbf{x}^{+, 1}, x, \mathbf{x}^{+, 2}, x_{-, 2}, \mathbf{x}', x_{-, 1})$ as an exercise.
The special case $r=0$ plays an important role in the proof of Proposition \ref{prop: comparing cy to cy-}.

\begin{lem}\label{boundary of moduli space of cy- disks}
Suppose the virtual dimension of the moduli space
\[
\mathcal{R}_{2; (-1, -1)}^{1, l; r, s}(\mathbf{x}^{+, 1}, x, \mathbf{x}^{+, 2}, x_{-, 2}, \mathbf{x}', x_{-, 1})
\]
is equal to one, 
i.e. when \eqref{dimension of moduli space of cy and cy- disks} is equal to one.
Then there exists a universal and conformally consistent choice of Floer data compactible with the chosen Floer data for all $\bar{\mathcal{R}}^{d+1, \mathbf{p}, \mathbf{w}}$, such that the Gromov compactification
\[
\bar{\mathcal{R}}_{2; (-1, -1)}^{1, l; r, s}(\mathbf{x}^{+, 1}, x, \mathbf{x}^{+, 2}, x_{-, 2}, \mathbf{x}', x_{-, 1})
\]
is a compact smooth one-dimensional manifold with boundary, 
and its boundary strata consist of product moduli spaces of the following kinds
\begin{equation}\label{boundary strata of moduli space of cy- disks}
\begin{split}
& \partial \bar{\mathcal{R}}_{2; (-1, -1)}^{1, l; r, s}(\mathbf{x}^{+, 1}, x, \mathbf{x}^{+, 2}, x_{-, 2}, \mathbf{x}', x_{-, 1}) \\
\cong & \coprod \mathcal{R}_{2; (-1, -1)}^{1, l - l' + 1; r, s}(\mathbf{x}^{+, 1}, x, \mathbf{x}^{+, 2}, x_{-, 2}, \mathbf{x}'_{h, l', I}, x_{-, 1}) 
 \times \mathcal{R}^{l'+1}(\mathbf{x}'_{h, l', II}) \\
\cup & \coprod \mathcal{R}_{2; (-1, -1)}^{1, l; r, s - s' +1}(\mathbf{x}^{+, 1}, x, \mathbf{x}^{+, 2}_{j, s, I}, x_{-, 2}, \mathbf{x}', x_{-, 1}) 
 \times \mathcal{R}^{s'+1}(\mathbf{x}^{+, 2}_{j, s, II}) \\
\cup & \coprod \mathcal{R}_{2; (-1, -1)}^{1, l; r - r' + 1, s}(\mathbf{x}^{+, 1}_{i, r', I}, x, \mathbf{x}^{+, 2}, x_{-, 2}, \mathbf{x}', x_{-, 1}) 
 \times \mathcal{R}^{r'+1}(\mathbf{x}^{+, 1}_{i, r', II}) \\
\cup & \coprod \mathcal{R}_{2; (-1, -1)}^{1, l_{1}; r_{1}, s}(\mathbf{x}^{+, 1}_{r_{1}, II}, x, \mathbf{x}^{+, 2}, x_{-, 2}, \mathbf{x}'_{l_{1}, II}, \tilde{x}_{-, 1}) 
 \times \mathcal{R}^{r_{2} + l_{2} + 2, \varnothing, \mathbf{w}_{-}}(x_{-, 1}, \mathbf{x}^{+, 1}_{r_{2}, II}, \tilde{x}_{-, 1}, \mathbf{x}'_{l_{2}, I}) \\
 \cup & \coprod \mathcal{R}^{1, l_{1}; r; s}_{2;(-1,-1)}(\mathbf{x}^{+, 1}, x, \mathbf{x}^{+, 2}, \tilde{x}_{-,2}, \mathbf{x}'_{l_{1}, I}, x_{-,1}) \times \mathcal{R}^{l_{2}+2}(x_{-, 2}, \mathbf{x}'_{l_{1}, II}, \tilde{x}_{-,2}) \\
\cup & \coprod \mathcal{R}^{r+l_{1}+s_{1}+2, \varnothing, \mathbf{w}_{-, r, s_{1}+l_{1}+1}}(x_{-, 1}, \mathbf{x}^{+, 1}, x, \mathbf{x}^{+, 2}_{s_{1}, I}, \tilde{x}_{-,2}, \mathbf{x}'_{l_{1}, I}) \times \mathcal{R}_{2; (0, 0)}^{1, l_{2}; s_{2}-1, 0}(\mathbf{x}^{+, 2}_{s_{1}, II}, x_{-, 2}, \mathbf{x}'_{l_{1}, II},  \tilde{x}_{-, 2}) \\
\cup & \coprod  \mathcal{R}^{r_{2} + s_{2} + 2, \varnothing, \mathbf{w}_{-, r_{2}, s_{2}}}(\tilde{x}, \mathbf{x}^{+, 1}_{r_{2}, I}, x, \mathbf{x}^{+, 2}_{s_{2}, I})
\times \mathcal{R}_{2; (-1, -1)}^{1, l; r_{1}, s_{1}}(\mathbf{x}^{+, 1}_{r_{1}, II} ,\tilde{x},  \mathbf{x}^{+, 2}_{s_{1}, II}, x_{-, 2}  , \mathbf{x}',  x_{-,1}).
\end{split}
\end{equation}
where
\begin{align}
\mathbf{x}^{+, 1}_{i, r', I} & = (x^{+, 1}_{r}, \ldots, x^{+, 1}_{i+r'+1}, x^{+, 1}_{new}, x^{+, 1}_{i}, \ldots, x^{+, 1}_{1}), \\
\mathbf{x}^{+, 1}_{i, r', II} & = (x^{+, 1}_{new}, x^{+, 1}_{i+1}, \ldots, x^{+, 1}_{i+r'}), \\
\mathbf{x}^{+, 2}_{j, s, I} & = (x^{+, 2}_{1}, \ldots, x^{+, 2}_{j}, x^{+, 2}_{new}, x^{+, 2}_{j+s'+1}, \ldots, x^{+, 2}_{s}), \\
\mathbf{x}^{+, 2}_{j, s, II} & = (x'_{new}, x'_{j+s'}, \ldots, x'_{j+1}), \\
\mathbf{x}'_{h, l', I} & = (x'_{l}, \ldots, x'_{h+l'+1}, x'_{new}, x'_{h}, \ldots, x'_{1}), \\
\mathbf{x}'_{h, l', II} & = (x'_{new}, x'_{h+l'}, \ldots, x'_{h+1}), \\
\mathbf{x}^{+, 1}_{i, I} & = (x^{+, 1}_{i}, \ldots, x^{+, 1}_{1}), \\
\mathbf{x}^{+, 1}_{i, II} & = (x^{+, 1}_{r}, \ldots, x^{+, 1}_{{i}+1}), \\
\mathbf{x}^{+, 2}_{j, I} & = (x^{+, 2}_{1}, \ldots, x^{+, 2}_{j}), \\
\mathbf{x}^{+, 2}_{j, II} & = (x^{+, 2}_{j+1}, \ldots, x^{+, 2}_{s}), \\
\mathbf{x}'_{h, I} & = (x'_{h}, \ldots, x'_{1}), \\
\mathbf{x}'_{h, II} & = (x'_{l}, \ldots, x'_{h+1}).
\end{align}
\end{lem}
\begin{proof}
Smoothness follows by a standard transversality and gluing argument, similar to Proposition \ref{master prop for smoothness and compactness}.

The description of the boundary strata is a consequence of domain degenerations listed in \eqref{boundary type 1} - \eqref{boundary type 6} as well as strip breaking,
with further considerations on choices of Floer data compatible with strip-like ends and gluing.
The possible limit broken popsicles can be classified depending on the positions of the three special punctures $z, z_{-, 1}, z_{-, 2}$ on the nodal disk:
\begin{enumerate}[label=(\roman*)]

\item The three special punctures are all on the same disk component.
There are indeed three such possible configurations, 
which are listed in the first three types on the right hand side of \eqref{boundary strata of moduli space of cy- disks}.

\item $z_{-, 1}$ alone is on one disk component, while $z_{-, 2}$ and $z$ are on the other disk component. 
The limit broken popsicles of such kind belong to the product moduli spaces of the fourth type.

\item $z_{-, 2}$ alone is on one disk component, while $z$ and $z_{-, 1}$ are on the other disk component.
The limit broken popsicles of such kind belong to the product moduli spaces of either the fifth or sixth type.
These two sub-cases depend on the position of the last puncture $z^{+,2}_{s}$ (when $s>0$) in the collection $\mathbf{z}^{+, 2}$,
in a way such that the strip-like ends for the node are determined by the relative position of $z^{+, 2}_{s}$ to the node subject to the rule that the negative strip-end should be on the same disk component as $z^{+, 2}_{s}$ (this is possible by one particular choice of Floer data).
These two sub-cases are therefore:
\begin{enumerate}

\item $s_{1} = s$. 
In this case, the node is an output for the top disk component,
so that the top disk component is still of type $\mathcal{R}^{1, l_{1}; r, s}_{2; (-1,-1)}$, 
while the bottom disk component is an ordinary $\ainf$-disk of type $\mathcal{R}^{l_{2}+2}$.

\item $s_{1} < s$.
In this case, the node is an output for the bottom disk component.
Now the top disk component has one distinguished input $z$ and output $z_{-, 1}$,
while the bottom disk component has one distinguished input $z^{+, 2}_{s}$ and two distinguished outputs ($z_{-,2}$ and the node) both carrying weight $0$,
and is therefore of type $\mathcal{R}_{2; (0, 0)}^{1, l_{2}; s_{2}-1, 0}$.

\end{enumerate}

\item $z$ alone is on one disk component, while $z_{-, 1}$ and $z_{-, 2}$ are on the other disk component.
The limit broken popsicles of such kind belong to the product moduli spaces of the last type.

\end{enumerate}
See Figure \ref{fig:boundary of moduli of cy disks} for pictures of representatives of elements in each of these strata in \eqref{boundary strata of moduli space of cy- disks}.
\end{proof}

\begin{figure}
	\centering
	\def\svgwidth{\columnwidth}
	\resizebox{1\textwidth}{!}{\includegraphics{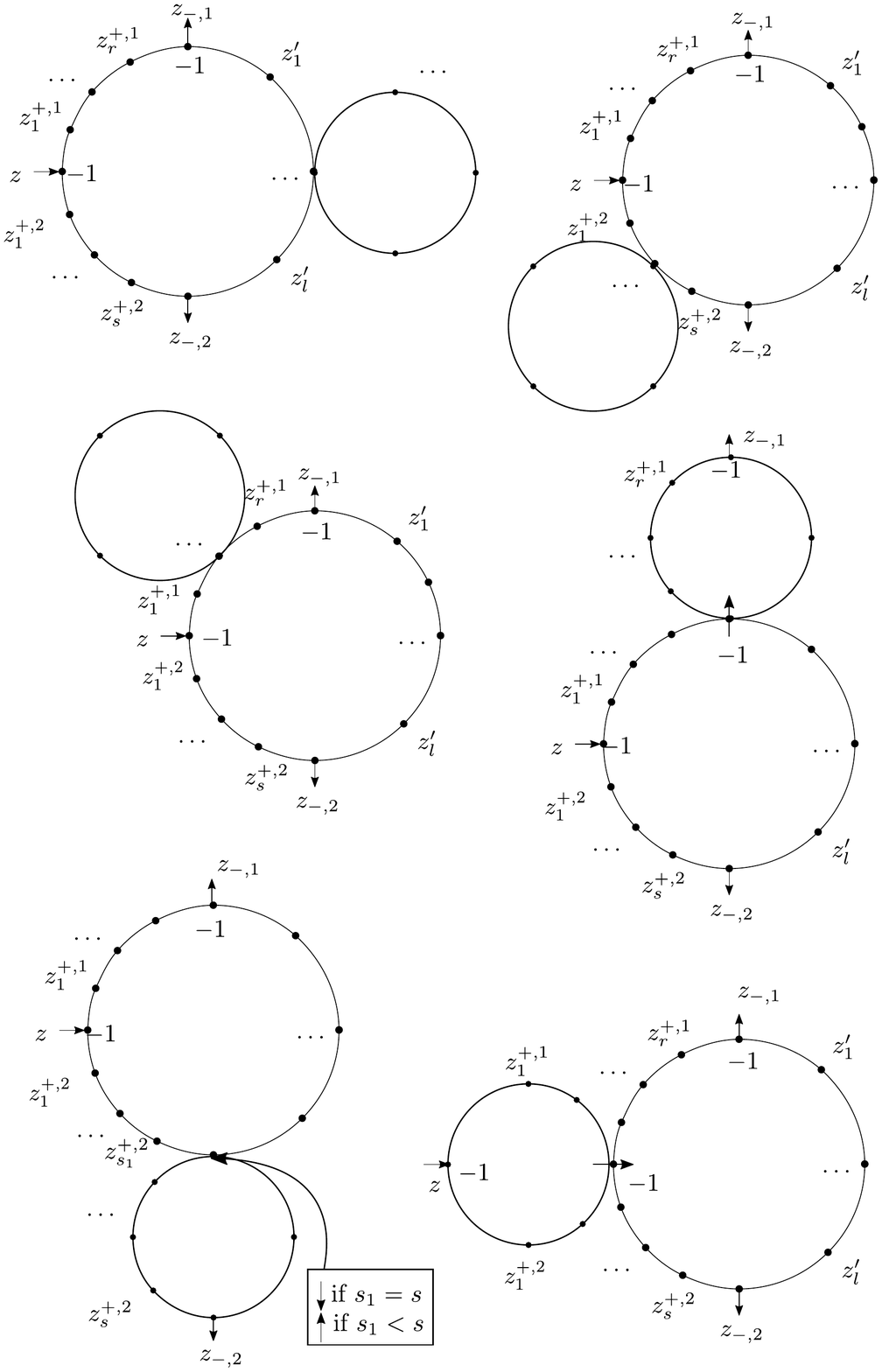}}
	\caption{Broken popsicles in the boundary strata of $\bar{\mathcal{R}}_{2; (-1, -1)}^{1, l; r, s}(\mathbf{x}^{+, 1}, x, \mathbf{x}^{+, 2}, x_{-, 2}, \mathbf{x}', x_{-, 1})$: the arrows indicate the distinguished input and outputs, some of which carry weight $-1$; all the other inputs carry weight $0$}
	\label{fig:boundary of moduli of cy disks}
\end{figure}

\subsection{A degeneration argument}\label{section: broken popsicles of a different type}

The purpose of this subsection is to provide the gluing argument establishing a relationship between two types of moduli spaces of broken disks appearing as boundary of of a third moduli space,
where on one side disks with two outputs are glued to disks with zero output,
and on the other two disks with one input and output are glued.
The gluing argument leaves the framework and specific constraints of popsicles (for which \eqref{maximum number of sprinkles} fails), but in a relatively mild way.
This type of argument is needed in \S\ref{section: homotopy argument} for comparing the composition of a Calabi-Yau type map with two outputs, one of which is negative (defined in \S\ref{section: cy-}) 
along with a Poincar\'e-duality-type pairing map (defined in \S\ref{section: Poincare duality} using the disks introduced previously in \S\ref{section: pairing disks}),
to a functor from the Rabinowitz Fukaya category (Definition \ref{def: Rabinowitz Fukaya category}), 
which is a Yoneda-type map built out of a subcollection of the moduli spaces of popsicle maps defined in \S\ref{section:moduli space of popsicle maps}.

We will encounter a new phenomenon when we try to glue disks with two inputs (introduced in \S\ref{section: pairing disks}) and disks with two outputs (introduced in \S\ref{section: disks with two outputs}).
Consider the situation where we want to glue an inhomogeneous pseudoholomorphic disk $u: S_{p}^{k, l} \to X$ with only inputs introduced in subsection \ref{section: pairing disks},
with an inhomogeneous pseudolomorphic disk $v: S_{2}^{1, l; r, s} \to X$ with two outputs introduced in subsection \ref{section: disks with two outputs},
in a way such that the output of $v$ at the puncture $z_{-, 1}$ is glued to the input of $u$ at the puncture $z_{-}$.
The outcome is a one-parameter family of certain smooth disks with exactly one output.

For our purpose, we shall consider a specific situation with certain number of punctures on the boundary of each of the two disks.
To be more specific, consider the product moduli space
\begin{equation}\label{moduli space of broken popsicles of a different type}
\mathcal{R}^{1, l_{1}}_{2; (-1, -1)}(x, x_{-, 2}, \mathbf{x}'_{II}, x_{-, 1}) \times \mathcal{R}^{l_{2}, 0}_{p}(\mathbf{x}_{I}; x_{-}, x_{+})
\end{equation}
such that $l_{1} + l_{2} = l$, and $x_{-, 1} = x_{-}$,
and the chords $\mathbf{x}'_{I}, \mathbf{x}_{II}$ are such that
\begin{equation}
(\mathbf{x}'_{II}, \mathbf{x}_{I}) = \mathbf{x}' = (x'_{l}, \ldots, x'_{1})
\end{equation}
Here we have changed the notation for the number of punctures in the moduli space $\mathcal{R}^{l_{2}, 0}_{p}(\mathbf{x}'_{2}; x_{-}, x_{+})$, by letting $k = l_{2}$ and $l = 0$,
following the notation of \eqref{moduli space of pairing disks with marked points}.
Correspondingly, we call those Lagrangians $K_{i}$'s now by $L'_{i}$'s.

If we glue a broken disk in this moduli space at the identified punctures $z_{-, 1} \sim z_{-}$, 
we obtain a one-parameter family of maps from smooth domains.
Each such domain is a Riemann surface $S^{1, l}_{1}$ isomorphic to a disk with $l+3$ boundary punctures
\[
z_{out}, z'_{l}, \ldots, z'_{1}, z_{+}, z,
\]
ordered counterclockwise along the boundary of the disk,
as well as an interior marked point $\zeta$ as in the case of $S^{k, l}_{p}$ in \S \ref{section: pairing disks}.
We regard $z_{out}$ as a special positive output,
$z_{+}$ as a special positive input, $z$ as a special negative input, 
and $z'_{l}, \ldots, z'_{1}$ as auxiliary positive inputs.
Up to automorphism, the three points $z_{+}, z, z_{out}$ can be placed at $i, -1, -i$ respectively.
The other points $z'_{l}, \ldots, z'_{l}$ can vary freely between $z_{+}$ and $z_{out}$.
The moduli space of such domains is denoted by 
\begin{equation}
\tilde{\mathcal{R}}_{1}^{l+3},
\end{equation}
inside which there is a codimension-one subspace cut out by constraining the interior marked point $\zeta$ to lie on a chosen geodesic connecting $z_{+}$ to $z_{out}$; 
denote that subspace by
\begin{equation}
\mathcal{R}_{1}^{l+3}.
\end{equation}
There is a natural compactification
\begin{equation}
\bar{\mathcal{R}}_{1}^{l+3}
\end{equation}
such that the boundary strata are given by product moduli spaces of various types, including $\bar{\mathcal{R}}^{d+1, \mathbf{p}, \mathbf{w}}, \bar{\mathcal{R}}_{p}^{l', 0}$ and $\bar{\mathcal{R}}^{1, l'}_{2;(-1,-1)}$ for appropriate indices $d, l'$.

We can define a Floer datum for the domain $S^{1, l}_{1}$ in a way similar to Definition \ref{def: Floer datum for popsicles},
with minor modifications as such a domain (even as a representative of an element in the subspace $\mathcal{R}_{1}^{l+3}$) is strictly speaking not a weighted popsicle (because \eqref{maximum number of sprinkles} fails).
The only matter that we need to take care of here is the rule for assigning strip-like ends, which should obey the following rule:
\begin{itemize}

\item positive strip-like ends $\e: Z^{+} \to S^{1, l}_{1}$ near positive inputs, which are $z'_{l}, \ldots, z'_{1}, z_{+}$;

\item negative strip-like ends $\e: Z^{-} \to S^{1, l}_{1}$ near the negative input $z$ (which can be thought of as a positive output) and the positive output $z_{out}$.

\end{itemize}
With this rule of choosing strip-like ends, the rest of Definition \ref{def: Floer datum for popsicles} applies.
Moreover, the definition of a universal and conformally consistent choice of Floer data has a straightforward generalization to such moduli spaces $\mathcal{R}_{1}^{l+3}$, compatible with choices for other types of moduli spaces with respect to gluing,
since the boundary strata of the compactification $\bar{\mathcal{R}}_{1}^{l+3}$ involve familiar moduli spaces $\bar{\mathcal{R}}^{d+1, \mathbf{p}, \mathbf{w}}, \bar{\mathcal{R}}_{p}^{l', 0}$ and $\bar{\mathcal{R}}^{1, l'}_{2;(-1,-1)}$ on which such choices already exist.
Call this simply a universal and conformally consistent choice of Floer data for all $\bar{\mathcal{R}}_{1}^{l+3}$.

Choose a Lagrangian label that is consistent with the gluing of the two Lagrangian labels for $\mathcal{R}^{1, l_{1}}_{2; (-1, -1)}(x, x_{-, 2}, \mathbf{x}'_{II}, x_{-, 1})$ and $\mathcal{R}^{l_{2}, 0}_{p}(\mathbf{x}_{I}; x_{-}, x_{+})$.
Consider smooth maps
\[
u: S^{1, l}_{1} \to X
\]
that satisfy the inhomogeneous Cauchy-Riemann equation with respect to the chosen Floer data,
with asymptotic conditions given by time-one chords $x_{out}$ at $z_{out}$, $x$ at $z$, $x'_{j}$ at $z'_{j}$ and $x_{+}$ at $z_{+}$.
The resulting moduli space is denoted by
\begin{equation}
\tilde{\mathcal{R}}_{1}^{l+3}(x_{out}, \mathbf{x}', x_{+}, x),
\end{equation}
Again, inside of this space there is a subspace of codimension one,
by constraining the interior marked point $\zeta$ on the geodesic connecting $z_{+}$ to $z_{out}$:
\begin{equation}\label{moduli space of glued broken popsicles with an interior marked point}
\mathcal{R}_{1}^{l+3}(x_{out}, \mathbf{x}', x_{+}, x).
\end{equation}
Now we suppose we have made a universal and conformally consistent choice of Floer data for all $\bar{\mathcal{R}}_{1}^{l+3}$.
Each moduli space admits a natural Gromov compactification 
\begin{equation}\label{compactified moduli space of popsicles with an interior marked point}
\bar{\mathcal{R}}_{1}^{l+3}(x_{out}, \mathbf{x}', x_{+}, x)
\end{equation}
whose codimension-one boundary strata contain not only \eqref{moduli space of broken popsicles of a different type},
but also some other types of moduli spaces of broken popsicles.
By standard transversality argument, Gromov compactness and maximum principle (similar to the proof of Proposition \ref{master prop for smoothness and compactness}, we can show that the zero-dimensional moduli spaces $\mathcal{R}_{1}^{l+3}(x_{out}, \mathbf{x}', x_{+}, x)$ are compact smooth manifolds, and therefore consist of finitely many points.
For one-dimensional moduli spaces, we have:

\begin{lem}\label{cobordism between moduli spaces of broken popsicles}	
	There exists a universal and conformally consistent choice of Floer data
such that when the virtual dimension is one, $\bar{\mathcal{R}}_{1}^{l+3}(x_{out}, \mathbf{x}', x_{+}, x)$ is a compact smooth one-dimensional manifold with boundary,
whose codimension-one boundary strata consist of the following union of product moduli spaces:
\begin{equation}\label{boundary of one-pointed popsicles defining homotopy}
\begin{split}
& \p \bar{\mathcal{R}}_{1}^{l+3}(x_{out}, \mathbf{x}', x_{+}, x) \\
\cong & \coprod  \mathcal{R}^{2, \varnothing, (-1, -1)}(\tilde{x}, x) 
\times \mathcal{R}_{1}^{l+3}(x_{out}, \mathbf{x}', x_{+}, \tilde{x}) \\
\cup & \coprod \mathcal{R}^{l'+1}(\mathbf{x}'_{h, l', II}) 
\times \mathcal{R}_{1}^{l - l' + 4}(x_{out}, \mathbf{x}'_{h, l', I}, x_{+}, x) \\
\cup &  \coprod  \mathcal{R}_{1}^{l_{1}+3}(\tilde{x}_{out}, \mathbf{x}'_{l_{1}, I}, x_{+}, x) 
 \times \mathcal{R}^{l_{2}+2}(x_{out}, \mathbf{x}'_{l_{1}, II}, \tilde{x}_{out})\\
 \cup & \coprod \mathcal{R}^{1, l_{1}}_{2; (-1, -1)}(x, x_{out}, \mathbf{x}'_{II}, x_{-, 1}) 
\times \mathcal{R}^{l_{2}, 0}_{p}(\mathbf{x}_{I}; x_{-, 1}, x_{+})\\
 \cup & \coprod \mathcal{R}^{l_{1}+3, \mathbf{p}_{I}, \mathbf{w}_{I}}(y, \mathbf{x}'_{l_{1}, I}, x_{+}, x) \times \mathcal{R}^{l_{2}+ 2, \mathbf{p}_{II}, \mathbf{w}_{II}}(x_{out}, \mathbf{x}'_{l_{1}, II}, y)\\
 \cup & \coprod \mathcal{R}^{l_{2}+2}(\tilde{x}_{+}, \mathbf{x}'_{l_{2}, I}, x_{+}) 
 \times \mathcal{R}_{1}^{l_{1}+3}(x_{out}, \mathbf{x}'_{l_{2}, II}, \tilde{x}_{+}, x) 
\end{split}
\end{equation}
where the disjoint unions are taken over all possible indices and intermediate chords,
such that all the relevant moduli spaces on the right hand side are zero-dimensional.
In the last line, the moduli spaces are moduli spaces of popsicles, 
where the popsicle structures $\sigma_{I}, \sigma_{II}$ of flavors $\mathbf{p}_{I}, \mathbf{p}_{II}$ and weights $\mathbf{w}_{I}, \mathbf{w}_{II}$ are described as follows.
For the weights, we have
\begin{align}
\mathbf{w}_{I} & = (-1, \underbrace{0, \ldots, 0}_{l_{1} +1 \text{ times }}, -1), \\
\mathbf{w}_{II} & = (0, \underbrace{0, \ldots, 0}_{l_{2} \text{ times}}, -1).
\end{align}
Then $\sigma_{I}$ is a trivial popsicle structure,
and $\sigma_{II}$ is the choice of one sprinkle on the geodesic connecting the $(l_{2}+1)$-th puncture to the $0$-th puncture.
\end{lem}
\begin{proof}
	As usual, the compactification is a result of domain degenerations as well as strip breaking.
\begin{enumerate}[label=(\roman*)]

\item The first case is a strip breaking at the puncture $z$;
the resulting broken disks belong to the first type on the right hand side of \eqref{boundary of one-pointed popsicles defining homotopy}.

\item If strip breaking occurs at a puncture $z'_{j}, j=1,\ldots, l$, or domain degeneration occurs along some boundary component between the $z'_{l}, \ldots, z'_{1}$ punctures,
the resulting broken disks represent elements in the second kind of product moduli space in \eqref{boundary of one-pointed popsicles defining homotopy}.
The top right broken disk in Figure \ref{fig:boundary of moduli of homotopy disks} provides a schematic picture for both cases.

\item The remaining cases occur due to domain degenerations such that the input and output $z_{+}, z_{out}$ are located on different disk components, 
and are classified by different behaviors of the interior marked point $\zeta$.
Suppose $\zeta$ goes to the first component where $z_{+}$ is located.
There are two different sub-cases, further depending on the position of the special negative input $z$,
in a way that the node should always be an output for the disk component on which $z$ lies.
The weight assigned to the node is uniquely determined by the rule, 
which is $0$ if $z$ is on the top disk component (containing $\zeta$), and $-1$ otherwise.
Such broken disks represent elements in the third (where the weight at the node $=0$) and the fourth (where the weight at the node $=-1$) types on the right hand side of \eqref{boundary of one-pointed popsicles defining homotopy}.
The pictures for these two cases are the middle two broken disks in Figure \ref{fig:boundary of moduli of homotopy disks}.

\item Suppose $\zeta$ goes to the second component where $z_{out}$ is located.
Again, there are two sub-cases depending on the position of the special negative input $z$,
in a way that the node is an input for the disk component containing $\zeta$.
The weight assigned to the node is determined by the same rule as in (iii).
Such broken disks represent elements in the fifth (where the weight at node $=0$) and sixth (where the weight at node $=-1$) types on the right hand side of \eqref{boundary of one-pointed popsicles defining homotopy}.
See the last two broken disks in Figure \ref{fig:boundary of moduli of homotopy disks}.

\end{enumerate}

\end{proof}

\begin{figure}
	\centering
	\def\svgwidth{\columnwidth}
	\resizebox{1\textwidth}{!}{\includegraphics{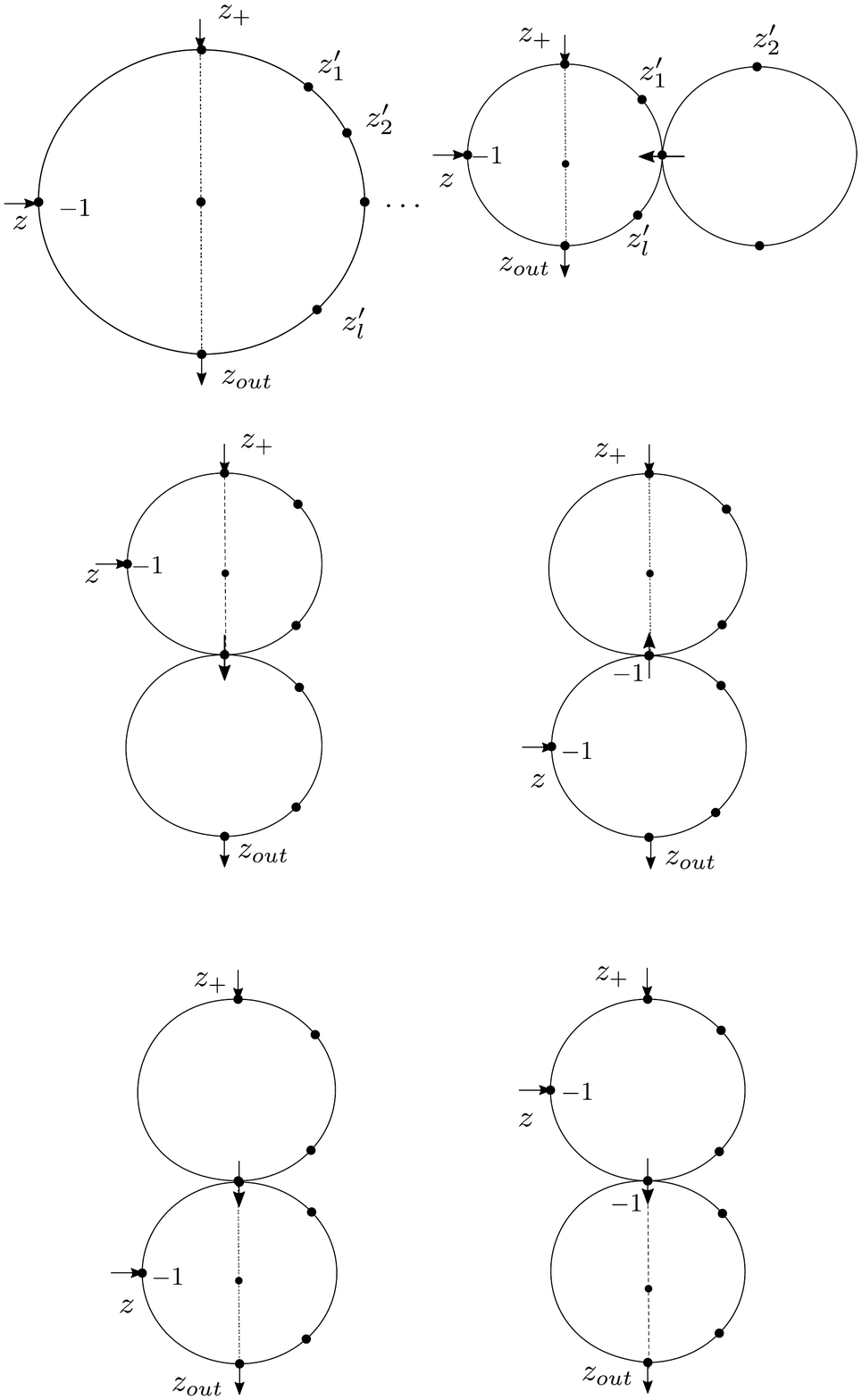}}
	\caption{A representative in $\mathcal{R}_{1}^{l+3}$, and broken disks in the boundary strata of $\bar{\mathcal{R}}_{1}^{l+3}$; the arrows indicate input vs. output}
	\label{fig:boundary of moduli of homotopy disks}
\end{figure}

\section{The Rabinowitz wrapped Fukaya category}\label{sec:rabinowitzwrappedFukaya}

The goal of this section is to define the Rabinowitz Fukaya category, reviewing and fixing conventions for (wrapped) Lagrangian Floer theory along the way.
First we recall the definition of the wrapped Floer complex using a quadratic Hamiltonian, and the $A_{\infty}$-structure it possesses,
following the framework of \cite{abouzaid1, ganatra}.
Then we introduce the Rabinowitz Floer complex and construct on it an $\ainf$-structure.
Our construction of these $\ainf$-operations uses the popsicle-moduli spaces of \S \ref{sec:modulispaces} and particularly \S \ref{section:moduli space of popsicle maps} as a means of geometrically defining operations on mapping cones, following \cite{abouzaidseidel, seidel6} adapted here to the setting of wrapped Floer theory. 

\subsection{The wrapped Fukaya category}\label{section:wrappedfukayacategory}

Fix an admissible Hamiltonian $H \in \mathcal{H}(M)$ and define
\[
\chi(L_{0}, L_{1}) = \chi(L_{0}, L_{1}; H)
\]
to be the set of all time-one chords for $H$ from admissible Lagrangians $L_{0}$ to $L_{1}$.
The choice of grading on $L_0$ and $L_1$ gives each $x$ a well-defined Maslov index, which we will denote by $\deg(x)$.  
To each chord $x \in \mathcal{X}(L_{0}, L_{1}; H)$, denote by $o_{x}$ its orientation line,
and $|o_x|_{\K}$ the associated $\K$-normalized orientation space.  
The {\it wrapped Floer cochain space} over $\K$ is defined to be, as a vector space, 
\begin{equation}\label{Floer cochains}
CW^{i}(L_0, L_1; H) = \bigoplus_{x \in \chi(L_{0}, L_{1}), \deg(x) = i} |o_x|_{\K}.
\end{equation}
There is a canonical basis, corresponding to the chord $x$, 
and we will denote each such element by $[x]$.

Consider the moduli space of maps 
\[ \mathcal{R}^{d+1}(\mathbf{x})\] 
defined in \eqref{wrapped moduli space}, where the degrees satisfy
\begin{equation}
\deg(x_{0}) - \deg(x_{1}) - \cdots - \deg(x_{d}) = 2 - d,
\end{equation}
so that all the elements are rigid for generic choices of Floer data.
Each rigid element $u \in \mathcal{R}^{d+1}(\mathbf{x})$ induces an isomorphism
\[
\mathcal{R}^{d+1}_{u}: o_{x_{d}} \otimes \cdots \otimes o_{x_{1}} \to o_{x_{0}},
\]
thus also an isomorphism of $\K$-normalized orientation space,
\[
\mathcal{R}^{d+1}_{u}: |o_{x_{d}}|_{\K} \otimes \cdots \otimes |o_{x_{1}}|_{\K} \to |o_{x_{0}}|_{\K},
\]
Define a map of degree $2-d$
\begin{equation}\label{eq:mud}
\begin{split}
\mu^{d}:  &CW^{*}(L_{d-1}, L_{d}; H) \otimes \cdots \otimes CW^{*}(L_{0}, L_{1}; H) \to CW^{*}(L_{0}, L_{d}; H) [2-d]\\
\mu^{d}([x_{d}] \otimes \cdots \otimes [x_{1}]) & = \sum_{\substack{x_{0}\\\deg(x_{0}) = \deg(x_{1}) + \cdots + \deg(x_{d}) + 2 - d}} 
\sum_{u \in \mathcal{R}^{d+1}(\mathbf{x})} (-1)^{*_{d}} \mathcal{R}^{d+1}_{u}([x_{d}] \otimes \cdots \otimes [x_{1}]),
\end{split}
\end{equation}
where
\begin{equation}
*_{d} = \sum_{j=1}^{d} j \deg(x_{j}).
\end{equation}
The fact that this map is well-defined follows from Proposition \ref{master prop for smoothness and compactness} plus an action estimate (see e.g., \cite{abouzaid1}*{\S B}) which guarantees that for a given $x_1, \ldots, x_d$ there are only finitely many $x_0$ for which $\mathcal{R}^{d+1}(\mathbf{x})$ are non-empty. The same action estimate tells us that there are only finitely many $x_0$ for which the one-dimensional components of the compactified moduli spaces $\bar{\mathcal{R}}^{d+1}(\mathbf{x})$ are non-empty, and by looking at the codimension-one boundary strata of these compact one-dimensional moduli spaces Proposition \ref{master prop for smoothness and compactness} implies that

\begin{lem}
The maps $\mu^{d}$ satisfy the $\ainf$-relations. \qed
\end{lem}

In particular, $\mu^{1}$ is a differential; call the resulting cohomology group the {\it wrapped Floer cohomology}, and denote it by $HW^{*}(L_{0}, L_{1})$, which turns out to be independent of choices of $H$ and $J_{t}$ by a standard continuation argument.

Fix an at most countable collection $\mathbf{L}$ of admissible exact Lagrangians that are cylindrical at infinity,
such that \eqref{non-degenerate Reeb dynamics} holds for these Lagrangians. 
The {\em wrapped Fukaya category} $\w = \w(\mathbf{L})$ of $(X, \lambda)$ is the $\ainf$-category with objects in $\mathbf{L}$,
morphism spaces 
\begin{equation}
\w(L_{0}, L_{1}) = \hom_{\w}(L_{0}, L_{1}): = CW^{*}(L_{0}, L_{1}; H),
\end{equation}
and $\ainf$-structure maps given by $\mu^{d}$ \eqref{eq:mud}.
In case where no confusion may occur, we refer to this also as the wrapped Fukaya category of $X$, and denote it simply by $\w = \w(X)$.

\subsection{The Rabinowitz Floer complex}\label{section:rc}

Consider the Hamiltonian $-H$ on $X$. 
It is not an admissible Hamiltonian for the definition of the wrapped Fukaya category,
but we can still define a Floer cochain space for this Hamiltonian.
There is a natural one-to-one correspondence between $H$-chords from $L_{1}$ to $L_{0}$ of Maslov index $n-i$ and $(-H)$-chords from $L_{0}$ to $L_{1}$ of Maslov index $i$.
For each $(-H)$-chord $x \in \chi(L_{0}, L_{1}; -H)$, 
let $\bar{x} \in \chi(L_{1}, L_{0}; H)$ be the corresponding $H$-chord,
and let $o^{-}_{\bar{x}}$ be the negative orientation line for the chord $\bar{x}$,
such that there is a canonical isomorphism $o_{x} \cong o^{-}_{\bar{x}}$.
We define
\begin{equation}\label{Floer chains}
CW^{i}(L_{0}, L_{1}; -H) = \prod_{\substack{x \in \chi(L_{0}, L_{1}; -H) \\ \deg(x) = i}} |o_{x}|_{\K} = \prod_{\substack{\bar{x} \in \chi(L_{1}, L_{0}; H)\\ \deg(\bar{x}) = n-i}} |o^{-}_{\bar{x}}|_{\K}.
\end{equation}
Here the direct product can be understood to be the completion of the direct sum with respect to the action filtration:
The action spectrum of $(-H)$-chord is a discrete subset of $\R$, which is bounded below but goes off to $+\infty$ (compared to the case of $H$-chords),
 and we allow infinite sums as long as the generators are corresponding to chords with action $\to +\infty$.
In the $\K$-normalized orientation space $|o^{-}_{\bar{x}}|_{\K}$ for the negative orientation line $o^{-}_{\bar{x}}$, 
there is a canonical generator, denoted by $[\bar{x}]^{-}$, corresponding to each chord $\bar{x}$.

When the Lagrangians are oriented, there is a Poincar\'{e} duality isomorphism
\begin{equation}\label{PD}
\bar{I}: CW^{*}(L_{0}, L_{1}; -H) \stackrel{\cong}\to \hom_{\K}(CW^{n-*}(L_{1}, L_{0}; H), \K),
\end{equation}
induced by pairings $o^{-}_{\bar{x}} \otimes o_{\bar{x}} \to \Lambda^{top}T_{p}L_{0} \cong \K$ for some point $p = x(0) \in L_{0}$,
where the determinant line of $TL_{0}$ is canonically isomorphic to the trivial bundle if and only if $L_{0}$ is oriented, and similarly for $L_{1}$.
Note that the linear dual of the direct sum \eqref{Floer cochains} is naturally identified with a direct product with the dual basis.

We could define the Floer differential in the usual way, 
by counting rigid elements in moduli spaces of inhomogeneous pseudoholomorphic strips similar to $\mathcal{R}^{2}(x_{0}, x_{1})$, 
with $H$ replaced by $-H$ in the Floer's equation.
Alternatively, we think of such strips as popsicles with weights $w_{0} = w_{1} = -1$,
which carry no sprinkles.
Each rigid element $u \in \mathcal{R}^{2, \varnothing, (-1, -1)}(x_{0}, x_{1})$ induces an isomorphism of negative orientation lines:
\begin{equation}
\mathcal{R}^{2, \varnothing, (-1, -1)}_{u}: o^{-}_{x_{1}} \to o^{-}_{x_{0}}.
\end{equation}
Recall that $x_{0}, x_{1}$ are by convention time-one $H$-chords from $L_{1}$ to $L_{0}$, 
which naturally corresponds to unique time-one $(-H)$-chords from $L_{0}$ to $L_{1}$ of complementary degrees $n - *$.
For notational convenience, denote by $x_{0}^{-}, x_{1}^{-}$ the corresponding $(-H)$-chords,
such that under the canonical isomorphism $o_{x_{j}}^{-} \cong o_{x_{j}^{-}}$, 
the induced isomorphism on the normalized orientation space is given by
\begin{equation}\label{negative orientation identification}
\begin{split}
|o_{x_{j}}^{-}|_{\K} &\to |o_{x_{j}^{-}}|_{\K}, \\
[o_{x_{j}}]^{-} &\mapsto [o_{x_{j}}^{-}].
\end{split}
\end{equation}
Summing over all rigid elements in $\mathcal{R}^{2, \varnothing, (-1, -1)}(x_{0}, x_{1})$, 
and taking the direct product over all possible $x_{0}$, we define a map
\begin{equation}\label{d-}
d_{-}: CW^{*}(L_{0}, L_{1}; -H)  \to CW^{*}(L_{0}, L_{1}; -H)[1]
\end{equation}
which on a basis element is given by
\begin{equation}\label{d- on basis}
\begin{split}
d_{-}([x_{1}^{-}]) & = \prod_{\substack{x_{0} \\ \deg(x_{1}) - \deg(x_{0})=1}} \sum_{u \in \mathcal{R}^{2, (-1, -1)}(x_{0}, x_{1})} (-1)^{n-\deg(x_{1})} \mathcal{R}^{2, \varnothing, (-1, -1)}_{u} ([x_{1}^{-}]) \\
& =  \prod_{\substack{x_{0}^{-} \\ \deg(x_{0}^{-}) - \deg(x_{1}^{-})=1}} \sum_{u \in \mathcal{R}^{2, (-1, -1)}(x_{0}, x_{1})} (-1)^{\deg(x_{1}^{-})} \mathcal{R}^{2, \varnothing, (-1, -1)}_{u} ([x_{1}^{-}]).
\end{split}
\end{equation}

\begin{lem}\label{d- well-defined}
$d_{-}$ is well-defined on the whole of $CW^{*}(L_{0}, L_{1}; -H)$, i.e. on the direct product as in \eqref{Floer chains}.
\end{lem}
\begin{proof}
Observe that the Floer differential increases the action of chords,
which means that for every $[x_{0}^{-}]$ appearing as an output for $[x_{1}^{-}]$, the corresponding $(-H)$-chords $x_{0}^{-}$ and $x_{1}^{-}$ satisfy
\begin{equation}
\mathcal{A}_{-H}(x_{0}^{-}) \ge \mathcal{A}_{-H}(x_{1}^{-}).
\end{equation}
Order the elements in the natural basis for $CW^{*}(L_{0}, L_{1}; -H)$ according to an increasing order on the action of the corresponding chords.
With respect to that basis, the operation $d_{-}$ can be written as an lower triangular matrix, 
which in particular implies that for any basis vector $v_{i}$, there can be at most finitely many other basis vectors that are mapped to $v_{i}$.
It follows that the operation $d_{-}$ defined on basis as in \eqref{d- on basis} can be extended to the full operation $d_{-}$ on the direct product \eqref{Floer chains}.
\end{proof}

When the Lagrangians are oriented, 
we can alternatively define the differential $d_{-}$ by first defining a differential $\p$ of degree $-1$:
\begin{equation}
\p: CW^{*}(L_{1}, L_{0}; H) \to CW^{*}(L_{1}, L_{0}; H)[-1]
\end{equation}
by counting inhomogeneous pseudoholomorphic strips with input and output interchanged, 
and putting
\begin{equation}
d_{-} = \bar{I}^{-1} \circ \p^{*} \circ \bar{I},
\end{equation}
where 
\[
\p^{*}: \hom_{\K}(CW^{n-*}(L_{1}, L_{0}; H), \K) \to \hom_{\K}(CW^{n-*}(L_{1}, L_{0}; H), \K)[1]
\]
is the linear dual of $\p$. Defining $d_-$ this way requires the almost complex structure chosen for $-H$ to agree with that previously chosen for $H$; however, for the direct definition of $d_-$ the almost complex structures need not coincide. 

There is a canonical {\it continuation map}
\begin{equation}\label{continuation}
c: CW^{*}(L_{0}, L_{1}; -H) \to CW^{*}(L_{0}, L_{1}; H)
\end{equation}
given by an increasing homotopy from $-H$ to $H$. 
Below we give a definition suited to our analytic framework using popsicles.
Consider the following moduli space
\[
\mathcal{R}^{2, \mathbf{p}, (0, -1)}(x_{0}, x_{1})
\]
with weights $w_{0} = 0, w_{1} = -1$,
where the popsicle structure of flavor $\mathbf{p}$ is a choice of a sprinkle on the geodesic connecting $z_{1}$ to $z_{0}$.
Each rigid element $u \in \mathcal{R}^{2, \mathbf{p}, (0, -1)}(x_{0}, x_{1})$ induces an isomorphism
\[
\mathcal{R}^{2, \mathbf{p}, (0, -1)}_{u}: o_{x_{1}}^{-} \to o_{x_{0}},
\]
and thus also an isomorphism of $\K$-normalized orientation spaces following the identification \eqref{negative orientation identification},
\[
\mathcal{R}^{2, \mathbf{p}, (0, -1)}_{u}: |o_{x_{1}}^{-}|_{\K} \to |o_{x_{0}}|_{\K},
\]
Summing over all rigid elements, and taking the direct product, we define
\begin{equation}
c([x_{1}^{-}]) = \prod_{\substack{x_{0} \\ \deg(x_{0}) = n - \deg(x_{1})}} \sum_{u \in \mathcal{R}^{2, \mathbf{p}, (0, -1)}(x_{0}, x_{1})} \mathcal{R}^{2, \mathbf{p}, (0, -1)}_{u}([x_{1}^{-}])
\end{equation}
Since the Lagrangians are oriented, we can alternatively define $c$ by first defining a {\it copairing}
\begin{equation}\label{copairing}
\hat{c}: \K \to CW^{n-*}(L_{1}, L_{0}; H) \otimes CW^{*}(L_{0}, L_{1}; H) 
\end{equation}
and setting
\begin{equation}
c(\cdot) = (\bar{I}(\cdot) \otimes \id)(\hat{c}).
\end{equation}
The copairing \eqref{copairing} is also defined by counting rigid elements in the moduli space of popsicles
\[
\mathcal{R}^{2, \mathbf{p}, (0, -1)}(x_{0}, x_{1})
\]

Note that by \eqref{virtual dimension formula for moduli space of popsicles},
the moduli space has virtual dimension 
\begin{equation}
\dim \mathcal{R}^{2, \mathbf{p}, (0, -1)}(x_{0}, x_{1}) = -n + \deg(x_{0}) + \deg(x_{1}),
\end{equation}
and this is why the copairing \eqref{copairing} has the desired degree.

\begin{defn}
The Rabinowitz Floer cochain space is defined to be the mapping cone of the continuation map \eqref{continuation}
\begin{equation}\label{Rabinowitz complex}
\begin{split}
RC^{*}(L_{0}, L_{1}) &:= \cone(c: CW^{*}(L_{0}, L_{1}; -H) \to CW^{*}(L_{0}, L_{1}; H)) \\
&= CW^{*}(L_{0}, L_{1}; -H)[1] \oplus CW^{*}(L_{0}, L_{1}; H)).
\end{split}
\end{equation}
The differential is the cone differential
\begin{equation}
\mu^{1}_{RC} = 
\begin{pmatrix}
d_{-} & c[1] \\
0 & \mu^{1}
\end{pmatrix}
\end{equation}
\end{defn}

\begin{lem}
$(\mu^{1}_{RC})^{2} = 0$.
\end{lem}
\begin{proof}
	The only thing to check is $c \circ d_{-} + \mu^{1} \circ c = 0$, which is to say that $c$ defined above is a chain map.
This follows by looking at the boundary strata of the one-dimensional moduli spaces $\mathcal{R}^{2, \mathbf{p}, (0, -1)}(x_{0}, x_{1})$,
where two different types of strata contribute to $c \circ d_{-}$ and $\mu^{1} \circ c$,
with the opposite signs.
\end{proof}

We denote the cohomology of $\mu^1_{RC}$ on $RC^{*}(L_{0}, L_{1})$ by $RW^*(L_0,L_1)$. 
\begin{lem}\label{compactnessvanishing}
    If either of $L_0$ or $L_1$ is compact, then $RW^*(L_0,L_1)$ is zero.
\end{lem}
\begin{proof}[Sketch]
    It is a well known (and a manifestation of Poincar\'{e} duality in Floer theory) that if one of $L_0$ or $L_1$ is compact, then \eqref{continuation} induces a quasi-isomorphism between the Floer complexes for $-H$ and $H$.
\end{proof}

\subsection{The \texorpdfstring{$\ainf$}{A-infinity} structure on the Rabinowitz Floer complex}\label{section:A-infinity on rc}

The goal of this subsection is to define the $\ainf$-structure on the Rabinowitz Floer cochain complexes \eqref{Rabinowitz complex}, which takes the form of a series of maps
\begin{equation}\label{A-infinity map on RC}
\mu^{d}_{\rw}: RC^{*}(L_{d-1}, L_{d}) \otimes \cdots \otimes RC^{*}(L_{0}, L_{1}) \to RC^{*}(L_{0}, L_{d}),
\end{equation}
extending $\mu^1$ defined in the previous section, satisfying the $\ainf$ equations. Given the form of $RC^*(L_i, L_j)$ as defined in \eqref{Rabinowitz complex}, our final map \eqref{A-infinity map on RC} will be a suitable sum of maps of the form
\[
 CW^{*}(L_{d-1}, L_{d}; s_{d} H) \otimes \cdots \otimes CW^{*}(L_{0}, L_{1}; s_{1}H) \to CW^{*}(L_{0}, L_{d}; s_{0}H)
 \]
 for each sequence $\{s_i\ |\ s_i \in \{+,-\}\}$. We first proceed to describe these constituent maps.

 Consider various kinds of moduli spaces of popsicles $\mathcal{R}^{d+1, \mathbf{p}, \mathbf{w}}(\mathbf{x})$ as defined in \S \ref{section:moduli space of popsicle maps} with weights satisfying 
\begin{equation}
w_{j} \in \{-1, 0\}, j = 0, \ldots, d.
\end{equation}
Suppose we have made a universal and conformally consistent choice of Floer data for all weighted popsicles with weights $\le 0$ to construct the moduli spaces $\mathcal{R}^{d+1, \mathbf{p}, \mathbf{w}}(\mathbf{x})$.
In addition, suppose Floer data are chosen generically such that all the zero and one dimensional components of the moduli spaces are regular.
Now suppose $w_{j} \in \{-1, 0\}$, and
\[
 d - 2 + |F| + n(1 + \sum_{j=0}^{d} \d_{j}) - \sum_{j=0}^{d} (-1)^{\d_{j}} \deg(x_{j}) = 0,
\]
where the $\d_{j}$'s are defined in \eqref{output symbol}.
The moduli space $\mathcal{R}^{d+1, \mathbf{p}, \mathbf{w}}(\mathbf{x})$ is a compact smooth manifold of dimension zero.
Each element $u \in \mathcal{R}^{d+1, \mathbf{p}, \mathbf{w}}(\mathbf{x})$ is rigid and induces an isomorphism
\begin{equation}\label{iso 2 orientation lines}
\mathcal{R}^{d+1, \mathbf{p}, \mathbf{w}}_{u}: \bigotimes_{j=1}^{k} o_{x_{j}}^{s_{j}} \to o_{x_{0}}^{s_{0}},
\end{equation}
where the signs $s_{j} \in \{+,-\}$'s are defined in \eqref{sign notation}.
Its induced map on normalized orientation space is
\begin{equation}
    \mathcal{R}^{d+1, \mathbf{p}, \mathbf{w}}_{u}: \bigotimes_{j=1}^{k} |o_{x_{j}}^{s_{j}}|_{\K} \to |o_{x_{0}}^{s_{0}}|_{\K}.
\end{equation}
Equivalently, using negative orientation lines, we may give an alternative description of the isomorphism \eqref{iso 2 orientation lines}.
Define inputs 
\begin{equation}\label{inputs}
\mathbf{x}_{in} = (x_{j})_{j: \d_{j} = 0},
\end{equation}
and outputs
\begin{equation}\label{outputs}
\mathbf{x}_{out} = (x_{j})_{j: \d_{j} = -1},
\end{equation}
where $j$ varies in $0, 1, \ldots, d$.
These are understood to be associated to the given collection of weights $\mathbf{w}$ in $\{-1, 0\}$ as well as the popsicle flavor $\mathbf{p}$.
With these notations, we may rewrite \eqref{iso 2 orientation lines} as
\begin{equation}\label{iso 1 orientation lines}
\mathcal{R}^{d+1, \mathbf{p}, \mathbf{w}}_{u}: o_{\mathbf{x}_{in}} := \bigotimes_{j: \d_{j} = 0} o_{x_{j}} \to o_{\mathbf{x}_{out}} := \bigotimes_{j: \d_{j} = - 1} o_{x_{j}}.
\end{equation}

\begin{rem}\label{remark on signs}
Passing from \eqref{iso 2 orientation lines} to \eqref{iso 1 orientation lines} involves a lot of changes in the order of the tensor factors of the orientation lines appearing in \eqref{iso 2 orientation lines}, 
as well as choices of trivializations of determinant lines for the tangent bundles of several Lagrangians,
which will result in lots of sign changes which are too tedious to keep track of.
Thus, when we actually define $\ainf$-operations, we will only use \eqref{iso 2 orientation lines} for a rigorous definition with the correct signs,
while refer to \eqref{iso 1 orientation lines} as a geometric interpretation, without specifying the signs.
\end{rem}

Summing over all rigid elements $u \in \mathcal{R}^{d+1, \mathbf{p}, \mathbf{w}}(\mathbf{x})$,
and taking the direct product over all possible $x_{0}$,
we obtain a map of degree $2-d+|F|$ of the following form
\begin{equation}\label{a-infinity maps for rw}
\mu^{d, \mathbf{p}, \mathbf{w}}:  CW^{*}(L_{d-1}, L_{d}; s_{d} H) \otimes \cdots \otimes CW^{*}(L_{0}, L_{1}; s_{1}H) \to CW^{*}(L_{0}, L_{d}; s_{0}H),
\end{equation}
 which on basis elements is given by
 \begin{equation}\label{mud on basis}
 \begin{split}
& \mu^{d, \mathbf{p}, \mathbf{w}}  ([x_{d}^{s_{d}}] \otimes \cdots \otimes [x_{1}^{s_{1}}]) \\
 = & \prod_{\substack{x_{0}\\ d - 2 + |F| + n(1 + \sum_{j=0}^{d} \d_{j}) - \sum_{j=0}^{d} (-1)^{\d_{j}} \deg(x_{j}) = 0}} \sum_{u \in \mathcal{R}^{d+1, \mathbf{p}, \mathbf{w}}(\mathbf{x})} 
  (-1)^{*_{d, \mathbf{p}, \mathbf{w}} + \Diamond_{d, \mathbf{p}, \mathbf{w}}}
\mathcal{R}^{d+1, \mathbf{p}, \mathbf{w}}_{u} ([x_{d}^{s_{d}}] \otimes \cdots \otimes [x_{1}^{s_{1}}]),
\end{split}
\end{equation}
where the isomorphism $\mathcal{R}^{d+1, \mathbf{p}, \mathbf{w}}_{u}$ are understood as \eqref{iso 2 orientation lines},
and basis vector $[x_{j}^{s_{j}}]$ of the normalized orientation space $|o_{x_{j}}^{s_{j}}|_{\K}$ for $s_{j} = -$ is defined in \eqref{negative orientation identification}.
Here the signs $*_{d, \mathbf{p}, \mathbf{w}}$ are
\begin{equation}\label{sign formula 1 for popsicles}
*_{d, \mathbf{p}, \mathbf{w}} = \sum_{j=1}^{d}( j + w_{1} + \cdots + w_{j-1}) \deg(x_{j}^{s_{j}}) + \sum_{j=1}^{d} (d-j) w_{j},
\end{equation}
\begin{equation}\label{sign formula 2 for popsicles}
\Diamond_{d, \mathbf{p}, \mathbf{w}} = \sum_{j=1}^{d} | \mathbf{p}^{-1}(\{ j+1, \ldots, d\})| (w_{j} + |\mathbf{p}^{-1}(j)|).
\end{equation}
Since we have chosen Floer data to be invariant under the action of $Aut(\mathbf{p})$ (in \S\ref{section:moduli space of popsicle maps}), following \cite{abouzaidseidel} we have the following vanishing result for non-injective $\mathbf{p}$:

\begin{lem}[\cite{abouzaidseidel} Lemma 3.7]\label{not injective implies vanishing} 
If $Aut(\mathbf{p})$ is nontrivial, then $\mu^{d, \mathbf{p}, \mathbf{w}}  ([x_{d}^{s_{d}}] \otimes \cdots \otimes [x_{1}^{s_{1}}])$ as in \eqref{mud on basis} vanishes.
\end{lem}

To obtain the desired map \eqref{a-infinity maps for rw}, we also need:

\begin{lem}\label{mud well-defined}
The operations $\mu^{d, \mathbf{p}, \mathbf{w}}$ defined on basis elements as in \eqref{mud on basis} can be extended to the whole of the tensor product 
\[
CW^{*}(L_{d-1}, L_{d}; s_{d} H) \otimes \cdots \otimes CW^{*}(L_{0}, L_{1}; s_{1}H) \to CW^{*}(L_{0}, L_{d}; s_{0}H).
\]
\end{lem}
\begin{proof}
Again, this follows by an action argument, similar to Lemma \ref{d- well-defined}.
To spell this out, note that each operation $\mu^{d, \mathbf{p}, \mathbf{w}}$ \eqref{mud on basis} increases the action,
\begin{equation}
\mathcal{A}_{s_{0}H}(x_{0}^{s_{0}}) \ge \sum_{j=1}^{d} \mathcal{A}_{s_{j}H}(x_{j}^{s_{j}}).
\end{equation}
The sum of the actions of basis elements can be extended to an action filtration on the algebraic tensor product
\[
CW^{*}(L_{d-1}, L_{d}; s_{d} H) \otimes \cdots \otimes CW^{*}(L_{0}, L_{1}; s_{1}H),
\]
which allows us to give an order on the basis vectors, i.e. simple tensors of orientation basis vectors for the corresponding chords, by the sum of actions.
With respect to the ordered basis $\mu^{d, \mathbf{p}, \mathbf{w}}$ \eqref{mud on basis} can be written lower triangular, 
which implies that it can be extended to the completed tensor product.
In particular, since the algebraic tensor product embeds into the completed tensor product as a filtered subspace,
we have a well-defined map \eqref{a-infinity maps for rw}.
\end{proof}

Since the Lagrangians are oriented, we may alternatively define operations
\begin{equation}\label{cochain operations}
\begin{split}
\hat{\mu}^{d, \mathbf{p}, \mathbf{w}}: & CW^{*}(L_{d}, L_{0}; H)^{\otimes (\d_{0} + 1)} \otimes CW^{*}(L_{d-1}, L_{d}; H)^{\otimes (\d_{d}+1)} \otimes \cdots \otimes CW^{*}(L_{0}, L_{1}; H)^{\otimes (\d_{1}+1)} \\
\to & CW^{*}(L_{0}, L_{d}; H)^{\otimes -\d_{0}} \otimes CW^{*}(L_{1}, L_{0}; H)^{\otimes -\d_{1}} \otimes \cdots \otimes CW^{*}(L_{d}, L_{d-1}; H)^{\otimes -\d_{d}}, \\
\hat{\mu}^{d, \mathbf{p}, \mathbf{w}}([\mathbf{x}_{in}]) & = \sum_{\substack{\mathbf{x}_{out} \\  d - 2 + |F| + n(1 + \sum_{j=0}^{d} \d_{j}) - \sum_{j=0}^{d} (-1)^{\d_{j}} \deg(x_{j}) = 0}}
\sum_{u \in \mathcal{R}^{d+1, \mathbf{p}, \mathbf{w}}(\mathbf{x})}
(-1)^{*} \mathcal{R}^{d+1, \mathbf{p}, \mathbf{w}}_{u}([\mathbf{x}_{in}]),
\end{split}
\end{equation}
where the isomorphism $\mathcal{R}^{d+1, \mathbf{p}, \mathbf{w}}_{u}$ are understood as \eqref{iso 1 orientation lines}.
Note that by Remark \ref{remark on signs}, we do not want to specify the sign here, so this formula only serves an an exposition for geometric meaning of the count and degrees, etc.
The degree of this map is
\begin{equation}
2 - d - |F| + n(1 + \sum_{j=0}^{d} \d_{j}) = n + 1 - d + (n-1)(\d_{0} + \cdots + \d_{d}),
\end{equation}
where the equality follows from that fact that $w_{j} \in \{-1, 0\}$, $\d_{0} =  -1 - w_{0}$ and $\d_{j} = w_{j}$ for $j = 1, \ldots, d$.
The conditions \eqref{weight-sprinkle} and \eqref{maximum number of sprinkles} ensure that there is always at least one output.
For each of such maps \eqref{cochain operations}, using the isomorphism $\bar{I}$ \eqref{PD} we may define a corresponding map 
\begin{equation}\label{interchanging inputs and outputs}
\mu^{d, \mathbf{p}, \mathbf{w}}:  CW^{*}(L_{d-1}, L_{d}; s_{d} H) \otimes \cdots \otimes CW^{*}(L_{0}, L_{1}; s_{1}H) \to CW^{*}(L_{0}, L_{d}; s_{0}H)
\end{equation}
(which has the same form as and in fact agrees with \eqref{a-infinity maps for rw}) 
in a way similar to the definition of the differential $\d_{-}$ via $\p$,
or the definition of the continuation map $c$ \eqref{continuation} via the copairing $\hat{c}$ \eqref{copairing}. 
The rule is that conjugating by the isomorphism $\bar{I}$ allows us to transfer any output in $CW^{*}(L_{j}, L_{j-1}; H)$ to an input from $CW^{*}(L_{j-1}, L_{j}; -H)$.
For every output in $CW^{*}(L_{j}, L_{j-1}; H)$ turned into an input from $CW^{*}(L_{j-1}, L_{j}; -H)$, the degree of the map is shifted by $n$.
And if an input from $CW^{*}(L_{d}, L_{0}; H)$ is turned into an output in $CW^{*}(L_{0}, L_{d}; -H)$,
the degree of the map is shifted by $-n$.
Thus, the degree of the map \eqref{interchanging inputs and outputs}
\begin{equation}
1 - d - (\d_{0} + \cdots + \d_{d}) = 2 - d + w_{0} - w_{1} - \cdots - w_{d} = 2 - d + |F|,
\end{equation}
as one would expect from the moduli of popsicles in the usual sense.

Using the definition of the Rabinowitz complexes \eqref{Rabinowitz complex},
we combine all these maps \eqref{a-infinity maps for rw} to define the overall $\mu^d$ map \eqref{A-infinity map on RC},
as on a given component taking the form \eqref{a-infinity maps for rw}.
Because of the shift of $CW^{*}(L_{0}, L_{1}; -H)$ in the mapping one \eqref{Rabinowitz complex}, the degree works out as follows.
If for some $1 \le j \le d$, $w_{j}=-1$, the input is from $CW^{*}(L_{j-1}, L_{j}; -H)$, whose absolute degree given by the Maslov index equals one plus the degree as a generator of $RC^{*}(L_{j-1}, L_{j})$,
so that the degree of a map from there is reduced by $1$, i.e. added by $w_{j}$.
If $w_{0} = -1$, the output is in $CW^{*}(L_{0}, L_{d}; -H)$, and the same degree shift applies.
Thus, each of the maps \eqref{a-infinity maps for rw}, when extended to a map on the Rabinowitz complexes, has degree
\begin{equation}
2 - d + |F| + w_{1} + \cdots + w_{k} - w_{0} = 2 - d,
\end{equation}
as desired.
In case $d=1$, we use interchangeably the notation $\mu^{1}_{RC} = \mu^{1}_{\rw}$.

\begin{lem}
The maps $\mu^{d}_{\rw}$ \eqref{A-infinity map on RC} satisfy the $\ainf$-relations.
\end{lem}
\begin{proof}
The proof is a standard argument by matching the terms in the $\ainf$-associativity equations with contributions from pieces of the boundary strata of the compactification $\bar{\mathcal{R}}^{d+1, \mathbf{p}, \mathbf{w}}(\mathbf{x})$.
The only trouble is that there are broken popsicles appearing the in boundary strata carrying weights $<-1$,
which would {\it a priori} break the $\ainf$-associativity equations,
because such popsicles are not counted in the definition of $\mu^{d}_{\rw}$. 
However, by Lemma \ref{undesired popsicles}, those contributions are zero:
\begin{enumerate}[label=(\roman*)]

    \item either the popsicle is not injective on the first component, so that the algebraic operation coming from that component is zero by Lemma \ref{not injective implies vanishing};

\item or the broken popsicle is the common boundary point of exact two moduli spaces $\bar{\mathcal{R}}^{d+1, \mathbf{p}, \mathbf{w}}$ which differ only in the flavors $\mathbf{p}$ as described in the case (ii) of Lemma \ref{undesired popsicles}, 
and the counts of them are in the opposite signs according to \eqref{sign formula 1 for popsicles} and \eqref{sign formula 2 for popsicles}.

\end{enumerate}
The two resulting operations $\mu^{d, \mathbf{p}, \mathbf{w}}$ are added in the definition of $\mu^{d}_{\rw}$, 
so that their compositions with some other $\mu$'s in the $\ainf$-equations cancel each other.
\end{proof}

\begin{defn}\label{def: Rabinowitz Fukaya category}
Fix an at most countable collection $\mathbf{L}$ of admissible exact Lagrangians that are cylindrical at infinity,
such that \eqref{non-degenerate Reeb dynamics} holds for these Lagrangians. 
The Rabinowitz (wrapped) Fukaya category $\rw = \rw(\mathbf{L})$ of $(X, \lambda)$ is the $\ainf$-category with objects in $\mathbf{L}$, 
morphism spaces
\begin{equation}
\rw(L_{0}, L_{1}) = \hom_{\rw}(L_{0}, L_{1}) := RC^{*}(L_{0}, L_{1}),
\end{equation}
and $\ainf$-structure maps given by $\mu^{d}_{\rw}$ \eqref{A-infinity map on RC}.
\end{defn}

When there is no confusion about the specific collection $\mathbf{L}$ of Lagrangians,
we also call it the Rabinowitz Fukaya category of $X$, and denote it by $\rw = \rw(X)$.

By the definition of the Rabinowitz complex $RC^{*}(L_{0}, L_{1})$ as in \eqref{Rabinowitz complex}, 
the wrapped Floer complex $CW^{*}(L_{0}, L_{1}; H)$ is a subcomplex. 
Define maps
\begin{equation}\label{components of j}
j_{\rw}^{d}: CW^{*}(L_{d-1}, L_{d}; H) \otimes \cdots \otimes CW^{*}(L_{0}, L_{1}; H) \to RC^{*}(L_{0}, L_{d})[1-d]
\end{equation}
by $j_{\rw}^{1}$ the inclusion of of the subcomplex $CW^{*}(L_{0}, L_{1}; H)$ into  $RC^{*}(L_{0}, L_{1})$, and $j_{\rw}^{d} = 0$ for all $d \ge 2$.

\begin{lem}
The maps \eqref{components of j} define an $\ainf$-functor
\begin{equation}\label{w to rw}
j_{\rw}: \w(X) \to \rw(X)
\end{equation}
\end{lem}
\begin{proof}
If all the inputs of $\mu^{d}_{\rw}$ are from $\hom_{\w}(L_{i-1}, L_{i})$, 
then any popsicle with these inputs must have a trivial popsicle structure (i.e. without sprinkles),
and is therefore an ordinary $\ainf$-disk counted in the definition of $\mu^{d}_{\w}$.
\end{proof}

\section{The functor from $\rw$ to \texorpdfstring{$\winf$}{W-hat-infinity}}\label{section: functor}

The main objective of this section is to construct an $\ainf$-functor
\begin{equation}\label{eq:rwtowinf}
\Phi: \rw \to \winf,
\end{equation}
and prove that it is a quasi-equivalence when the Liouville manifold is non-degenerate (a notion recalled in Definition \ref{defn:nondegenerate}).
On each Lagrangian $L \in \ob \w$, $\Phi$ acts as the identity, $\Phi(L) = L$.
To define the action of the functor on morphism spaces, let us first recall that the definition of the formal punctured neighborhood at infinity, as described in \eqref{cinfmorphismspelledout}. 
There is a natural equivalence,
\begin{equation}\label{winfascone}
    \winf (K, L)\cong \cone (\r{CC}^{*}(\w^{op}, (Y^{r}_{K})^{*} \otimes_{\K} Y^{r}_{L}) \stackrel{ev \circ}{\to} \r{CC}^{*}(\w^{op}, \hom_{\K}(Y^{r}_{K}, Y^{r}_{L}))),
\end{equation}
with the $\ainf$-structures described as: $\mu^{1}$ the cone differential on Hochschild complexes, and $\mu^{2}$ in the induced Yoneda product, given in \eqref{mu2 in cinf}.
The $\ainf$-functor $\Phi$ will formally constructed in terms of operations $\Phi_-$ landing in the first (shifted) factor of the cone \eqref{winfascone}, constructed in \S \ref{negpart} and operations $\Phi_+$ landing in the second (unshifted) factor of the cone \eqref{winfascone}, constructed in \S \ref{pospart}. 
The verification of the $\ainf$ functor equations and hence a proof of Theorem \ref{thm:main}(i) / Theorem \ref{thmfunctor} then occurs in \S \ref{functorequations}. In \S \ref{statementofequivalence}, Theorem \ref{thm:main}(ii) / Theorem \ref{thm:equivalence} is restated, and the remaining subsections are devoted to its proof. The strategy of the proof is to observe that the linear term of the $\ainf$-functor $\Phi^1$ respects the natural two step cone-induced filtrations on both sides, and therefore will be a quasi-isomorphism if the maps from the shifted respectively unshifted parts of the Rabinowitz complex to the shifted respectively unshifted parts of \eqref{winfascone} are each quasi-isomorphisms. The map on positive parts is essentially the Yoneda map as shown in \S \ref{positiveyoneda}. The map on negative parts is comparable to a Calabi-Yau type morphism of the form studied in \cite{ganatra}; the precise comparison occurs in \S \ref{section: inverse dualizing bimodule}, \S \ref{section: cy-}, and \S \ref{section: homotopy argument}. Both of these maps are known to be quasi-isomorphisms, the first always and the second assuming non-degeneracy. All together these arguments are combined to prove Theorem \ref{thm:main}(ii)/Theorem \ref{thm:equivalence} at the end of \S \ref{section: homotopy argument}.

\subsection{The positive part}\label{pospart}

Let $k \ge 1$ be a positive integer, and let $L_{0}, \cdots, L_{k}$ be $k+1$ objects of $\rw$. 
We shall construct a sequence of maps
\begin{equation}\label{phi+k}
\Phi^{k}_{+}: RC^{*}(L_{k-1}, L_{k}) \otimes \cdots \otimes RC^{*}(L_{0}, L_{1}) \to \r{CC}^{*}(\w^{op}, \hom_{\K}(Y_{L_{0}}^{r}, Y_{L_{k}}^{r}))
\end{equation}
of degree $1 - k$.
To write these maps out in a more concrete way, let $L'_{1}, \cdots, L'_{l}$ be another $l+1$ testing objects.
For each $c_{i} \in RC^{*}(L_{i-1}, L_{i})$, we want to define the value of $\Phi^{k}_{+}$ on the $k$-tensor 
\[
	c_{k} \otimes \cdots \otimes c_{1}
\]
to be the Hochschild cochain whose $l$-th order term is the a $\K$-linear homomorphism between the wrapped Floer cochain spaces:
\begin{equation}\label{components of phi+}
\begin{split}
 \Phi^{k, l}_{+} (c_{k} \otimes \cdots \otimes c_{1}): &
 CW^{*}(L'_{1}, L'_{0}; H) \otimes \cdots \otimes CW^{*}(L'_{l}, L'_{l-1}; H) \\
 & \to \hom_{\K}(CW^{*}(L'_{0}, L_{0}; H), CW^{*}(L'_{l}, L_{k}; H)),
\end{split}
\end{equation}
for all $L'_{0}, \ldots, L'_{l}$.

Note that by \eqref{Rabinowitz complex}, each element $c_{i}$ has two components, with respect to the decomposition
\[
RC^{*}(L_{i-1}, L_{i}) = CW^{*}(L_{i-1}, L_{i}; -H)[-1] \oplus CW^{*}(L_{i-1}, L_{i}; H).
\]
Thus each of the maps \eqref{components of phi+} is first defined on components,
corresponding to generators $c_{i} = [x_{i}]$ for some time-one $H$-chord $x_{i}$ from $L_{i-1}$ to $L_{i}$, as an input from $CW^{*}(L_{i-1}, L_{i}; H)$,
or a time-one $H$-chord from $L_{i}$ to $L_{i-1}$, as an output in $CW^{*}(L_{i}, L_{i-1}; H)$,
which can be turned into an input from $CW^{*}(L_{i-1}, L_{i}; -H)$ using the isomorphism $\bar{I}$ \eqref{PD}.

Consider the moduli spaces of popsicles
\begin{equation}
\mathcal{R}^{k + l + 2, \mathbf{p}_{+}, \mathbf{w}_{+}}(y'_{out}, \mathbf{x}', y_{in}, \mathbf{x}),
\end{equation}
where we rename the $0$-th puncture as $z'_{out}$,
the $j$-th puncture as $z'_{l-j+1}$, for $j = 1, \ldots, l$,
the $(l+1)$-th puncture as $z_{in}$,
and the $(l+1+i)$-th puncture as $z_{i}$, for $i = 1, \ldots, k$.
Here the weights are
\begin{equation}
\mathbf{w}_{+} = (0, \underbrace{0, \ldots 0}_{l \text{ times}}, 0, \mathbf{w}),
\end{equation}
such that
\begin{equation}
\mathbf{w} = (w_{1}, \ldots, w_{k}), \text{ and } w_{i} \in \{-1, 0\}.
\end{equation}
and a popsicle structure $\sigma_{+}$ of flavor $\mathbf{p}_{+}$ is the choice of sprinkles,
one on each geodesic connecting the $(l+1+i)$-th puncture $z_{i}$ to the $0$-th puncture $z'_{out}$,
whenever 
\[
w_{i} = -1.
\]
The Lagrangian label is chosen in the following order:
\begin{equation}
L'_{l}, \ldots, L'_{0}, L_{0}, \ldots, L_{k}
\end{equation}
counterclockwise along the boundary of the punctured disk,
such that $L'_{l}$ is assigned to the boundary component between the $0$-th puncture $z'_{out}$ and the $1$-st puncture $z'_{l}$.

By \eqref{virtual dimension formula for moduli space of popsicles}, the relevant moduli spaces have dimension zero whenever the degrees of the chords satisfy
\begin{equation}
\deg(y'_{out}) = \deg(y_{in}) + \sum_{i=1}^{k} (-1)^{\d_{i}} \deg(x_{i}) + \sum_{j=1}^{l} \deg(x'_{j}) + 1 - k - l -(n+1) \sum_{i=1}^{k} \d_{i}.
\end{equation}
For notational convenience, let us write
\begin{equation}
o_{\mathbf{x}}^{\mathbf{s}} = o_{x_{k}}^{s_{k}} \otimes \cdots \otimes o_{x_{1}}^{s_{1}},
\end{equation}
and
\begin{equation}
o_{\mathbf{x}'} = o_{x'_{1}} \otimes \cdots \otimes o_{x'_{l}}.
\end{equation}
Each rigid element $u \in \mathcal{R}^{k + l + 2, \mathbf{p}_{+}, \mathbf{w}_{+}}(y'_{out}, \mathbf{x}', y_{in}, \mathbf{x})$ induces an isomorphism between the orientation lines,
\begin{equation}
\mathcal{R}^{k+l+2, \mathbf{p}_{+}, \mathbf{w}_{+}}_{u}:  o_{\mathbf{x}}^{\mathbf{s}} \otimes o_{y_{in}} \otimes o_{\mathbf{x}'} \to o_{y'_{out}}.
\end{equation}
Summing over all rigid elements $u \in \mathcal{R}^{k + l + 2, \mathbf{p}_{+}, \mathbf{w}_{+}}(y'_{out}, \mathbf{x}', y_{in}, \mathbf{x})$, 
and taking the direct product over all possible outputs $y'_{out}$ and $\mathbf{x}_{out}$, 
we obtain a map of degree 
\[
1 - k - l -(n+1) \sum_{i=1}^{k} \d_{i}
\]
of the following form
\begin{equation}\label{components of phi+kl}
\begin{split}
\Phi_{+}^{k, l; \mathbf{p}_{+}, \mathbf{w}_{+}}: & CW^{*}(L_{k-1}, L_{k}; s_{k}H) \otimes \cdots \otimes CW^{*}(L_{0}, L_{1}; s_{0}H) \\
 \to & \hom_{\K}(CW^{*}(L'_{1}, L'_{0}; H) \otimes \cdots \otimes CW^{*}(L'_{l}, L'_{l-1}; H), \\
 & \hom_{\K}(CW^{*}(L'_{0}, L_{0}; H), CW^{*}(L'_{l}, L_{k}; H))), \\
 \end{split}
 \end{equation}
 which on basis elements takes the following values
 \begin{equation}\label{phi+kl on basis elements}
 \begin{split}
 & \Phi_{+}^{k, l; \mathbf{p}_{+}, \mathbf{w}_{+}} ([\mathbf{x}^{\mathbf{s}}] \otimes [y_{in}] \otimes [\mathbf{x}']) \\
 = & \sum_{\substack{y'_{out} \\ \deg(y'_{out}) = \deg(y_{in}) + \sum_{i=1}^{k} (-1)^{\d_{i}} \deg(x_{i}) + \sum_{j=1}^{l} \deg(x'_{j}) + 1 - k - l}} \\
& \sum_{u \in \mathcal{R}^{k + l + 2, \mathbf{p}_{+}, \mathbf{w}_{+}}(y'_{out}, \mathbf{x}', y_{in}, \mathbf{x})} 
(-1)^{*_{k, l; \mathbf{p}_{+}, \mathbf{w}_{+}} + \Diamond_{k, l; \mathbf{p}_{+}, \mathbf{w}_{+}}} \mathcal{R}^{k+l+2, \mathbf{p}_{+}, \mathbf{w}_{+}}_{u}([\mathbf{x}^{\mathbf{s}}] \otimes [y_{in}] \otimes [\mathbf{x}']).
\end{split}
\end{equation}
where the signs $*_{k, l; \mathbf{p}_{+}, \mathbf{w}_{+}}$ and $\Diamond_{k, l; \mathbf{p}_{+}, \mathbf{w}_{+}}$ follows from the general formulas \eqref{sign formula 1 for popsicles} and \eqref{sign formula 2 for popsicles},
by taking $d = k + l + 1$ and matching the popsicle structures and weights.

\begin{lem}\label{lem: phi+ well-defined}
The operations $\Phi_{+}^{k, l; \mathbf{p}_{+}, \mathbf{w}_{+}}$ \eqref{phi+kl on basis elements} extend to well-defined maps \eqref{components of phi+kl} on the whole Floer complexes.
\end{lem}
\begin{proof}
Since the formula in \eqref{phi+kl on basis elements} defines the output as a sum,
when there is some $i \in \{1, \ldots, k\}$ with $w_{i} = -1$ so that $s_{i} = -$, 
we must show that the operation $\Phi_{+}^{k, l; \mathbf{p}_{+}, \mathbf{w}_{+}}$ vanishes for all but finitely many basis elements in the infinite direct product $CW^{*}(L_{i-1}, L_{i}; -H)$.
This follows from an action argument: when we define the operation $\Phi_{+}^{k, l; \mathbf{p}_{+}, \mathbf{w}_{+}}$,
the input $[x_{i}^{-}]$ is treated as an output $H$-chord $x_{i}$ for any inhomogeneous pseudoholomorphic map counted.
It follows that 
\begin{equation}
\mathcal{A}_{H}(y'_{out}) + \sum_{i: w_{i}=-1} \mathcal{A}_{H}(x_{i}) \ge \mathcal{A}_{H}(y_{in}) + \sum_{i: w_{i} = 0} \mathcal{A}_{H}(x_{i}).
\end{equation}
Since the action spectrum of $H$-chords are bounded above and can only go to $-\infty$, 
it follows that for fixed $y_{in}$ and $x_{i}$'s with $w_{i}=0$, 
there are finitely many possibilities for the outputs $y'_{out}$ and $x_{i}$'s with $w_{i} = -1$.
In particular, for any $x_{i}$ with $w_{i}=-1$ and sufficiently negative action, the moduli space is empty, which implies that the operation \eqref{phi+kl on basis elements} takes zero value on the tensor having that factor $[x_{i}^{-}]$.
\end{proof}

Then we extend these maps to maps on the Rabinowitz complexes \eqref{Rabinowitz complex},
\begin{equation}\label{phi+kl}
\begin{split}
\Phi_{+}^{k, l} : & RC^{*}(L_{k-1}, L_{k}) \otimes \cdots \otimes RC^{*}(L_{0}, L_{1}) \\
 \to & \hom_{\K}(CW^{*}(L'_{1}, L'_{0}; H) \otimes \cdots \otimes CW^{*}(L'_{l}, L'_{l-1}; H), \\
 & \hom_{\K}(CW^{*}(L'_{0}, L_{0}; H), CW^{*}(L'_{l}, L_{k}; H))),
\end{split}
\end{equation}
which has degree $1 - k - l - \sum_{i=1}^{k} \d_{i}$.
whose restriction to components determined by $\mathbf{p}_{+}$ and $\mathbf{w}_{+}$ are \eqref{components of phi+kl},
in a way similar to extending $\mu^{d, \mathbf{p}, \mathbf{w}}$ \eqref{a-infinity maps for rw} to $\mu^{d}_{\rw}$ \eqref{A-infinity map on RC}.
This map has the desired degree $1 - k - l$.

To get the map with values in the desired Hochschild cochain complex, 
we take the direct product over all testing objects $L'_{0}, \cdots, L'_{l}$, and over all $l \ge 0$,
and note that the $l$-th fold tensor product $CW^{*}(L'_{1}, L'_{0}; H) \otimes \cdots \otimes CW^{*}(L'_{l}, L'_{l-1}; H)$
is precisely the morphism spaces in the opposite category
\begin{equation}
\w^{op}(L'_{0}, \ldots, L'_{l}) = \w(L'_{l}, \cdots, L'_{0}) = CW^{*}(L'_{1}, L'_{0}; H) \otimes \cdots \otimes CW^{*}(L'_{l}, L'_{l-1}; H).
\end{equation}
This completes the construction of the map \eqref{phi+k}.

\subsection{The negative part}\label{negpart}

The way of defining $\Phi_{-}$ is slightly different from that of $\Phi_{+}$, 
as it is part of the $A_{\infty}$-functor that lands in the degree one shift of a complex.
This degree one shift suggests that it should be defined as a count of broken popsicles with two disk components.

Let $L_{0}, \cdots, L_{k}$ be $k+1$ objects of $\rw$. 
We shall construct a sequence of maps
\begin{equation}\label{phi-k}
\Phi^{k}_{-}: RC^{*}(L_{k-1}, L_{k}) \otimes \cdots \otimes RC^{*}(L_{0}, L_{1}) \to \r{CC}^{*}(\w^{op}, (Y_{L_{0}}^{r})^{\vee} \otimes_{\K} Y_{L_{k}}^{r})[1]
\end{equation}
of degree $1 - k$. To the unshifted complex, this map has degree $2 - k$.

Let $L'_{0}, \cdots, L'_{l}$ be $l+1$ testing objects. 
Similar to $\Phi_{+}^{k, l}$, we are going to define the map by components
\begin{equation}\label{components of phi-}
\begin{split}
 \Phi_{-}^{k, l} (c_{k} \otimes \cdots \otimes c_{1}): & CW^{*}(L'_{1}, L'_{0}; H) \otimes \cdots \otimes CW^{*}(L'_{l}, L'_{l-1}; H)  \\
 & \to \hom_{\K}(CW^{*}(L'_{0}, L_{0}; H), CF^{*}(L'_{l}, L_{k}; H))[1],
\end{split}
\end{equation}
on generators $c_{i} = [x_{i}]$ represented by either a time-one chord $x_{i}$ from $L_{i-1}$ to $L_{i}$,
or a time-one chord from $L_{i}$ to $L_{i-1}$.

Consider the compactified moduli space of popsicles:
\begin{equation}
\bar{\mathcal{R}}^{k+l+2, \mathbf{p}_{+}, \mathbf{w}_{+}}(y'_{out}, \mathbf{x}', y_{in}, \mathbf{x}),
\end{equation}
in which there are codimension boundary strata consisting of the union of the following products:
\begin{equation}\label{broken popsicle moduli space for phi-}
\begin{split}
& \partial_{1} \bar{\mathcal{R}}^{k+l+2, \mathbf{p}_{+}, \mathbf{w}_{+}}(y'_{out}, \mathbf{x}', y_{in}, \mathbf{x}) \\
=&  \coprod_{y'_{I, out} = y_{II, in}} \mathcal{R}^{k_{1} + l_{1}+2, \mathbf{p}_{I}, \mathbf{w}_{I}}(y'_{I, out}, \mathbf{x}'_{I}, y_{in}, \mathbf{x}_{I}) \times \mathcal{R}^{k_{2}+k_{2}+2, \mathbf{p}_{II}, \mathbf{w}_{II}}(y'_{out}, \mathbf{x}'_{II}, y_{II, in}, \mathbf{x}_{II}).
\end{split}
\end{equation}
for all $k_{1}, k_{2}, l_{1}, l_{2}$ such that $k_{1} + k_{2} = k$ and $l_{1} + l_{2} = l$.
These are moduli spaces of broken popsicles with two disk components,
where the $0$-th puncture $z'_{I, out}$ of the first disk component is glued to the $(l_{2}+1)$-th puncture $z_{II, in}$ of the second disk component.
The induced popsicle structures $\sigma_{I}, \sigma_{II}$ of flavors $\mathbf{p}_{I}, \mathbf{p}_{II}$ and weights $\mathbf{w}_{I}, \mathbf{w}_{II}$ are similar to those of flavors $\mathbf{p}_{+}$ and weights $\mathbf{w}_{+}$,
except that we must require that the weight assigned to the $0$-th puncture $z'_{I, out}$ of the first disk component be $-1$,
\begin{equation}\label{intermediate weight -1}
w'_{I, out} = -1,
\end{equation}
such that the number of sprinkles in the popsicle structure $\sigma_{I}$ of flavor $\mathbf{p}_{I}$ goes down by $1$ compared to popsicle structures $\sigma_{+}$ of flavor $\mathbf{p}_{+}$ with `output' weight $w'_{out} = 0$.

We are mainly interested in the case where the virtual dimension of \eqref{broken popsicle moduli space for phi-} is zero,
which holds if the degrees of the chords satisfy:
\begin{equation}\label{degree condition for broken count}
\deg(y'_{out}) = \deg(y_{in}) + \sum_{i=1}^{k} (-1)^{\d_{i}} \deg(x_{i}) + \sum_{j=1}^{l} (-1)^{\d'_{j}} \deg(x'_{j}) + 2 - k - l -(n+1) \sum_{i=1}^{k} \d_{i}.
\end{equation}
Note that in particular this is independent of the `intermediate' chord $y'_{I, out} = y_{II, in}$,
as a result of canceling the common term $\deg(x_{new})$ in the dimension formulas for the two components.
However, we have implicitly assumed that Floer data are chosen generically so that all the relevant moduli spaces are regular,
in which case we only take those products such that all the moduli spaces in the right hand side of \eqref{broken popsicle moduli space for phi-} are zero-dimensional.
Each rigid broken popsicle $\mathbf{u}$ in \eqref{broken popsicle moduli space for phi-} induces an isomorphism of orientation lines,
by composing the isomorphisms induced by the two disk components:
\begin{equation}
\partial_{1} \bar{\mathcal{R}}^{k+l+2, \mathbf{p}_{+}, \mathbf{w}_{+}}_{\mathbf{u}}:  o_{\mathbf{x}}^{\mathbf{s}} \otimes o_{y_{in}} \otimes o_{\mathbf{x}'} \to o_{y'_{out}}.
\end{equation}
Summing over all $\mathbf{u}$ and then over all possible outputs $y'_{out}, \mathbf{x}'_{out}$ as well as `intermediate' chords $y'_{I, out} = y_{II, in}$ subject to \eqref{degree condition for broken count},
with the signs determined by the induced boundary orientations on the product moduli space \eqref{broken popsicle moduli space for phi-},
we obtain a map
\begin{equation}\label{components of phi-kl}
\begin{split}
\Phi_{-}^{k, l; \mathbf{p}_{+}, \mathbf{w}_{+}}: & CW^{*}(L_{k-1}, L_{k}; s_{k}H)\otimes \cdots \otimes CW^{*}(L_{0}, L_{1}; s_{1}H) \\
 \to & \hom_{\K}(CW^{*}(L'_{1}, L'_{0}; H) \otimes \cdots \otimes CW^{*}(L'_{l}, L'_{l-1}; H), \\
 & \hom_{\K}(CW^{*}(L'_{0}, L_{0}; H), CW^{*}(L'_{l}, L_{k}; H))),
\end{split}
\end{equation}
which has degree $2 - k - l - \sum_{i=1}^{k} \d_{i}$.
The well-definedness of this map follows from the same argument as that of Lemma \ref{lem: phi+ well-defined}.

Finally, we follow the same way as in the case of $\Phi_{+}^{k, l}$ to extend these maps to maps on Rabinowitz complexes:
\begin{equation}\label{phi-kl}
\begin{split}
\Phi_{-}^{k, l}: & RC^{*}(L_{k-1}, L_{k}) \otimes \cdots \otimes RC^{*}(L_{0}, L_{1})\\
\to & \hom_{\K}(CW^{*}(L'_{1}, L'_{0}; H) \otimes \cdots \otimes CW^{*}(L'_{l}, L'_{l-1}; H), \\
&\hom_{\K}(CW^{*}(L'_{0}, L_{0}; H), CW^{*}(L'_{l}, L_{k}; H)))[1].
\end{split}
\end{equation}
This has degree $1 - k - l$ because we have shifted the degree of the target cochain complex by $1$.
A priori, we only obtain a map with values in Hochschild cochains with coefficients in the space of $\K$-linear homomorphisms 
\[
\hom_{\K}(CW^{*}(L'_{0}, L_{0}; H), CW^{*}(L'_{l}, L_{k}; H)),
\]
To obtain the correct map $\Phi_{-}^{k}$ that takes value in the desired tensor product bimodule $(Y^{r}_{K})^{*} \otimes_{\K} Y^{r}_{L}$, 
we must verify that the map defined above indeed lands in the homomorphism space of finite rank,
\begin{equation}
\hom_{\K}^{fin}(CW^{*}(L'_{0}, L_{0}; H), CW^{*}(L'_{l}, L_{k}; H)) \cong (CW^{*}(L'_{0}, L_{0}; H))^{\vee} \otimes_{\K} CW^{*}(L'_{l}, L_{k}; H).
\end{equation}
To see why this is the case, note that the map $\Phi_{-}^{k, l}$ is defined by counting broken popsicles in the moduli spaces \eqref{broken popsicle moduli space for phi-}.
Thus it suffices to show that these moduli spaces are empty except for finitely many pairs of inputs and outputs.

\begin{lem}\label{finite-rank hom}
	For fixed chords $x_{1}, \ldots, x_{k}$ and $x'_{1}, \ldots, x'_{l}$, the moduli spaces \eqref{broken popsicle moduli space for phi-} are all empty except for finitely many pairs of chords $(y_{in}, y'_{out})$.
\end{lem}
\begin{proof}
	Consider the second component of a broken popsicle in the moduli space \eqref{broken popsicle moduli space for phi-}.
Note that $y_{II, in}$ is required to be a time-one $H$-chord but regarded as an output (despite its notation) for the second disk component because of the constraint \eqref{intermediate weight -1}.
Since the other chords $\mathbf{x}_{II}, \mathbf{x}'_{II}$ are fixed, 
the total action of $y_{II, in}$ and $y'_{out}$, both of which are outputs for the second disk component,
 has an upper bound. 
This implies that there are finitely many possible pairs $(y_{II, in}, y'_{out})$ for which the moduli space is non-empty.
For each $y_{II, in} = y'_{I, out}$, as an input for the first disk component, there are finitely many possible $y_{in}$ for which the moduli space is non-empty, as all the other chords $\mathbf{x}_{I}, \mathbf{x}'_{I}$ are fixed.
\end{proof}

It follows from Lemma \ref{finite-rank hom} that the map \eqref{phi-kl} indeed takes values in the finite-rank $\hom$-space, or equivalently the tensor product $(CW^{*}(L'_{l}, L_{0}))^{\vee} \otimes_{\K} CW^{*}(L'_{0}, L_{k})$,
\begin{equation}
\begin{split}
\Phi^{k, l}_{-}: & RC^{*}(L_{0}, L_{1}) \otimes \cdots \otimes RC^{*}(L_{k-1}, L_{k})\\
\to & \hom_{\K}(CW^{*}(L'_{1}, L'_{0}; H) \otimes \cdots \otimes CW^{*}(L'_{l}, L'_{l-1}; H), (CW^{*}(L'_{0}, L_{0}; H))^{\vee} \otimes_{\K} CW^{*}(L'_{l}, L_{k}; H))[1].
\end{split}
\end{equation}
Taking the direct product over all $l$-tuples of Lagrangians $L'_{0}, \cdots, L'_{l}$ and all $l \ge 0$, 
we get the desired map \eqref{phi-k}.

\subsection{\texorpdfstring{$\ainf$}{A-infinity} functor equations} \label{functorequations}

The last step in the construction of the functor $\Phi$ is to verify that the sequence of maps $\Phi^k: = (\Phi_{-}^{k}, \Phi_{+}^{k})$ satisfy the $A_{\infty}$-functor equations.
Recall that the formal punctured neighborhood at infinity  $\winf$ is in fact a dg category, 
which implies that all higher $A_{\infty}$-products of order greater than or equal to three vanish. 
So the equations will only involve $\mu^{1}$ and $\mu^{2}$ products of outcomes of $\Phi$.
Recall also the definition of $\mu^2$ in $\winf$ given in \eqref{mu2 in cinf}.

\begin{lem}\label{lemma: equations for phi+}
	The maps $\Phi_{+}^{k, l; \mathbf{p}_{+}, \mathbf{w}_{+}}$ and $\Phi_{-}^{k, l; \mathbf{p}_{+}, \mathbf{w}_{+}}$ satisfy the following equations
\begin{equation}\label{equations for phi+}
\begin{split}
& \sum_{i, j} (-1)^{*_{i}} \Phi_{+}^{k-j, l; \mathbf{p}_{+, i, j}, \mathbf{w}_{+, i, j}}([x_{k}^{s_{k}}], \ldots, [x_{i+j+1}^{s_{i+j+1}}],
 \mu^{j, \mathbf{p}_{i, j}, \mathbf{w}_{i, j}} ([x_{i+j}^{s_{i+j}}], \ldots, [x_{i+1}^{s_{i+1}}]), \\
 & [x_{i}^{s_{i}}], \ldots, [x_{1}^{s_{1}}])([x'_{1}], \ldots, [x'_{l}]) ([y_{in}]) \\
+ & \sum_{i', j'} (-1)^{*_{i'}} \Phi_{+}^{k, l-j'; \mathbf{p}'_{+, i', j'}, \mathbf{w}'_{+. i', j'}}([x_{k}^{s_{k}}], \ldots, [x_{1}^{s_{1}}])([x'_{1}], \ldots, [x'_{i'}], 
 \mu^{j'} ([x'_{i'+1}], \ldots, [x'_{i'+j'}]), \\
& [x'_{i'+j'+1}], \ldots, [x'_{l}])([y_{in}]) \\
+ & \sum_{\substack{k_{1} + k_{2} = k\\ l_{1} + l_{2} = l}} 
[\Phi_{+}^{k_{2}, l_{2}; \mathbf{p}_{+, II}, \mathbf{w}_{+, II}}([x_{k}^{s_{k}}], \ldots, [x_{k_{1}+1}^{s_{k_{1}+1}}]) ([x'_{l_{1}+1}], \ldots, [x'_{l}]) \\
& \circ \Phi_{+}^{k_{1}, l_{1}; \mathbf{p}_{+, I}, \mathbf{w}_{+, I}}([x_{k_{1}}^{s_{k_{1}}}], \ldots, [x_{1}^{s_{1}}]) ([x'_{1}], \ldots, [x'_{l_{1}}])] ([y_{in}])\\
+ & \Phi_{-}^{k, l; \mathbf{p}_{+}, \mathbf{w}_{+}}([x_{k}^{s_{k}}], \ldots, [x_{1}^{s_{1}}])([x'_{1}], \ldots, [x'_{l}])([y_{in}]) = 0,
\end{split}
\end{equation}
and
\begin{equation}\label{equations for phi-}
\begin{split}
& \sum_{i, j} (-1)^{*_{i}} \Phi_{-}^{k-j, l; \mathbf{p}_{+, i, j}, \mathbf{w}_{+, i, j}}([x_{k}^{s_{k}}], \ldots, [x_{i+j+1}^{s_{i+j+1}}], 
 \mu^{j, \mathbf{p}_{i, j}, \mathbf{w}_{i, j}}([x_{i+j}^{s_{i+j}}], \ldots, [x_{i+1}^{s_{i+1}}]),\\
& [x_{i}^{s_{i}}], \ldots, [x_{1}^{s_{1}}])([x'_{1}], \ldots, [x'_{l}] ([y_{in}]) \\
+ & \sum_{i', j'}  (-1)^{*_{i'}} \Phi_{-}^{k, l-j'; \mathbf{p}'_{+, i', j'}, \mathbf{w}'_{+, i', j'}}([x_{k}^{s_{k}}], \ldots, [x_{1}^{s_{1}}])([x'_{1}], \ldots, [x'_{i'}] ,
 \mu^{j'} ([x'_{i'+1}], \ldots, [x'_{i'+j'}]), \\
& [x'_{i'+j'+1}], \ldots, [x'_{l}])([y_{in}]) \\
+ & \sum_{\substack{k_{1} + k_{2} = k\\ l_{1} + l_{2} = l}}
[\Phi_{-}^{k_{2}, l_{2}; \mathbf{p}_{+, II}, \mathbf{w}_{+, II}}([x_{k}^{s_{k}}], \ldots, [x_{k_{1}+1}^{s_{k_{1}+1}}]) ([x'_{l_{1}+1}], \ldots, [x'_{l}]) \\
& \circ \Phi_{+}^{k_{1}, l_{1}; \mathbf{p}_{+, I}, \mathbf{p}_{+, I}}([x_{k_{1}}^{s_{k_{1}}}], \ldots, [x_{1}^{s_{1}}])([x'_{1}], \ldots, [x'_{l_{1}}])] ([y_{in}])\\
+ &  \sum_{\substack{k_{1} + k_{2} = k\\ l_{1} + l_{2} = l}}
[\Phi_{+}^{k_{2}, l_{2}; \mathbf{p}_{+, II}, \mathbf{w}_{+, II}}([x_{k}^{s_{k}}], \ldots, [x_{k_{1}+1}^{s_{k_{1}+1}}]) ([x'_{l_{1}+1}], \ldots, [x'_{l}]) \\
& \circ \Phi_{-}^{k_{1}, l_{1}; \mathbf{p}_{+, I}, \mathbf{p}_{+, I}}([x_{k_{1}}^{s_{k_{1}}}], \ldots, [x_{1}^{s_{1}}])([x'_{1}], \ldots, [x'_{l_{1}}])] ([y_{in}])
 = 0.
\end{split}
\end{equation}
Here the signs are defined as in \eqref{koszulsign}.

The popsicle structure $\sigma'_{+}$ of flavor $\mathbf{p}'_{+}$ is such that there are fewer punctures $z'_{j}$,
but there are no effects on the number of sprinkles since there are no sprinkles on geodesics connecting them to the puncture $z'_{out}$.
\end{lem}
\begin{proof}
	This is a direct consequence of the study of certain boundary strata of the compactification of the moduli space of popsicles
\[
\bar{\mathcal{R}}^{k+l+2, \mathbf{p}_{+}, \mathbf{w}_{+}}.
\]
The first equation follows by looking at all the codimension-one boundary strata of the moduli space of dimension one.

The second equation follows by looking at the codimension-one boundary strata of codimension-one boundary stratum
\[
\partial_{1} \bar{\mathcal{R}}^{k+l+2, \mathbf{p}_{+}, \mathbf{w}_{+}}(y'_{out}, \mathbf{x}', y_{in}, \mathbf{x}),
\]
inside the moduli space of dimension two.
\end{proof}

	The above two lemmas give preliminary versions of the $\ainf$-functor equations. 
Taking the direct product over all $L'_{0}, \ldots, L'_{l}$ for all $l \ge 1$, 
we obtain the following:
	
\begin{thm}[Theorem \ref{thm:main}(i)] \label{thmfunctor}
    The sequence of maps $\Phi^k  = (\Phi^{k}_{-}, \Phi^{k}_{+})$ form an $\ainf$-functor
\begin{equation*}
\Phi: \rw \to \winf.
\end{equation*}
Concretely, this means that the collection $\{(\Phi^{k}_{-}, \Phi^{k}_{+})\}_{k \geq 1}$ satisfy the following equations:
\begin{equation}\label{master equation for phi+}
\begin{split}
& \mu^{1}_{\winf} \circ \Phi^{k}_{+}(c_{k}, \ldots, c_{1})  
+ \mu^{2}_{\winf}(\Phi^{k_{2}}_{+}(c_{k}, \ldots, \mathbf{x}_{k_{1}+1}), \Phi^{k_{1}}_{+}(c_{k_{1}}, \ldots, c_{1})) \\
= & \sum (-1)^{*_{i}} \Phi^{k}_{+}(c_{k}, \ldots, c_{i+j+1}, \mu^{j}_{\rw}(c_{i+j}, \ldots, c_{i+1}), \mathbf{x}_{i}, \ldots, c_{1}) + \Phi^{k}_{-}(c_{k}, \ldots, c_{1}),
\end{split}
\end{equation}
and
\begin{equation}\label{master equation for phi-}
\begin{split}
& \mu^{1}_{\winf} \circ \Phi^{k}_{-}(c_{k}, \ldots, c_{1}) 
+ \mu^{2}_{\winf}(\Phi^{k_{2}}_{+}(c_{k}, \ldots, c_{k_{1}+1}), \Phi^{k_{1}}_{-}(c_{k_{1}}, \ldots, c_{1}))\\
& + \mu^{2}_{\winf}(\Phi^{k_{2}}_{-}(c_{k}, \ldots, c_{k_{1}+1}), \Phi^{k_{1}}_{+}(c_{k_{1}}, \ldots, c_{1})) \\
= &  \sum (-1)^{*_{i}} \Phi^{k}_{-} (c_{k}, \ldots, c_{i+j+1}, \mu^{j}_{\rw}(c_{i+j}, \ldots, c_{i+1}), c_{i}, \ldots, c_{1}).
\end{split}
\end{equation}
for every generator $c_{i} \in RC^{*}(L_{i-1}, L_{i})$, where the signs are defined as in \eqref{koszulsign}.
\end{thm}
\begin{proof}

By \eqref{mu2 in cinf}, the structure maps $\mu^{1}_{\winf}$ and $\mu^{2}_{\winf}$ are given by combinations of $A_{\infty}$-structure maps $\mu_{\w}$ of the wrapped Fukaya category,
and composition of linear maps in the $\hom$-space.
More precisely, $\mu^{1}_{\winf}$ is the cone differential of two Hochschild cochain complexes,
which involves all those operations $\mu^{j'}$ appearing in the second term of each of \eqref{equations for phi+} and \eqref{equations for phi-}.
And $\mu^{2}_{\winf}$ is induced by the composition in $\hom_{\K}(\w^{op}_{\D}, \w^{op}_{\D})$ and the action of $\hom_{\K}(\w^{op}_{\D}, \w^{op}_{\D})$ on $(\w^{op})^{\vee} \otimes_{\K} \w^{op}_{\D}$.
For the two Hochschild cochains 
\[
\Phi^{k_{1}}_{+}(c_{k_{1}}, \ldots, c_{1})
\]
 and 
 \[
 \Phi^{k_{2}}_{+}(c_{k}, \ldots, c_{k_{1}+1}),
 \]
their product 
\[
\mu^{2}_{\winf}(\Phi^{k_{2}}_{+}(c_{k}, \ldots, c_{k_{1}+1}), \Phi^{k_{1}}_{+}(c_{k_{1}}, \ldots, c_{1}))
\]
 is defined to be the Hochschild cochain whose $l$-th order component acting on 
 \[
 [x'_{1}] \otimes \cdots \otimes [x'_{l}] \in CW^{*}(L'_{1}, L'_{0}; H) \otimes \cdots \otimes CW^{*}(L'_{l}, L'_{l-1}; H)
 \]
 is the composition of linear maps that takes an input $[y_{in}]$ to the following output:
\[
\begin{split}
& \sum \Phi_{+}^{k_{2}, l_{2}; \mathbf{p}_{+, II}, \mathbf{w}_{+, II}}([x_{k}^{s_{k}}], \ldots, [x_{k_{1}+1}^{s_{k_{1}+1}}]) ([x'_{l_{1}+1}], \ldots, [x'_{l}]) \\
& \circ \Phi_{+}^{k_{1}, l_{1}; \mathbf{p}_{+, I}, \mathbf{w}_{+, I}}([x_{k_{1}}^{s_{k_{1}}}], \ldots, [x_{1}^{s_{1}}]) ([x'_{1}], \ldots, [x'_{l_{1}}])] ([y_{in}]),
\end{split}
\]
where is sum is taken over all $l_{1}, l_{2}$ such that $ l_{1} + l_{2} = l$.

With these observations, if we restrict the equations \eqref{master equation for phi+} and \eqref{master equation for phi-} to their $l$-th order terms of the Hochschild cochains, 
and evaluate the Hochschild cochains on the corresponding wrapped Floer cochain spaces, 
we get exactly \eqref{equations for phi+} and \eqref{equations for phi-}.
\end{proof}

\subsection{Statement of Equivalence}\label{statementofequivalence}

While for any Liouville manifold there is such an $\ainf$-functor
\[
\Phi: \rw \to \winf,
\]
we can only expect it to have best features when the wrapped Fukaya category $\w$ does not behave too wildly.
A natural condition is that the Liouville manifold $(X, \lambda)$ has `enough Lagrangians',
which can be formally stated as the `non-degeneracy' condition, in the sense of \cite{ganatra}.
Let us first recall the definition.

\begin{defn}\label{defn:nondegenerate}
A Liouville manifold $(X, \lambda)$ is said to be non-degenerate,
if it admits a finite collection of Lagrangians $\{L_{i}\}$ which are objects of $\w$,
such that the open-closed map 
\begin{equation}
\mathcal{OC}: \r{HH}_{*-n}(\w(X), \w(X)) \to SH^{*}(X)
\end{equation}
restricted to the subcategory of the wrapped Fukaya category having $\{L_{i}\}$ as objects hits the identity element in symplectic cohomology.
\end{defn}

The main result regarding the functor $\Phi$ is the following:

\begin{thm}[Theorem \ref{thm:main}(ii)] \label{thm:equivalence}
	Suppose the Liouville manifold $(X, \lambda)$ is non-degenerate.
Then the $\ainf$-functor
\[
\Phi: \rw \to \winf
\]
is a quasi-equivalence of $\ainf$-categories.
\end{thm}

It suffices to check that the linear term $\Phi^{1}$ induces an isomorphism on cohomology,
\[
	H^{*}(\Phi^{1}): H^{*}(\rw(K, L)) = H^{*}(RC^{*}(K, L), \mu^{1}_{RC}) \to H^{*}(\winf(K, L)).
\]
The cohomology of $\winf$ is computed to be the cohomology of the cone of Hochschild cohomology cochain complexes:
\[
	H^{*}(\winf(K, L)) = H^{*}(\cone(\r{CC}^{*}(\w^{op}, (Y_{K}^{r})^{*} \otimes_{\K} Y_{L}^{r}) \to \r{CC}^{*}(\w^{op}, \hom_{\K}(Y_{K}^{r}, Y_{L}^{r}))).
\]
Our strategy is to compute the cohomology of $\rw$ by a cone complex of $\ainf$-bimodules,  
and construct a quasi-isomorphism between the two cones.
In fact, $\rw$ as a bimodule over $\w$ (or rather $\w^{op}$) via the natural functor $\w \to \rw$, is quasi-equivalent to the cone of bimodules
\[
\cone(\r{CC}^{*}(\w^{op}, (\w^{op})^{\vee} \otimes_{\K} \w^{op}_{\D}) \to \r{CC}^{*}(\w^{op}, \hom_{\K}(\w^{op}_{\D}, \w^{op}_{\D}))
\]
which is briefly mentioned at the beginning of \S\ref{section: functor},
although for our purpose of proving Theorem \ref{thm:equivalence} we do not need this generality.
In the remaining subsections, 
we shall construct the various $\ainf$-bimodule maps involved in comparing the two cones,
culminating (in \S \ref{section: homotopy argument}) in the proof of Theorem \ref{thm:equivalence}.

\subsection{The cohomology of $\rw$ as a bimodule}

By definition, the Rabinowitz complex \eqref{Rabinowitz complex} is a cone complex.
This can be upgraded to the statement that $\rw$, as a bimodule over $\w^{op}$, is a cone of a morphism between certain $\ainf$-bimodules.
To describe this cone, we shall introduce an $\ainf$-bimodule over $\w^{op}$ with underlying chain complexes being $CW^{*}(K, L; -H)$, denoted by
\begin{equation}
	\w_{-},
\end{equation}
called the {\em negative bimodule} of $\w^{op}$.
For each pair of objects $(K, L) \in \ob \w^{op} \times \ob \w^{op} = \ob \w \times \ob \w$, define
\begin{equation}
	\w_{-}(K, L) := CW^{*}(K, L; -H).
\end{equation}
The $\ainf$-bimodule structure maps
\begin{equation}\label{w- structure maps}
\begin{split}
\mu^{k, l}_{\w_{-}}: & CW^{*}(K_{1}, K_{0}; H) \otimes \cdots \otimes CW^{*}(K_{k}, K_{k-1}; H)   \otimes  \w_{-}(K_{0}, L_{0})\\
& \otimes CW^{*}(L_{l-1}, L_{l}; H) \otimes \cdots \otimes CW^{*}(L_{0}, L_{1}; H) \to  \w_{-}(K_{k}, L_{l})   
\end{split}
\end{equation}
are defined by counting rigid elements in the moduli spaces of popsicles $\mathcal{R}^{d+1, \mathbf{p}, \mathbf{w}}(\mathbf{x})$ \eqref{popsicle moduli space}
with several constraints on the popsicle flavors $\mathbf{p}$, the weights $\mathbf{w}$ and the chords $\mathbf{x}$,
plus a suitable sign twisted.
(The order of tensor products does not appears in the usual way because here we are defining $\w_{-}$ as a bimodule over the opposite category $\w^{op}$.)
More concretely, the bimodule structure maps are defined as follows.
Let $x_{i}$ be a time-one $H$-chord from $K_{i}$ to $K_{i-1}$,
$x'_{j}$ a time-one $H$-chord from $L_{j-1}$ to $L_{j}$,
 $y$ a time-one $H$-chord from $L_{0}$ to $K_{0}$,
 $z$ a time-one $H$-chord from $L_{l}$ to $K_{k}$.
Consider the moduli space of popsicles
\[
\mathcal{R}^{k+l+2, \varnothing, \mathbf{w}_{-, k, l}}(z, \mathbf{x}, y, \mathbf{x}'),
\]
introduced in \eqref{w- disks}.
Counting rigid elements in the zero-dimensional moduli spaces, with the following sign specification, yields a map
\begin{equation}\label{w- kl}
\begin{split}
\mu^{k, l}_{\w_{-}}: & CW^{*}(K_{1}, K_{0}; H) \otimes \cdots \otimes CW^{*}(K_{k}, K_{k-1}; H)   \otimes  CW^{*}(K_{0}, L_{0}; -H)\\
& \otimes CW^{*}(L_{l-1}, L_{l}; H) \otimes \cdots \otimes CW^{*}(L_{0}, L_{1}; H) \to  CW^{*}(K_{k}, L_{l}; -H), \\
\mu^{k, l}([\mathbf{x}] \otimes [y^{-}] \otimes [\mathbf{x}']) = & \sum_{\substack{y\\ \deg(y) = \deg(z) + \sum_{i=1}^{k} \deg(x_{i}) + \sum_{j=1}^{l} \deg(x'_{j}) + 1 - k - l}} \\
& \sum_{u \in \mathcal{R}^{k+l+2, \varnothing, \mathbf{w}_{-, k, l}}(z, \mathbf{x}, y, \mathbf{x}')} 
(-1)^{\dagger_{k, l}} \mathcal{R}^{k+l+2, \varnothing, \mathbf{w}_{-}}_{u}([\mathbf{x}] \otimes [y^{-}] \otimes [\mathbf{x}']),
\end{split}
\end{equation}
where the sign is
\begin{equation}\label{bimodule sign}
\dagger_{k, l} = \sum_{i=1}^{k} i \deg(x_{k-i+1}) + k \deg(y) + \sum_{j=1}^{l} (k+j) \deg(x'_{j}).
\end{equation}
Well-definedness of the maps \eqref{w- kl} follows from a proof similar to those of Lemma \ref{d- well-defined} and Lemma \ref{mud well-defined}.

A standard argument by looking at the boundary strata of the one-dimensional moduli spaces implies the following:

\begin{lem}
The maps $\mu^{k, l}_{\w_{-}}$ satisfy the equations for an $\ainf$-bimodule structure.
\end{lem}

Recall that there is a continuation map \eqref{continuation} with source being $CW^{*}(K, L; -H)$,
the underlying cochain complex for the negative bimodule $\w_{-}$.
This continuation map can be upgraded to a map of $\ainf$-bimodules.
The {\it diagonal bimodule}
\begin{equation}
	\w_{\D}
\end{equation}
of $\w$ is the $\ainf$-bimodule over $(\w, \w)$ whose underlying cochain complexes are
\begin{equation}
	\w_{\D}(K, L) = \w(L, K) = CW^{*}(L, K; H),
\end{equation}
and whose structure maps are given by the $\ainf$-structure maps of $\w$,
i.e. given by counting rigid elements in the moduli space
\[
\mathcal{R}^{k+l+2}(y, \mathbf{x}', z, \mathbf{x}),
\]
with a sign twist in the form of \eqref{bimodule sign}.
Similarly, the diagonal module $\w^{op}_{\D}$ of the opposite category $\w^{op}$ has cochain spaces being the wrapped Floer cochain spaces in the 'correct' direction,
\begin{equation}
	\w^{op}_{\D}(K, L) = CW^{*}(K, L; H).
\end{equation}

\begin{lem}
	The continuation map \eqref{continuation} extends to a natural (closed) morphism of $\ainf$-bimodules
\begin{equation}
	\mathfrak{c}: \w_{-} \to \w^{op}_{\D}.
\end{equation}
\end{lem}
\begin{proof}
	The construction of the $\ainf$-bimodule morphism uses the moduli spaces of popsicles
\begin{equation}
\mathcal{R}^{k+l+2, \mathbf{p}_{-+}, \mathbf{w}_{-+}}(z, \mathbf{x}, y, \mathbf{x}'),
\end{equation}
where the popsicle structure $\sigma_{-+}$ of flavor $\mathbf{p}_{-+}$ carries exactly one sprinkle on the geodesic connecting the $0$-th puncture $z'_{out}$ to the $(l+1)$-th puncture $z_{in}$,
and the weights are
\begin{equation}
\mathbf{w}_{-+} = (0, \underbrace{0, \ldots, 0}_{l \text{ times }}, -1, \underbrace{0, \ldots, 0}_{k \text{ times }}).
\end{equation}
{\it A priori}, counting rigid elements defines a pre-morphism of $\ainf$-bimodules, 
and closedness follows from an analysis of the codimension-one boundary strata of one-dimensional moduli spaces of the same kind.
\end{proof}

The mapping cone of the morphism $\mathfrak{c}$ is an $\ainf$-bimodule over $\w^{op}$,
\begin{equation}
\cone(\mathfrak{c}).
\end{equation}
Our definition and computation immediately implies that its cohomology agree with the cohomology of $\rw$.

\begin{lem}\label{rw as cone from w- to w}
	There is a natural quasi-isomorphism of complexes
\begin{equation}\label{map from cone to rw}
	\iota: \cone(\mathfrak{c})(K, L) \to RC^{*}(K, L).
\end{equation}
\end{lem}
\begin{proof}
On the level of underlying chain complexes,
both sides of \eqref{map from cone to rw} are the following Floer complexes
\[
CW^{*}(K, L; -H)[1] \oplus CW^{*}(K, L; H).
\]
We define $\iota$ to be the identity map.
The differentials can also be identified,
as we can make the same choices of Floer data for the strip $Z$,
as well as for the popsicle with two positive punctures carrying a sprinkle on the geodesic connecting the punctures,
in both the setups of $\w_{-}$ and $\rw$.
\end{proof}

\begin{rem}
In fact, if we regard $\rw$ as a bimodule over $\w^{op}$ via the natural functor $\w \to \rw$,
then one can further upgrade \eqref{map from cone to rw} to a quasi-isomorphism of bimodules over $\w^{op}$.
\end{rem}

\subsection{Bimodule Poincar\'{e} duality}\label{section: Poincare duality}

Wrapped Floer homology is known to linear dual to wrapped Floer cohomology over a field $\K$, see for example \cite{CO}.
Combined with Poincar\'{e} duality, this says that the isomorphism \eqref{PD} is compatible with differentials,
which we have verified in the construction of the differential on $CW^{*}(K, L; -H)$.
The isomorphism $\bar{I}$ \eqref{PD} can be further upgraded to a quasi-isomorphism of $\ainf$-bimodules from $\w_{-}$ to the linear dual diagonal bimodule $(\w^{op})^{\vee}$ (Definition \ref{def: c dual}).

The linear duality between wrapped Floer homology (computed as the cohomology of $-H$) and cohomology can be rephrased in terms of an equivalence of bimodules:

\begin{prop}\label{equivalence between w* and w-}
	There is a quasi-isomorphism of $\ainf$-bimodules
\begin{equation}\label{map from w- to w*}
    I: \w_{-} \stackrel{\sim}{\to} (\w^{op})^{\vee}[-n]
\end{equation}
The grading on $(\w^{op})^{\vee}$ is such that 
\[
(\w^{op})^{\vee}[-n](K, L) = \hom_{\K}(CW^{-(*-n)}(L, K), \K) = \hom_{\K}(CW^{n-*}(L, K), \K).
\]
\end{prop}

While the $(0, 0)$-th order term is given by the map $\bar{I}$ \eqref{PD},
we want to give it a different interpretation in terms of disk counts,
so that the construction is naturally generalized to higher order terms by inserting more punctures on the boundary.
The relevant moduli spaces have been introduced in subsection \ref{section: pairing disks}.
First, consider the moduli space
\[
\mathcal{R}_{p}(x_{-}, x_{+})
\]
defined in \eqref{moduli space of pairing disks}.
Counting rigid elements in zero-dimensional moduli spaces, with the sign twist $(-1)^{\deg(x_{+})}$, yields a map
\begin{equation}
\hat{I}^{0, 0}: CW^{*}(L, K; H) \to CW^{*}(L, K; H).
\end{equation}
Using the isomorphism $\bar{I}$ \eqref{PD}, this is equivalent to a map
\begin{equation}\label{I_w as count of pairing disks}
I^{0, 0}: CW^{*}(K, L; -H) \otimes CW^{n-*}(L, K; H) \to \K,
\end{equation}
which by adjunction is in turn equivalent to
\begin{equation}
I^{0, 0}: CW^{*}(K, L; -H) \to \hom_{\K}(CW^{n-*}(L, K; H), \K),
\end{equation}
by abuse of notation.

\begin{lem}
	For sufficiently small perturbations of the one-parameter family $J_{t}$ as part of the Floer data,
the map $\hat{I}^{0, 0}$ is the identity map,
so that the map $\bar{I}$ agrees with the map $I^{0, 0}$. 
\end{lem}
\begin{proof}
	This follows immediately from Lemma \ref{the only pairing disk}.
\end{proof}

For $k, l \ge 0$, consider the moduli spaces
\[
\mathcal{R}_{p}^{k, l}(\mathbf{x}, \mathbf{x}'; x_{-}, x_{+})
\]
defined in \eqref{moduli space of pairing disks with marked points}.
When the virtual dimension of the moduli space is zero, i.e. when
\[
\deg(x_{-}) = \deg(x_{+}) + \sum_{i=1}^{k} \deg(x_{i}) + \sum_{j=1}^{l} \deg(x'_{j}) - k - l,
\]
we define a map
\begin{equation}\label{I map}
\begin{split}
I^{k, l}: & CW^{*}(K_{1}, K_{0}) \otimes \cdots \otimes CW^{*}(K_{k}, K_{k-1}) \otimes \w_{-}(K_{0}, L_{0})\\
& \otimes CW^{*}(L_{l-1}, L_{l}) \otimes \cdots \otimes CW^{*}(L_{0}, L_{1}) \to \hom_{\K}(CW^{n + k + l -*}(L_{l}, K_{k}), \K),
\end{split}
\end{equation}
by its action on basis elements:
\begin{equation}\label{formula for the I map}
\begin{split}
& I^{k, l}([\mathbf{x}] \otimes [x_{-}^{-}] \otimes [\mathbf{x}']) \\
= & \sum_{\substack{x_{-} \\ \deg(x_{-}) = \deg(x_{+}) - \sum_{i=1}^{k} \deg(x_{i}) - \sum_{j=1}^{l} \deg(x'_{j}) - k - l}} 
  \sum_{u \in \mathcal{R}_{p}^{k, l}(\mathbf{x}, \mathbf{x}'; x_{-}, x_{+})} (-1)^{\dagger_{k, l}} (\mathcal{R}_{p}^{k, l})_{u}([\mathbf{x}] \otimes [x_{-}^{-}] \otimes [\mathbf{x}']),
 \end{split}
\end{equation}
where the sign is \eqref{bimodule sign}.
Again, we use an action filtration argument to show that the operation \eqref{formula for the I map} extended to a well-defined map \eqref{I map}.

\begin{lem}\label{lemminusmorphism}
The maps $I^{k, l}$ form a morphism of $\ainf$-bimodules
\[
I: \w_{-} \to (\w^{op})^{\vee}[-n].
\]
\end{lem}
\begin{proof}
This follows Lemma \ref{boundary of moduli space of pairing disks with marked points}.
The first line on the right hand side of \eqref{boundary strata of moduli space of pairing disks with marked points} contributes to the $\ainf$-bimodule structure maps of $\w_{-}$.
The second line contributes to the $\ainf$-bimodule structure maps of $(\w^{op})^{\vee}[-n]$.
The third and the last lines contribute to the extra terms 
corresponding to the action of the $\ainf$-structure maps of $\w^{op}$
on the bimodules $\w_{-}$ and $(\w^{op})^{\vee}[-n]$.
The fact that all these moduli spaces on the right hand side of \eqref{boundary strata of moduli space of pairing disks with marked points} appear in the boundary strata of some moduli space of one dimension higher
implies that the sum of these algebraic terms is zero. 
It follows that the maps $I^{k, l}$ form a pre-morphism of $\ainf$-bimodules,
which is closed under the differential on the category of $\ainf$-bimodules.
That is, $I$ is a bimodule morphism.
\end{proof}

\begin{proof}[Proof of Proposition \ref{equivalence between w* and w-}]
    Lemma \ref{lemminusmorphism} constructs the desired bimodule morphism $I$.
Since the $(0, 0)$-th order term $I^{0, 0}$ is an isomorphism of complexes,
it follows that $I$ is a quasi-isomorphism.
\end{proof}

\subsection{The Yoneda map}\label{positiveyoneda}

Recall that the morphism space in the category of right modules over $\w$ can be computed by the following Hochschild cohomology cochain complex
\begin{equation} 
\hom_{\fun(\w^{op}, \ch_\K)}(Y_{K}^{r}, Y_{L}^{r}) = \r{CC}^{*}(\w^{op}, \hom_{\K}(Y_{K}^{r}, Y_{L}^{r})).
\end{equation}
Recall that there is a cohomologically full and faithful Yoneda embedding (see e.g., \cite{seidel_book})
\[
y: \w \to \fun(\w^{op}, \ch_{\K})
\]
whose associated chain map on morphism spaces, for each pair of objects $(K,L)$
\begin{equation}\label{yoneda map for w+}
y_{+}: \w(K, L) = CW^{*}(K, L; H) \to \r{CC}^{*}(\w^{op}, \hom_{\K}(Y_{K}^{r}, Y_{L}^{r}))
\end{equation}
is given in terms of the $\ainf$-structure maps of the wrapped Fukaya category $\w$ with a suitable sign twist:
\begin{equation}
    y_{+}(b)(c'_{1}, \ldots, c'_{l})(a) = (-1)^{\deg(b) + \sum_{j=1}^{l} \deg(c'_{j}) - l - 1} \mu^{l+2}_{\w}(b, a ,c'_{1}, \ldots, c'_{l}).
\end{equation}

Consider the following composition of maps
\begin{equation}\label{composition of phi+}
\w^{op}_{\D}(K, L) = CW^{*}(K, L; H) \stackrel{\iota_{++}} \to RC^{*}(K, L) \stackrel{\Phi_{+}^{1}} \to \r{CC}^{*}(\w^{op}, \hom_{\K}(Y_{K}^{r}, Y_{L}^{r})),
\end{equation}
where the first map 
\begin{equation}
\iota_{++}: CW^{*}(K, L; H) \to RC^{*}(K, L)
\end{equation}
is the component of the quasi-isomorphism \eqref{map from cone to rw} with image contained in $CW^{*}(K, L; H)$,
which is the identity map on $CW^{*}(K, L; H)$,
and $\Phi_{+}^{1}$ is defined in \eqref{phi+k}, for $k = 1$.
The following observation and corollary constitute essentially half
of the proof of Theorem \ref{thm:equivalence}.  
\begin{lem}\label{phi+ = yoneda}
The map \eqref{composition of phi+} agrees with the Yoneda map $y_{+}$ \eqref{yoneda map for w+}.
\end{lem}
Since the Yoneda embedding is cohomologically fully faithful, it follows from this Lemma that
\begin{cor}\label{plusquasi}
    The map \eqref{composition of phi+} is a quasi-isomorphism. \qed
\end{cor}
\begin{proof}[Proof of Lemma \ref{phi+ = yoneda}]
	The composition of $\Phi_{+}^{1}$ with $\iota_{++}$ is a map which is defined by counting rigid elements in the moduli space of popsicles $\mathcal{R}^{l+3, \mathbf{p}_{+}, \mathbf{w}_{+}}(y'_{out}, \mathbf{x}', y_{in}, x)$,
such that the weight $w = w_{1}$ at the puncture $z = z_{1}$ is required to be $0$,
so that all the weights are $0$.
This condition implies that the popsicle structure does not carry any sprinkle,
and that the moduli space is just the moduli space of the usual inhomogeneous pseudoholomorphic disks that are counted in the definition of the $\ainf$-structure maps of the wrapped Fukaya category.
\end{proof}

\begin{cor}\label{compatibility with w to rw}
The composition $\Phi \circ j_{\rw}$ of $j_{\rw}: \w \to \rw$ \eqref{w to rw} with $\Phi: \rw \to \winf$ \eqref{eq:rwtowinf} agrees with the induced Yoneda functor $\bar{y}: \w \to \winf$, up to homotopy.
\end{cor}
\begin{proof}
Since the only nonzero terms of $j_{\rw}$ are the linear terms, this immediately follows from Lemma \ref{phi+ = yoneda}.
\end{proof}

\subsection{The smooth Calabi-Yau structure}\label{section: inverse dualizing bimodule}

The second remaining key step of the proof of Theorem \ref{thm:equivalence} is to show that the other composition 
\begin{equation}\label{composition of phi-}
\w_{-}(K, L)[1] \stackrel{\iota_{--}} \to RC^{*}(K, L) \stackrel{\Phi_{-}^{1}}\to \r{CC}^{*}(\w^{op}, (Y^{r}_{K})^{*} \otimes_{\K} Y^{r}_{L})[1].
\end{equation}
is also a quasi-isomorphism,
where
\begin{equation}
\iota_{--}: \w_{-}(K, L)[1] = CW^{*}(K, L; -H)[1] \to RC^{*}(K, L)
\end{equation}
is the component of the quasi-isomorphism \eqref{map from cone to rw} landing in $CW^{*}(K, L; -H)[1]$,
which is by definition the inclusion on $CW^{*}(K, L; -H)[1]$.
To prove that this composition \eqref{composition of phi-} is a quasi-isomorphism,
we need the {\it non-degeneracy} condition on the Liouville manifold $X$, 
and the proof will involve a full package of algebraic structures related to the {\it (weak) smooth Calabi-Yau structure} of $\w(X)$, proven to exist in \cite{ganatra}.

Recall from \eqref{inverse dualizing bimodule} that the inverse dualizing bimodule of $\w$ is by definition the bimodule dual of the diagonal bimodule $\w_{\D}$ of $\w$,
\begin{equation}\label{wshriek}
	\w^{!} := \hom_{\w-\w}(\w_{\D}, \w_{\D} \otimes_{\K} \w_{\D}),
\end{equation}
In this paper, we will take a different but quasi-equivalent model of $\w^{!}$ as the following Hochschild complex:
\begin{equation}\label{smallwshriek}
\r{CC}^{*}(\w, \w_{\D} \otimes_{\K} \w_{\D}).
\end{equation}

Counting rigid elements in the moduli space
\[
\mathcal{R}_{2; (0, 0)}^{1, l; r, s}(\mathbf{x}^{+, 1}, x, \mathbf{x}^{+, 2}, x_{-, 2}, \mathbf{x}', x_{-, 1})
\]
given in \eqref{moduli space of cy disks with marked points} in  \S\ref{section: disks with two outputs},
we obtain maps of the following form
\begin{equation}\label{components of cy}
\begin{split}
\mathcal{CY}^{l; r, s}: & CW^{*}(K_{1}, K_{0}; H) \otimes \cdots \otimes CW^{*}(K_{r}, K_{r-1}; H) \otimes \w^{op}_{\D}(K_{0}, L_{0}) \\
& \otimes CW^{*}(L_{s-1}, L_{s}; H) \otimes \cdots \otimes CW^{*}(L_{0}, L_{1}; H) \\
&\to \hom_{\K}(CW^{*}(L'_{1}, L'_{0}; H) \otimes \cdots \otimes CW^{*}(L'_{l}, L'_{l-1}; H), \w^{op}_{\D}(K_{r}, L'_{0}) \otimes_{\K} \w^{op}_{\D}(L'_{l}, L_{s}))[n],
\end{split}
\end{equation}
by the formula
\begin{equation}\label{formula for cy}
\begin{split}
&\mathcal{CY}^{l; r, s}([\mathbf{x}^{+, 1}] \otimes [x] \otimes [\mathbf{x}^{+, 2}])([\mathbf{x}']) \\
= & \sum_{\substack{x_{-, 1}, x_{-, 2}\\ \deg }} \\
& \sum_{u \in \mathcal{R}_{2; (0, 0)}^{1, l; r, s}(\mathbf{x}^{+, 1}, x, \mathbf{x}^{+, 2}, x_{-, 2}, \mathbf{x}', x_{-, 1})} (-1)^{\dagger_{l, r, s}} (\mathcal{R}_{2; (0, 0)}^{1, l; r, s})_{u}([\mathbf{x}^{+, 1}] \otimes [x] \otimes [\mathbf{x}^{+, 2}] \otimes [x']).
\end{split}
\end{equation}
where the sign is
\begin{equation}\label{sign for cy}
\dagger_{l, r, s} = \sum_{h = 1}^{l} h \deg(x'_{l-h+1}) \sum_{i=1}^{r} i \deg(x^{+, 1}_{r - i + 1}) + r \deg(x) + \sum_{j=1}^{s} (r+j) \deg(x^{+, 2}_{j})
\end{equation}
for all $r, s, l \ge 0$ and all possible testing objects $L'_{1}, \ldots, L'_{l}$.
Taking the direct over all possible testing objects $L'_{0}, \ldots, L'_{l}$ for all $l \ge 0$, 
we get $(r, s)$-term of a pre-morphism of $\ainf$-bimodules.
Then it is a book-keeping exercise,
by looking at the codimension-one boundary strata of the relevant moduli spaces,
to check that this pre-morphism is closed, i.e. a morphism of bimodules.

\begin{lem}
The maps $\mathcal{CY}^{l; r, s}$ form a morphism of $\ainf$-bimodules:
\begin{equation}\label{cy Hochschild}
    \mathcal{CY}: \w^{op}_{\D} \to \r{CC}^{*}(\w^{op}, \w^{op}_{\D} \otimes_{\K} \w^{op}_{\D})[n].
\end{equation}
\qed
\end{lem}

One of the main results of \cite{ganatra} implies that:
\begin{prop}[essentially \cite{ganatra}, Theorem 1.3]\label{prop:cy qi}
	Suppose the Liouville manifold $(X, \lambda)$ is non-degenerate.
Then the morphism $\mathcal{CY}$ \eqref{cy Hochschild} is a quasi-isomorphism of $\ainf$-bimodules.
\end{prop}
\begin{proof}[Sketch]
    It suffices to show that the underlying map of chain complexes 
    \[
        \mathcal{CY}_{K,L}: \hom(K,L) \to \r{CC}^{*}(\w, Y^l_K \otimes_{\K} Y^r_L)[n]
    \]
    is a quasi-isomorphism for each pair of objects $K,L$.
    We will sketch how the result from \cite{ganatra}*{Thm. 1.3}, which involves slightly different moduli spaces, implies the desired result. The reference
    {\em loc. cit.} proves that for the wrapped Fukaya category $\w$ of a non-degenerate Liouville manifold,
    there is a quasi-isomorphism of bimodules 
\begin{equation}\label{calabiyaustr}
    { }_2\mathcal{CY}: \w_{\D} \stackrel{\sim}\to \w^{!}[n],
\end{equation}
(called $\mathcal{CY}$ in {\em loc. cit.}) where $\w^!$ is the complex \eqref{wshriek} rather than \eqref{smallwshriek}. Let's call the resulting map of chain complexes ${ }_2\mathcal{CY}_{K,L}$ for each pair of objects $K,L$. The map ${ }_2\mathcal{CY}$ is defined by counting moduli spaces of discs with four distinguished points of fixed cross ratio, alternating input and output, along with arbitrary many other inputs in between; for ${ }_2\mathcal{CY}_{K,L}$ there are no additional inputs before or after one of the distinguished input points. Now there is a quasi-isomorphism of chain complexes $\mathbf{\Psi}_{K,L}: \r{CC}^{*}(\w, Y^l_K \otimes_{\K} Y^r_L) \stackrel{\sim}{\to} \w^!(K,L)$ (see e.g., \cite{ganatra}*{Prop. 2.5}). Finally a variation on the argument of \cite{ganatra}*{Prop. 5.6} which follows by considering the moduli spaces associated to ${ }_2\mathcal{CY}_{K,L}$ and degenerating the fixed cross ratio of the four special points to $\infty$, implies that there is a chain homotopy ${ }_2 \mathcal{CY}_{K,L} \simeq \mathbf{\Psi}_{K,L} \circ \mathcal{CY}_{K,L}$. Since $\mathbf{\Psi}_{K,L}$ and ${ }_2 \mathcal{CY}_{K,L}$ are quasi-isomorphisms, the result follows.
\end{proof}

We refer to the above as the fact that $\w$ has a {\it smooth Calabi-Yau structure}, 
and call this map the {\it Calabi-Yau map}.

\subsection{A different Calabi-Yau-type map}\label{section: cy-}

There is a variant of the Calabi-Yau map \eqref{calabiyaustr},
by considering the moduli spaces
\[
\mathcal{R}_{2; (-1, -1)}^{1, l; r, s}(\mathbf{x}^{+, 1}, x_{+}, \mathbf{x}^{+, 2}, x_{-, 2}, \mathbf{x}', x_{-, 1})
\]
as in \eqref{moduli space of cy- disks with marked points}. 
The outcome of counting rigid elements is a map of the following form
\begin{equation}\label{cy- bimodules}
\mathcal{CY}_{-}: \w_{-} \to \r{CC}^{*}(\w^{op}, \w_{-} \otimes_{\K} \w^{op}_{\D})[n]
\end{equation}
as a map of bimodules, 
and is specialized to the following map of complexes
\begin{equation}\label{cy- complexes}
\mathcal{CY}_{-}: \w_{-}(K, L) \to \r{CC}^{*}(\w^{op}, \w_{-}(K, \cdot) \otimes_{\K} Y^{r}_{L})[n]
\end{equation}
as the $(0, 0)$-th term of the bimodule morphism.
Here the coefficient of the Hochschild cochains is the bimodule given as
the algebraic tensor product of the right module $\w_{-}(K, \cdot)$ with the left module $Y^{r}_{L}$, both as modules over $\w^{op}$.

Concretely, the bimodule morphism \eqref{cy- bimodules} has components
\begin{equation}
\begin{split}
\mathcal{CY}_{-}^{l; r, s}: & CW^{*}(K_{1}, K_{0}; H) \otimes \cdots \otimes CW^{*}(K_{r}, K_{r-1}; H) \otimes \w_{-}(K_{0}, L_{0}) \\
& \otimes CW^{*}(L_{s-1}, L_{s}; H) \otimes \cdots \otimes CW^{*}(L_{0}, L_{1}; H) \\ 
& \to \hom_{\K}(CW^{*}(L'_{1}, L'_{0}; H) \otimes \cdots \otimes CW^{*}(L'_{l}, L'_{l-1}; H), \w_{-}(K_{r}, L'_{l}) \otimes_{\K} \w^{op}_{\D}(L'_{1}, L_{s})),
\end{split}
\end{equation}
which is defined by counting rigid elements in the moduli space
\[
\mathcal{R}_{2; (-1, -1)}^{1, l; r, s}(\mathbf{x}^{+, 1}, x, \mathbf{x}^{+, 2}, x_{-, 2}, \mathbf{x}', x_{-, 1}),
\]
similar to the construction of the maps \eqref{components of cy}.
The formula follows the similar pattern as \eqref{formula for cy},
and the sign is the same as \eqref{sign for cy}.
The $(0, 0)$-th term, i.e. the chain map \eqref{cy- complexes} is the map defined by counting rigid elements in the moduli spaces 
\[
\mathcal{R}_{2; (-1, -1)}^{1, l}(x, x_{-, 2}, \mathbf{x}', x_{-, 1}).
\]
 
The main result concerning this map $\mathcal{CY}_{-}$ is similar to Proposition \ref{prop:cy qi};
that is, under non-degeneracy hypotheses $\mathcal{CY}_{-}$ is also a quasi-isomorphism of $\ainf$-bimodules:

\begin{prop}\label{prop:cy- qi}
Suppose the Liouville manifold $(X, \lambda)$ is non-degenerate. Then
the map 
\[
\mathcal{CY}_{-}: \w_{-} \to \r{CC}^{*}(\w^{op}, \w_{-} \otimes_{\K} \w^{op}_{\D})[n]
\]
 is a quasi-isomorphism of bimodules.
\end{prop}
The idea of this proof will be to reduce to Proposition \ref{prop:cy qi} by a degeneration argument which relates $\mathcal{CY}_-$ to $\mathcal{CY}$ followed by multiplication in the bimodule $\w_-$ (as operations on the bimodule $\w^{op}_{\D} \otimes_{\w^{op}} \w_{-}$ that is quasi-isomorphic to $\w_{-}$). In order to state the needed commutative diagram precisely, we will define several chain maps.

First we define the map
\begin{equation}\label{cytensorid}
    \mathcal{CY} \otimes \id_{\w_{-}}: \w^{op}_{\D} \otimes_{\w^{op}} \w_{-} \to \r{CC}^{*}(\w^{op}, \w^{op}_{\D} \otimes_{\K} \w^{op}_{\D})[n] \otimes_{\w} \w_{-}
\end{equation}
to be the one-sided tensor product over $\w^{op}$ of the Calabi-Yau map $\mathcal{CY}$ \eqref{calabiyaustr} with the identity map of the negative bimodule $\w_{-}$.
The domain of \eqref{cytensorid}
is quasi-equivalent to the negative bimodule $\w_{-}$ by the collapse map \eqref{collapse map}.

Under the stated hypotheses $\w$ is smooth \cite{ganatra}*{Thm 1.3}; hence
by Proposition \ref{prop: bring in tensor product} there is a canonical quasi-isomorphism
\begin{equation}\label{bringinquasi}
    \r{CC}^{*}(\w^{op}, \w^{op}_{\D} \otimes_{\K} \w^{op}_{\D})[n] \otimes_{\w} \w_{-} \stackrel{\cong}{\to} \r{CC}^{*}(\w^{op}, (\w^{op}_{\D} \otimes_{\w^{op}} \w_{-}) \otimes_{\K} \w^{op}_{\D})[n].
\end{equation}

Now recall that for any bimodule $\mathcal{B}$ over $\w^{op}$,
the collapse map \eqref{collapse map} 
\[
\mu_{\D}: \w^{op}_{\D} \otimes_{\w^{op}} \mathcal{B} \to \mathcal{B}
\]
is a quasi-isomorphism of bimodules, by Lemma \ref{lemcollapse}.
This can be regarded as a bimodule version of Proposition 2.2 of \cite{ganatra}.

In cases where $\mathcal{B}$ is defined geometrically by disk counts,
 such that the cochain space $\mathcal{B}(K, L)$ is a Floer complex,
the collapse map is also geoemtrically defined as the bimodule structure maps for $\mathcal{B}$; see
\eqref{collapseformula}.
In particular, this holds for $\mathcal{B} = \w_{-}$,
and the relevant moduli spaces are 
\[
\mathcal{R}^{k+l+2, \varnothing, \mathbf{w}_{-}}(y, \mathbf{x}', z, \mathbf{x})
\]
as in \eqref{w- disks},
i.e. the moduli spaces of popsicles defining the bimodule structure maps $\mu^{k, l}_{\w_{-}}$ \eqref{w- structure maps} for $\w_{-}$.

Composing with the collapse map for the negative bimodule $\w_{-}$ also induces a quasi-isomorphism on the Hochschild cochain complexes:
\begin{equation}\label{composewithcollapse}
    \r{CC}^{*}(\w^{op}, (\w^{op}_{\D} \otimes_{\w^{op}} \w_{-}) \otimes_{\K} \w^{op}_{\D})[n] \stackrel{\cong}{\to} \r{CC}^{*}(\w^{op}, \w_{-} \otimes_{\K} \w^{op}_{\D})[n].
\end{equation}
 The composition of \eqref{composewithcollapse} with \eqref{bringinquasi} and \eqref{cytensorid} defines a map from $\w^{op}_{\D} \otimes_{\w^{op}} \w_{-}$ to $\r{CC}^{*}(\w^{op}, \w_{-} \otimes_{\K} \w^{op}_{\D})[n]$.

On the other hand, the composition of the collapse map for $\w_{-}$ with $\mathcal{CY}_{-}$ 
\[
\w^{op}_{\D} \otimes_{\w^{op}} \w_{-}  \stackrel{\mu_{\D}}\to \w_{-} \stackrel{\mathcal{CY}_{-}}\to \r{CC}^{*}(\w^{op}, \w_{-} \otimes_{\K} \w^{op}_{\D})[n]
\]
gives another morphism from $\w^{op}_{\D} \otimes_{\w^{op}} \w_{-}$ to $\r{CC}^{*}(\w^{op}, \w_{-} \otimes_{\K} \w^{op}_{\D})[n]$.

After specializing to the underlying complexes of the bimodules for each pair of objects $(K, L)$, the above discussion produces two chain maps from $\w^{op}_{\D} \otimes_{\w^{op}} \w_{-} (K, L)$ to $\r{CC}^{*}(\w^{op}, \w_{-}(K, \cdot) \otimes_{\K} Y^{r}_{L})[n]$. The commutative diagram below, 
which makes precise the relationship needed between $\mathcal{CY}$ and $\mathcal{CY}_-$, asserts that these two chain maps are chain homotopic:
\begin{prop}\label{prop: comparing cy to cy-}
The following diagram commutes up to chain homotopy: 
\begin{equation}
\begin{tikzcd}
    \w^{op}_{\D} \otimes_{\w^{op}} \w_{-}(K, L) \arrow[r, "\mathcal{CY} \otimes \id_{\w_{-}}"] \arrow[dd, "\mu_{\D}"] & \r{CC}^{*}(\w^{op}, \w^{op}_{\D} \otimes_{\K} \w^{op}_{\D})[n] \otimes_{\w^{op}} \w_{-}(K, L) \arrow[d, "\eqref{bringinquasi}"] \\
& \r{CC}^{*}(\w^{op}, (\w^{op}_{\D} \otimes_{\w^{op}} \w_{-}) \otimes_{\K} \w^{op}_{\D})(K, L)[n] \arrow[d, "(\mu_{\D} \otimes \id_{\w^{op}_{\D}})_{*}"] \\
\w_{-}(K, L) \arrow[r, "\mathcal{CY}_{-}"] & \r{CC}^{*}(\w^{op}, \w_{-}(K, \cdot) \otimes_{\K} Y^{r}_{L})[n]
\end{tikzcd}
\end{equation}
\end{prop}
Assuming this Proposition for a moment, we complete the proof of Proposition \ref{prop:cy- qi}:

\begin{proof}[Proof of Proposition \ref{prop:cy- qi}]
It suffices to prove that $\mathcal{CY}_{-}$ induces a quasi-isomorphism of chain complexes for each pair of objects $(K, L)$.
Consider the commutative diagram stated in Proposition \ref{prop: comparing cy to cy-}. The two vertical maps are always quasi-isomorphisms by the discussion above Proposition \ref{prop: comparing cy to cy-}, and Proposition \ref{prop:cy qi} implies that the top horizontal map is a quasi-isomorphism under the stated hypotheses. It follows that (under the same hypotheses) the remaining bottom horizontal map,
$\mathcal{CY}_{-}$  is also a quasi-isomorphism.
\end{proof}

It remains only to establish the needed commutative diagram:
\begin{proof}[Proof of Proposition \ref{prop: comparing cy to cy-}]

Consider the moduli space
\begin{equation}
\mathcal{R}_{2; (-1, -1)}^{1, l; 0, s+1}(x_{+}, \tilde{\mathbf{x}}^{+, 2}, x_{-, 2}, \mathbf{x}', x_{-, 1})
\end{equation}
as special cases of \eqref{moduli space of cy- disks with marked points} where $r = 0$ and $s$ is replaced by $s+1$.
Here we think of $z^{+, 2}_{s+1}$ as an extra distinguished input and set
\begin{equation}
\tilde{\mathbf{x}}^{+, 2} = (\mathbf{x}^{+, 2}, x_{0}),
\end{equation}
where $x_{0}$ is a Floer cochain for the diagonal bimodule $\w^{op}_{\D}$.
Counting rigid elements in the above moduli space,
with the sign $\dagger_{l, 0, s+1}$ as in \eqref{sign for cy} with $r=0$ and $s$ replaced by $s+1$,
yields a sequence of multilinear maps between various tensor products of Floer cochain spaces,
which can be collected as the data for a map of the following form:
\begin{equation}
S: \w^{op}_{\D} \otimes_{\w^{op}} \w_{-} (K, L) \to \r{CC}^{*}(\w^{op}, \w_{-}(K, \cdot) \otimes_{\K} Y^{r}_{L})[n-1].
\end{equation}

To see that it is the desired chain homotopy, 
we study the boundary strata of the relevant moduli spaces by looking at the right hand side of \eqref{boundary strata of moduli space of cy- disks} in Lemma \ref{boundary of moduli space of cy- disks} (see Figure \ref{fig:boundary of moduli of cy disks} for the pictures),
with $r=0$ and $s$ replaced by $s+1$ such that the last auxiliary input $z^{+,2}_{s+1}$ becomes a distinguished input for the diagonal bimodule $\w^{op}_{\D}$.
\begin{itemize}

\item Since $r = 0$, the product moduli spaces of the third type in \eqref{boundary strata of moduli space of cy- disks} (middle left in Figure \ref{fig:boundary of moduli of cy disks}) are all empty,
and the first (top left) and the second (top right) contribute to parts of $d_{\r{CC}^{*}} \circ S$ and $S \circ \mu^{0, 0}_{\w^{op}_{\D} \otimes_{\w^{op}} \w_{-}}$, respectively.

\item The fourth type of boundary strata in \eqref{boundary strata of moduli space of cy- disks} (middle right picture in Figure \ref{fig:boundary of moduli of cy disks}) contributes to part of $d_{\r{CC}^{*}} \circ S$ as well, 
by the definition of the differential on the Hochschild cochain complex $\r{CC}^{*}(\w^{op}, \w_{-}(K, \cdot) \otimes_{\K} Y^{r}_{L})[n]$.

\item The fifth type of boundary strata in \eqref{boundary strata of moduli space of cy- disks} (bottom left of Figure \ref{fig:boundary of moduli of cy disks} with $s_{1}=s+1$, in which case the node is an output for the {\it top} disk component) again contributes to part of $d_{\r{CC}^{*}} \circ S$,
since the top disk component is of type $\mathcal{R}_{2; (-1, -1)}^{1, l_{1}; 0, s+1}$ defining the chain homotopy $S$,
and the bottom disk component is an ordinary $\ainf$-disk which contributes to the differential on the Hochschild cochain complex.

\item For a broken disk of the six type in \eqref{boundary strata of moduli space of cy- disks} (still bottom left of Figure \ref{fig:boundary of moduli of cy disks} but with $s_{1}<s+1$, in which case the node is an output for the {\it bottom} disk component), 
the bottom component carries the one extra distinguished input $z^{+,2}_{s+1}$ and two distinguished outputs, 
and therefore contributes to $\mathcal{CY}$.
Since $r=0$, and the top disk component carries one distinguished negatively weighted input and one output, 
it follows that this disk component contributes to the induced map by $\mu_{\D} \otimes \id_{\w^{op}_{\D}}$ on the Hochschild cochain complex.
By the pattern how these two disk components are glued together,
it follows that such a broken disk contributes to $(\mu_{\D} \otimes \id_{\w^{op}_{\D}})_{*} \circ (\mathcal{CY} \otimes \id_{\w_{-}})$.

\item Finally, for a broken disk of the seventh i.e., last type in \eqref{boundary strata of moduli space of cy- disks} (bottom right of Figure \ref{fig:boundary of moduli of cy disks}),
the left component contributes to $\mu_{\D}$ since $r=0$ and the distinguished input and output carry negative weights,
and the right component contributes to $\mathcal{CY}_{-}$.

\end{itemize}
This proves (modulo sign checking by comparing \eqref{sign formula 1 for popsicles} and \eqref{sign formula 2 for popsicles} and \eqref{sign for cy}) that
\begin{equation}
(\mu_{\D} \otimes \id_{\w^{op}_{\D}})_{*} \circ (\mathcal{CY} \otimes \id_{\w_{-}}) - \mathcal{CY}_{-} \circ \mu_{\D} = d_{\r{CC}^{*}} \circ S + S \circ \mu^{0, 0}_{\w^{op}_{\D} \otimes_{\w^{op}} \w_{-}},
\end{equation}
as desired.
\end{proof}

\subsection{A homotopy argument}\label{section: homotopy argument}

The key step in proving Theorem \ref{thm:equivalence} is to observe that
the broken popsicles needed for the construction of the negative part $\Phi_{-}$ of the functor $\Phi$ 
can be compared to other types of broken popsicles described in \S\ref{section: disks with two outputs},
 via a gluing-degeneration argument.
This subsection translates that geometric argument into the algebraic outcome,
which establishes the needed relation between $\mathcal{CY}_{-}$ and the composition
\[
\w_{-}(K, L)[1] \stackrel{\iota_{--}} \to RC^{*}(K, L) \stackrel{\Phi_{-}^{1}}\to \r{CC}^{*}(\w^{op}, (Y^{r}_{K})^{*} \otimes_{\K} Y^{r}_{L})[1].
\]

Recall that Proposition \ref{equivalence between w* and w-} asserts that the bimodules $\w_{-}$ and $(\w^{op})^{\vee}[-n]$ are quasi-equivalent by the map $I$ \eqref{map from w- to w*}.
The quasi-isomorphism $I$ induces, by evaluating on the left at the object $K$,  a quasi-isomorphism of modules
$I_{K,-}: \w_{-}(K, \cdot) \stackrel{\sim}{\to} (Y^{r}_{K})^{*}$ (whose higher-order terms are $I^{l, 0}$), and hence for any $L$ a quasi-isomorphism of bimodules $I_{K,-} \otimes id_{Y^{r}_{L}}: \w_{-}(K, \cdot) \otimes_{\K} Y^{r}_{L} \stackrel{\sim}{\to} (Y^{r}_{K})^{*} \otimes_{\K} Y^{r}_{L}$. There is therefore an induced quasi-isomorphism on the associated Hochschild cochain complexes
\begin{equation}\label{induced map on Hochschild complexes}
    I_{*}: \r{CC}^{*}(\w^{op}, \w_{-}(K, \cdot) \otimes_{\K} Y^{r}_{L})[n] \stackrel{\sim} \to \r{CC}^{*}(\w^{op}, (Y^{r}_{K})^{*} \otimes_{\K} Y^{r}_{L}).
\end{equation}
By the definition of the pushforward map between Hochschild co-chain complexes associated to a morphism of bimodules we obtain a formula for $I_*$:

\begin{lem}\label{formula for I*}
For a Hochschild cochain $\phi \in \r{CC}^{*}(\w^{op}, \w_{-}(K, \cdot) \otimes_{\K} Y^{r}_{L})[n]$, and each testing input $c_{+} \in Y^{r}_{K}(L'_{l})$, the output $I_{*}(\phi) \in \r{CC}^{*}(\w^{op}, (Y^{r}_{K})^{*} \otimes_{\K} Y^{r}_{L})$ is a Hochschild cochain whose length $l$ term is given by the following formula
\begin{equation}
    I_{*}(\phi)^{l}(c'_{1} \otimes \cdots \otimes c'_{l})(c_{+}) = \sum_{l_{1}+l_{2}=l} I^{l_{2}, 0}(c'_{1} \otimes \cdots \otimes c'_{l_{2}} \otimes \phi^{l_{1}}(c'_{l_{2}+1} \otimes \cdots \otimes c'_{l}))(c_{+}).
\end{equation}
\qed
\end{lem}

With this formula, we can now study the relation between the two maps $I_{*} \circ \mathcal{CY}_{-}$ and $\Phi_{-}^{1} \circ \iota_{--}$, 
both with source $\w_{-}(K, L)$ and target $\r{CC}^{*}(\w^{op}, (Y^{r}_{K})^{*} \otimes_{\K} Y^{r}_{L})$.
We define a chain homotopy
\begin{equation}\label{chain homotopy between cy- and phi-}
T: \w_{-}(K, L) \to \r{CC}^{*}(\w^{op}, (Y^{r}_{K})^{*} \otimes_{\K} Y^{r}_{L})[-1]
\end{equation}
by counting rigid elements in the moduli spaces 
\[
\mathcal{R}_{1}^{l+3}(x_{out}, \mathbf{x}', x_{+}, x)
\]
as in \eqref{moduli space of glued broken popsicles with an interior marked point} introduced in \S \ref{section: broken popsicles of a different type},
where $x \in \chi(K, L; -H)$,
$x'_{j} \in \chi(L'_{j}, L'_{j-1}; H)$,
 $x_{+} \in \chi(L'_{l}, K; H)$, 
 and $x_{out} \in \chi(L'_{0}, L; H)$.
The formula for the chain homotopy \eqref{chain homotopy between cy- and phi-} takes a form similar to the formulas for $\Phi_{+}$ in \eqref{components of phi+} with $k = 1$,
and but sign for the count is $(-1)^{\Diamond_{l}}$ where
\begin{equation}\label{sign for chain homotopy}
\Diamond_{l} = \sum_{i=1}^{l} i \deg(x'_{l+1-i}) + (l+1) \deg(x_{+}) + (l+2) \deg(x) - 1.
\end{equation}

\begin{lem}\label{cd for cy- and phi-}
The following diagram commutes up to chain homotopy: 
\begin{equation}
\begin{tikzcd}
\w_{-}(K, L) \arrow[r, "\Phi_{-}^{1} \circ \iota_{--}"] \arrow[d, "\mathcal{CY}_{-}"] & \r{CC}^{*}(\w^{op}, (Y^{r}_{K})^{*} \otimes_{\K} Y^{r}_{L}) \arrow[d, "\id"]\\
\r{CC}^{*}(\w^{op}, \w_{-}(K, \cdot) \otimes_{\K} Y^{r}_{L})[n] \arrow[r, "I_{*}"] & \r{CC}^{*}(\w^{op}, (Y^{r}_{K})^{*} \otimes_{\K} Y^{r}_{L}) 
\end{tikzcd}
\end{equation}
with the chain homotopy given by the map $T$ \eqref{chain homotopy between cy- and phi-}.
\end{lem}
\begin{proof}
The proof that $T$ gives the desired chain homotopy is similar to the proof of Proposition \ref{prop: comparing cy to cy-}.
Consider the codimension-one boundary strata \eqref{boundary of one-pointed popsicles defining homotopy} of one-dimensional compactified moduli spaces $\bar{\mathcal{R}}_{1}^{l+3}(x_{out}, \mathbf{x}', x_{+}, x)$ \eqref{compactified moduli space of popsicles with an interior marked point}.
We want to show that each type of boundary stratum contributes to a term in the chain homotopy equation
\begin{equation}
\Phi^{1}_{-} \circ \iota_{--} - I_{*} \circ \mathcal{CY}_{-} = T \circ \mu^{0,0}_{\w_{-}} + d_{\r{CC}^{*}} \circ T,
\end{equation}
where $d_{\r{CC}*}$ standards for the differential on the Hochschild cochain complex $ \r{CC}^{*}(\w^{op}, (Y^{r}_{K})^{*} \otimes_{\K} Y^{r}_{L})$.
Now we match the types of boundary strata with the algebraic terms that they contribute to when rigid elements are counted.
\begin{enumerate}[label=(\roman*)]

\item Product moduli spaces of the first type on the right hand side of \eqref{boundary of one-pointed popsicles defining homotopy} contribute to the term $T \circ \mu^{0,0}_{\w_{-}}$;

\item By the formula of $d_{\r{CC}^{*}}$ on $\r{CC}^{*}(\w^{op}, (Y^{r}_{K})^{*} \otimes_{\K} Y^{r}_{L})$, 
where we identify 
\[
(Y^{r}_{K})^{*} \otimes_{\K} Y^{r}_{L} \cong \hom_{\K}^{fin}(Y^{r}_{K}, Y^{r}_{L}),
\]
product moduli spaces of the second, the third and the sixth types on the right hand side of \eqref{boundary of one-pointed popsicles defining homotopy} all contribute to the term $d_{\r{CC}*} \circ T$;

\item By Lemma \ref{formula for I*} we find that the map $I_{*}$,
is defined by counting rigid elements in the moduli spaces
\[
\mathcal{R}^{l, 0}_{p}(\mathbf{x}; x_{-}, x_{+}), l \ge 0
\]
which are special sub-collections of of \eqref{moduli space of pairing disks with marked points}.
Therefore, the composition $I_{*} \circ \mathcal{CY}_{-}$ is defined by counting rigid broken popsicles in the product moduli space
\[
\mathcal{R}^{1, l_{1}}_{2; (-1, -1)}(x, x_{out}, \mathbf{x}'_{II}, x_{-, 1}) \times \mathcal{R}^{l_{2}, 0}_{p}(\mathbf{x}'_{I}; x_{-}, x_{+})
\]
appearing in the fourth type of boundary stratum on the right hand side of \eqref{boundary of one-pointed popsicles defining homotopy}.

\item Note that the map $\Phi_{-}^{1} \circ \iota_{--}$ is defined by counting rigid broken popsicles in the moduli space of broken popsicles
\[
\partial_{1} \bar{\mathcal{R}}^{k+l+2, \mathbf{p}_{+}, \mathbf{w}_{+}}(x_{out}, \mathbf{x}', x_{+}, x) 
\]
which is a special case of \eqref{broken popsicle moduli space for phi-} where $k = 1$ and $y'_{out}=x_{out}, y_{in}=x_{+}$,
with certain specific weight constraints.
Comparing \eqref{broken popsicle moduli space for phi-} to the moduli space of the fifth type of the right hand side of \eqref{boundary of one-pointed popsicles defining homotopy},
we see that these two product moduli spaces agree.

\end{enumerate}

The last step is to verify the signs are consistent,
which geometrically speaking requires matching the induced boundary orientation on 
\[
\partial_{1} \bar{\mathcal{R}}^{k+l+2, \mathbf{p}_{+}, \mathbf{w}_{+}}(y'_{out}, \mathbf{x}', y_{in}, x_{1})
\]
and the orientation on
\[
\mathcal{R}^{1, l_{1}}_{2; (-1, -1)}(x, x_{-, 2}, \mathbf{x}'_{II}, x_{-, 1}) \times \mathcal{R}^{l_{2}, 0}_{p}(\mathbf{x}_{I}; x_{-}, x_{+})
\]
with the induced boundary orientation on $\p \bar{\mathcal{R}}_{1}^{l+3}(x_{-, 2}, \mathbf{x}', x_{+}, x)$.
In terms of the signs in the relevant algebraic maps, this is a routine computation by comparing \eqref{sign for chain homotopy} with the signs as \eqref{bimodule sign} in the formula for the bimodule map $I$ as in \eqref{formula for the I map}, 
\eqref{sign for cy} in the formula for the map $\mathcal{CY}_{-}$,
and the standard signs for counting broken popsicles determined by the boundary orientations on the moduli space of broken popsicles from moduli space of smooth popsicles,
for which the signs are given in \eqref{sign formula 1 for popsicles} and \eqref{sign formula 2 for popsicles}.
\end{proof}

\begin{proof}[Proof of Theorem \ref{thm:equivalence}/Theorem \ref{thm:main}(ii)]
    By Proposition \ref{prop:cy- qi}, under the stated hypotheses the map $\mathcal{CY}_{-}$ is a quasi-isomorphism. Also, $I_{*}$, the induced map \eqref{induced map on Hochschild complexes} associated to a quasi-isomorphism of bimodules, is a quasi-isomorphism. So the composition is as well, which implies by Lemma \ref{cd for cy- and phi-} that
that $\Phi_{-}^{1} \circ \iota_{--}$ is a quasi-isomorphism.
Recall that Corollary \ref{plusquasi} previously showed that $\Phi_+^{1} \circ \iota_{++}$ was a quasi-isomorphism.

Now, observe that the linear map $\Phi^{1}$ is upper-triangular with respect to the $\{-, +\}$ decompositions of the Rabinowitz Floer complex \eqref{Rabinowitz complex} and that of $\winf$ as in \eqref{winfascone}.
That is, if one restricts $\Phi^{1}$ to the subspace $CW^{*}(K, L; H)$, 
which is also a subcomplex with respect to the Floer differential, 
the image must be contained in 
\[
\r{CC}^{*}(\w^{op}, \hom_{\K}(Y_{K}^{r}, Y_{L}^{r})).
\]
In particular, $\Phi^1$ is a filtered map with respect to the filtration $\{+ \} \subseteq \{+,-\}$ on both sides. Since the associated graded parts of $\Phi^1$, which are $\Phi_{-}^{1} \circ \iota_{--}$ and $\Phi_+^{1} \circ \iota_{++}$, are quasi-isomorphisms, a standard argument shows that $\Phi^1$ is as well.
\end{proof}

\section{Applications and examples} \label{sec:applicationsexamples}

\subsection{Situating Rabinowitz Fukaya categories into mirror symmetry}\label{sec:rabinowitzhms}

The goal of this subsection is to review Efimov's geometric descriptions of the formal punctured neighborhood of infinity of coherent sheaf categories,
and show how these along with our main Theorem \ref{thm:main} lead to a proof of Corollary \ref{hmsrabinowitz}.

Recall that in algebraic geometry, the category $\perf(Y)$ of perfect complexes on a variety $Y$ is proper if and only if $Y$ is proper; in particular, the categorical formal punctured neighborhood $\widehat{\perf(Y)}_{\infty}$ is a category whose non-triviality obstructs properness of $Y$. Work of Efimov \cite{efimov} gives a purely algebro-geometric model of this category if $Y$ is a smooth variety over $\C$, defined as follows:
by Hironaka's resolution of singularities, we can find a smooth compactification $\bar{Y}$ with $D = \bar{Y} - Y$ a normal crossings divisor. Consider first the formal completion of $\bar{Y}$ along $D$, denoted $\hat{\bar{Y}}_D$.
This has a well-defined (dg) category of perfect complexes  $\perf(\hat{\bar{Y}}_D)$, defined as follows.
Let $D_{n}$ be the $n$-th infinitesimal neighborhood of $D$, defined by taking the $n$-th power of the sheaf of ideals $\mathcal{I}_{D}^{n} \subset \mathcal{O}_{\bar{Y}}$.
There are natural inclusion morphisms
\[
i_{n}: D_{n} \to \bar{Y}
\]
and
\[
i_{n, m}: D_{n} \to D_{m}
\]
for $n < m$.
We define $\perf(\hat{\bar{Y}}_D)$ to be the homotopy limit
\begin{equation}
\perf(\hat{\bar{Y}}_D) := \r{holim}_{n} \perf(D_{n}).
\end{equation}
The formal scheme $\hat{\bar{Y}}_D$ itself can be defined in a similar way,
in which the local pieces are given by taking homotopy colimit 
\begin{equation}
\hat{\bar{Y}}_{D}(U) = \r{hocolim}_{n} \r{Spec} \mathcal{O}_{\bar{Y}}(U)/ \mathcal{I}_{D}^{n}(U)
\end{equation}
for any Zaraski open $U \subset \bar{Y}$.
The inclusion $D \hookrightarrow \bar{Y}$ factors through $i: D \hookrightarrow \hat{\bar{Y}}_D$, inducing a pushforward on perfect complexes:
\[
    i_*: \perf(D) \to \perf(\hat{\bar{Y}}_D),
\]
which can also be seen as the homotopy limit of functors
\[
i_{n, *}: \perf(D) \to \perf(D_{n}).
\]
Next one attempts to ``puncture'' (obtaining an algebraic version of a ``punctured tubular neighborhood'') by removing the image of $D$ from $\hat{\bar{Y}}_D$. Although the result isn't well-defined as a formal scheme (it would have no points), Efimov observes that by emulating on the level of perfect complexes the process of puncturing (i.e., taking the quotient of the image), one can associate a well-defined category, called the {\em perfect complexes on the formal punctured neighborhood of infinity of $Y$}, which is in fact only depends on the open locus $Y$:
\begin{equation}\label{efimovgeometricdef}
    \perf(\hat{Y}_{\infty}):= \perf(\hat{\bar{Y}}_D) / i_*(\perf(D)).
\end{equation}
To clarify, we have not defined the space $\hat{Y}_{\infty}$ (which should morally be thought of as the algebro-geometric analogue of a small punctured tubular neighborhood around the locus at infinity in any compactification of $Y$, though this is not well-defined as a formal scheme), simply its associated category directly. There is a natural ``restriction-to-infinity'' functor 
\begin{equation}
    j^*: \perf(Y) \to \perf(\hat{Y}_{\infty}),
\end{equation}
whose essential image, denoted 
\begin{equation}\label{perfalg}
    \perf_{alg}(\hat{Y}_{\infty}):= j^*(\perf(Y))
\end{equation}
is called the category of {\em algebraizable perfect complexes on the formal punctured neighborhood of infinity}. (As an example, for $Y = \mathbb{A}^1_{\mathbb{C}} = \mathrm{Spec}(\C[t])$, the image of $\mathcal{O}_{Y}$ in $\perf(\hat{Y}_{\infty})$ has homological endomorphisms $\C(( t^{-1} ))$). 
Note first that if $Y$ is proper that $\perf(\hat{Y}_{\infty}) = 0$ and more generally that any compactly supported sheaf is sent by $j^*$ to 0.

The first main result of Efimov is that the categorical construction of the formal punctured neighborhood matches with the above algebro-geometric construction.

\begin{thm}[Theorem 3.5 of \cite{efimov}]\label{thm:efimov1}
In the above situation, one has a quasi-equivalence $\perf_{alg}(\hat{Y}_{\infty}) \stackrel{\sim}\to \widehat{\coh(Y)}_{\infty}$.
\end{thm}
\begin{rem}\label{perftopvsperf}
    Efimov \cite{efimov} in fact describes a larger category $\perf_{top}(\hat{Y}_{\infty})$ of perfect complexes on the formal punctured neighborhood of infinity, with the subcategory $\perf_{alg}(\hat{Y}_{\infty})$ defined as the essential image of the restriction from $\perf(Y)$. Although \cite[Thm 3.5]{efimov} concerns the larger category $\perf_{top}(\hat{Y}_{\infty})$, Theorem \ref{thm:efimov1} as stated follows immediately from compatibility of {\em loc. cit.} with restriction from $\perf(Y)$.
\end{rem}

Now consider a proper but not necessarily smooth variety $Y$ over any perfect field $\K$.
Recall that the category of singularities $\sing(Y)$ is defined as the (dg) quotient
\[
\sing(Y) = \coh(Y)/\perf(Y).
\]
In this case, Efimov relates it to the category of algebraizable perfect complexes on the formal punctured neighborhood in the following way:

\begin{thm}[Theorem 0.4 of \cite{efimov}]\label{thm:efimov2}
Let $Y$ be a proper scheme over a perfect field $\K$. 
Then there is a natural quasi-equivalence $\sing(Y)^{op} \stackrel{\sim}\to \widehat{\coh(Y)}_{\infty}$.
\end{thm}

The strategy of the proof of Theorem \ref{thm:efimov2} uses algebraic Koszul duality, 
which we have formulated in \S\ref{section:koszuldual}.
The key observation is that when $Y$ is proper,
the inclusion of the subcategory of {\it dual perfect} complexes, i.e. those complexes $E$ such that $\R\mathcal{H}om_{\mathcal{O}_{Y}}(E, D_{Y})$ where $D_{Y}$ is the dualizing complex, 
 into the coherent sheaf category is a Koszul dualizing subcategory.

 \begin{proof}[Proof of Corollary \ref{hmsrabinowitz}]
     Combine Theorems \ref{thm:main}, \ref{thm:efimov1} and \ref{thm:efimov2}.
 \end{proof}

\subsection{Computations in $T^{*}S^{1}$}

Here we spell out our main result in a simple example in which the category $\rw(X)$ can be computed rather directly: $T^*S^1$.
This subsection will be mostly expository with computations matching with our main result Theorem \ref{thm:main}; 
and we will take $\K = \mathbb{F}_{2}$ to keep the demonstration concise and clean without sign concerns.
Recall that there is a well-known HMS equivalence 
\[
\w(T^*S^1; \mathbb{F}_{2}) \cong Coh(\mathbb{G}_{m}),
\]
where $\mathbb{G}_{m} = \mathrm{Spec}(\mathbb{F}_{2}[t, t^{-1}])$.
  In particular, if $L_0$ is a cotangent fiber and $H = r^2$ is radial in the fiber, there is an isomorphism of $A_{\infty}$ algebras
\[
CF^*(L_0, L_0; H) = \mathbb{F}_{2}[t, t^{-1}],
\]
with $\mu_k = 0$ for all $k \neq 2$.

Efimov's geometric definition of the category $\perf_{alg}(\hat{Y}_{\infty})$ as in \eqref{perfalg} requires the existence of a smooth proper compactification (which holds in general in characteristic zero); once such a compactification exists \cite{efimov}*{Theorem 3.5} (Theorem \ref{thm:efimov1} above plus Remark \ref{perftopvsperf}) compares the geometric definition \eqref{perfalg} to the categorical definition $\widehat{\coh(Y)}_{\infty}$. It follows that the resulting geometric definition given is independent of choice of compactification, provided one exists (note a direct geometric proof of this fact given in \cite{efimov}*{Cor. 2.16} requires one to further know about the existence of ``common refinements'' of any two such choices of compactification, as holds in characteristic zero). Here for the case of  $\mathbb{G}_{m}$ over $\mathbb{F}_2$, we define the category $\perf_{alg}(\widehat{\mathbb{G}_m}_{\infty})$ 
by compactifying $\mathbb{G}_{m}$ to $\mathbb{P}^{1}_{\mathbb{F}_{2}}$. 
Corollary \ref{hmsrabinowitz} of our main theorem implies an equivalence  of $\rw(T^*S^1)$ over $\mathbb{F}_2$ with the image of $\perf(\mathbb{F}_{2}[t,t^{-1}])$ in $\perf(\mathbb{F}_{2}((t^{-1})) \times \mathbb{F}_{2}((t)))$.
We will demonstrate this equivalence directly by explicit computation of the endomorphisms of $L_0$ in $\rw(T^*S^1)$, which corresponds via mirror symmetry to the endormophisms of the image of $\mathcal{O}$ in $\perf_{alg}(\widehat{\mathbb{G}_m}_{\infty})$.

The Rabinowitz complex of $L_0$ is 
\[
RC^*(L_0, L_0) \simeq \prod_{r\in\mathbb{Z}}\mathbb{F}_{2}\langle\phi_r\rangle \oplus \bigoplus_{r\in\mathbb{Z}}\mathbb{F}_{2}\langle \tau^r\rangle
\]
with both $|t^r| = 0$ and $|\phi_r| = 0$.  We will show that there is a (non-canonical!) $A_{\infty}$-algebra isomorphism.
\begin{equation}
\label{eq:ts1-iso}
RC^*(L_0, L_0) \simeq \mathbb{F}_{2}((t^{-1})) \oplus \mathbb{F}_{2}((t)).
\end{equation}
By degree considerations, $\mu_k = 0$ for all $k\neq 2$.  It therefore suffices to show that (\ref{eq:ts1-iso}) is a ring isomorphism.  Working with a single Lagrangian yields mild non-genericity that is unwieldy in computation, and so we instead work with three isomorphic, disjoint Lagrangians.

Let $L_0$, $L_1$, and $L_2$ be pairwise disjoint cotangent fibers.  A choice of ordering of the Lagrangians propagates into the chain-level product structure, and so we fix a choice: assume that the isotopies taking $L_0$ to $L_1$ and $L_1$ to $L_2$ are positive on the end of $T^*S^1$ that contains Reeb chords corresponding to positive powers of $t$.  There are chain isomorphisms
\[
RC^*(L_0, L_0; H) \simeq RC^*(L_i, L_j; H) \simeq \prod_{r\in\mathbb{Z}}\mathbb{F}_{2}\langle\phi_r\rangle \oplus \bigoplus_{r\in\mathbb{Z}}\mathbb{F}_{2}\langle \tau^r\rangle
\]
for all $0\leq i, j\leq 2$, such that the following diagram commutes
\begin{equation}
\begin{tikzcd}
RC^*(L_0, L_0; H)\otimes RC^*(L_0, L_0; H) \arrow[r, "\mu_2"] \arrow[d, "\simeq"]  & RC^*(L_0, L_0; H) \arrow[d, "\simeq"] \\
RC^*(L_1, L_2; H) \otimes RC^*(L_0, L_1; H) \arrow[r, "\mu_2"] & RC^*(L_0, L_2; H)
\end{tikzcd}
\end{equation}
Define maps
\begin{align*}
\Phi_{ij}: RC^*(L_i, L_j; H) &\rightarrow \mathbb{F}_{2}((t^{-1})) \oplus \mathbb{F}_{2}((t)) & \\
\tau^r &\mapsto (t^r, t^r) \\
\phi_r &\mapsto (0, t^r) &r \geq 0 \\
\phi_r &\mapsto (t^r, 0) &r < 0
\end{align*}
We show that 
\begin{equation}
\label{eq:ts1-commute}
\begin{tikzcd}
RC^*(L_1, L_2; H) \otimes RC^*(L_0, L_1; H) \arrow[r, "\mu_2"] \arrow[d, "\Phi_{21}\otimes\Phi_{10}"] & RC^*(L_0, L_2; H) \arrow[d, "\Phi_{02}"] \\
 \mathbb{F}_{2}((t^{-1})) \oplus \mathbb{F}_{2}((t)) \otimes  \mathbb{F}_{2}((t^{-1})) \oplus \mathbb{F}_{2}((t)) \arrow[r] & \mathbb{F}_{2}((t^{-1})) \oplus \mathbb{F}_{2}((t)) 
\end{tikzcd}
\end{equation}
commutes by analyzing algebraic operations on path spaces of $S^1$.

Define $\mathcal{P}_{ij}$ to be the pathspace on $S^1$ of paths starting at $L_i\cap\{\text{the zero section}\}$ and ending at $L_j\cap\{\text{the zero section}\}$.  We denote by 
\[
C_0(\mathcal{P}_{ij}; \mathbb{F}_{2}) = \bigoplus_{r\in\mathbb{Z}}\mathbb{F}_{2}\langle\gamma_r\rangle
\] 
the degree-zero chains on $\mathcal{P}_{ij}$, with coefficients in $\C$.  We denote its dual by
\[
\hom(C_0(\mathcal{P}_{ij}; \mathbb{F}_{2}), \mathbb{F}_{2}) = \prod_{r\in\mathbb{Z}} \mathbb{F}_{2}\langle\delta_r\rangle,
\]
where
\[
\delta_r(\gamma_s) = \left\{\begin{array}{cc} 1 & s = -r \\ 0 & \text{else}\end{array}\right.
\]
Cieliebak-Hingston-Oancea explored the loopspace analogue of 
\[
\mathcal{C}_{ij} := \hom(C_0(\mathcal{P}_{ij}; \mathbb{F}_{2}), \mathbb{F}_{2})\oplus C_0(\mathcal{P}_{ij}; \mathbb{F}_{2})
\]
in \cite{CHOcoproduct} and showed that it carries the structure of a ring.  Their results extend to path spaces, yielding a map
\[
\mathcal{C}_{jk} \otimes \mathcal{C}_{ij}\rightarrow \mathcal{C}_{ik}.
\]
We describe this product structure on generators.  It relies on the Goresky-Hingston coproduct, which is the decomposition
\begin{align}
C: C_0(\mathcal{P}_{ik}; \mathbb{F}_{2}) &\rightarrow C_0(\mathcal{P}_{jk}; \mathbb{F}_{2})\otimes C_0(\mathcal{P}_{ij}; \mathbb{F}_{2}) \\
\gamma_r &\mapsto \sum_{0 \leq s \leq r} \gamma_s\otimes\gamma_{r - s} & r \geq 0 \\
\gamma_r &\mapsto\sum_{0 > s > r} \gamma_s\otimes\gamma_{r - s} & r < 0
\end{align}

\begin{itemize}
\item The pairing
\[
C_0(\mathcal{P}_{jk}; \mathbb{F}_{2})\otimes C_0(\mathcal{P}_{ij}; \mathbb{F}_{2}) \rightarrow  \hom(C_0(\mathcal{P}_{ik}; \mathbb{F}_{2}), \mathbb{F}_{2})\oplus C_0(\mathcal{P}_{ik}; \mathbb{F}_{2})
\]
is given by the usual product induced by path concatenation, and it lands in $C_0(\mathcal{P}_{ik}; \mathbb{F}_{2})$.

\item The pairing 
\[
 \hom(C_0(\mathcal{P}_{jk}; \mathbb{F}_{2}), \mathbb{F}_{2})\otimes  \hom(C_0(\mathcal{P}_{ij}; \mathbb{F}_{2}), \mathbb{F}_{2})) \rightarrow  \hom(C_0(\mathcal{P}_{ik}; \mathbb{F}_{2}), \mathbb{F}_{2})\oplus C_0(\mathcal{P}_{ik}; \mathbb{F}_{2})
\]
lands in $\hom(C_0(\mathcal{P}_{ik}; \mathbb{F}_{2}), \mathbb{F}_{2})$ and is given by the tensor product
\[
(\delta_r\cdot\delta_s)(\gamma_N) = \delta_r\otimes\delta_s(C(\gamma_N)) = \left\{ \begin{array}{cc} \delta_{r + s} & r, s \geq0 \text{ or } r,s < 0 \\ 0 & \text{else}\end{array}\right.
\]
\item The pairings
\[
C_0(\mathcal{P}_{jk}; \mathbb{F}_{2})\otimes\hom(C_0(\mathcal{P}_{ij}; \mathbb{F}_{2}), \mathbb{F}_{2})) \rightarrow  \hom(C_0(\mathcal{P}_{ik}; \mathbb{F}_{2}), \mathbb{F}_{2})\oplus C_0(\mathcal{P}_{ik}; \mathbb{F}_{2})
\]
and 
\[
\hom(C_0(\mathcal{P}_{jk}; \mathbb{F}_{2}), \mathbb{F}_{2})) \otimes C_0(\mathcal{P}_{ij}; \mathbb{F}_{2})\rightarrow  \hom(C_0(\mathcal{P}_{ik}; \mathbb{F}_{2}), \mathbb{F}_{2})\oplus C_0(\mathcal{P}_{ik}; \mathbb{F}_{2})
\]
are given by the formulas
\begin{align}
t^s\cdot\delta_r &= \phi_r\circ \mu_2(t^s, -) +  \id\otimes\phi_r (C(t^s)) \\ &= \delta_{r+s} + \left\{\begin{array}{cc} t^{s + r} & 0 < -r \leq s \text{ or } 0 \geq -r > s \\ 0 & \text{else}\end{array}\right.
\end{align}
and
\begin{align}
\delta_r\cdot t^s &= \phi_r\circ \mu_2(-, t^s) +  \phi_r\otimes \id (C(t^s)) \\ &= \delta_{r+s} + \left\{\begin{array}{cc} t^{s + r} & 0 < -r \leq s \text{ or } 0 \geq -r > s \\ 0 & \text{else}\end{array}\right.
\end{align}
\end{itemize}

There are isomorphisms
\begin{align*}
CW^*(L_i, L_j; H) &\simeq C_0(\mathcal{P}_{ij}; \mathbb{F}_{2}) \\
\tau^r &\mapsto \gamma_r
\end{align*}
intertwining $\mu_2$ with path concatenation, and dual isomorphisms
\begin{align*}
CW^*(L_i, L_j; -H)[1] &\simeq \hom(C_0(\mathcal{P}_{ij}; \mathbb{F}_{2}), \mathbb{F}_{2}) \\
\phi_r\mapsto \delta_r.
\end{align*}
induced by Poincar\'e duality.  As sketched in \cite{abouzaid_loop} or \cite{CHOcoproduct} the induced correspondence
\[
RC^*(L_i, L_j) \simeq \mathcal{C}_{ij}
\]
is a ring isomorphism.

Showing that (\ref{eq:ts1-commute}) commutes is now a straight-forward algebra computation.
\begin{itemize}
\item \begin{align}
\Phi_{ik}\left(\tau^r\cdot\tau^s\right) &= \Phi_{ik}\left(\tau^{r + s}\right) = (t^{r+s}, t^{r+s}) \\
&= (t^r, t^r)\cdot(t^s, t^s) = \Phi_{jk}(\tau^r)\cdot\Phi_{ij}(\tau^s)
\end{align}
\item \begin{align}
\Phi_{ik}(\phi_r\cdot\phi_s) &= \left\{ \begin{array}{cc} \Phi_{ik}(\phi_{r + s}) & r, s \geq 0 \text{ or } r, s < 0 \\ 0 & \text{else}\end{array}\right. \\
&= \left\{ \begin{array}{cc} (t^{r+s}, 0) & r, s < 0 \\ (0, t^{r + s}) & r, s \geq 0 \\ 0 & \text{else}\end{array}\right. \\
&= \left\{\begin{array}{cc} (t^r, 0)\cdot(t^s, 0) & r, s < 0 \\ (0, t^r)\cdot (0, t^s) & r, s \geq 0 \\ (t^r, 0)\cdot(0, t^s) & \text{else}\end{array}\right. \\
&= \Phi_{jk}(\phi_r)\cdot\Phi_{ij}(\phi_s)
\end{align}
\item \begin{align}
\Phi_{ik}(\phi_r\cdot \tau^s) &= \left\{\begin{array}{cc} \Phi_{ik}(\phi_{r+s} + t^{r+s}) & 0 < -r \leq s \text{ or } 0 \geq -r > s \\ \Phi_{ik}(\phi_{r+s}) & \text{else}\end{array}\right. \\
&= \left\{ \begin{array}{cc} (t^{r+s}, 0) + (t^{r+s}, t^{r+s}) & r \geq 0, s + r < 0 \\
(0, t^{r+s}) + (t^{r+s}, t^{r+s}) & r < 0, s + r \geq 0 \\
(0, t^{r+s}) & r \geq 0, s + r \geq 0 \\
(t^{r+s}, 0) & r < 0, s + r < 0 
\end{array}\right. \\
&= \left\{ \begin{array}{cc} (0, t^{r + s}) & r \geq 0 \\ (t^{r + s}, 0) & r < 0 \end{array}\right. \\
&= \left\{\begin{array}{cc} (0, t^r)\cdot(t^s, t^s) & r \geq 0 \\ (t^r, 0)\cdot(t^s, t^s) & r < 0 \end{array} \right. \\
&= \Phi_{jk}(\phi_r)\cdot\Phi_{ij}(\tau^s).
\end{align}
\end{itemize}

\subsection{Koszul duality}\label{sec:koszul}

The compact Fukaya category $\mathcal{F} = \mathcal{F}(X)$ of $X$ has closed exact Lagrangians of $X$ as objects,
and embeds fully faithfully in the wrapped Fukaya category $\w$.
Every closed exact Lagrangian $K$ is a (both left and right) proper object in the wrapped Fukaya category,
as both wrapped Floer cohomology groups $HW^{*}(K, L)$ and $HW^{*}(L, K)$ are finite-dimensional for any other $L$.
Notice that the Rabinowitz Floer homology groups $RW^{*}(K, L) = H^{*}(RC^{*}(K, L))$ and $RW^{*}(L, K)$ vanish for such a closed exact Lagrangian $K$ by Lemma \ref{compactnessvanishing}. It follows that there is a map from $\mathcal{W}/\mathcal{F}$ to $\rw$, and one can ask how well this map characterizes $\rw$ in terms of $\mathcal{W}$ and $\mathcal{F}$. 
Compact Lagrangians are special cases of (right) proper objects, so one can similarly ask:
do right proper objects of $\w$ also vanish in $\rw$ and if so how far is the
Rabinowitz Fukaya category $\rw$ from the quotient category $\w /
\w_{\r{prop}}$?

Theorem \ref{thm:main} provides an algebraic formula for $\rw$ in terms of $\w$ and thereby reduces the above question to a purely algebraic one which can be studied using Lemma \ref{koszul implies quotient} (due to Efimov). All together one obtains:
\begin{prop}
    If $X$ is non-degenerate then the natural functor $j_{\rw}: \w \to \rw$ \eqref{w to rw} sends $\w_{\r{prop}}$ to zero. If $\w$ further has a Koszul dualizing subcategory, 
then the induced functor $\w / \w_{\r{prop}} \to \rw$ is a quasi-equivalence.
If this Koszul dualizing subcategory is the compact Fukaya category $\mathcal{F}$, then we have a quasi-equivalence
\[
\w / \mathcal{F} \stackrel{\sim}\to  \rw 
\]
\end{prop}
\begin{proof}
Non-degeneracy of $X$ implies smoothness of $\w$.
The Yoneda functor $\bar{y}: \w \to \winf$ by definition sends $\w_{\r{prop}}$ to 0, so the first assertion follows from Corollary \ref{compatibility with w to rw}. Under the Koszul duality hypotheses, 
Lemma \ref{koszul implies quotient} implies that the functor induced by $\bar{y}$ from $\w / \w_{\r{prop}} \to \winf$ is a quasi-equivalence, which by Corollary \ref{compatibility with w to rw} and Theorem \ref{thm:main} implies that the functor induced by $j_{\rw}$ gives the desired quasi-equivalence $\w / \w_{\r{prop}} \cong \rw$.
The final statement follows from the proof of Lemma \ref{koszul implies quotient},
which indeed asserts that $\winf$ is quasi-equivalent to $\w / \mathcal{F}$ under the stated hypotheses.
\end{proof}

\subsection{The open-closed string relationship}

For any Liouville manifold $X$, the open-closed map relates the Hochschild homology of $\w(X)$ to the symplectic cohomology of $X$:
\begin{equation}\label{ocwrapped}
\mathcal{OC}: \r{HH}_{*-n}(\mathcal{W}(X)) \to SH^{*}(X),
\end{equation}
which is an isomorphism when $X$ is non-degenerate \cite{ganatra}.
In general, this is a homomorphism of graded vector space, 
and a homomorphism of $SH^{*}(X)$-modules, 
where we view $\r{HH}_{*-n}(\mathcal{W}(X))$ as a module over $SH^{*}(X)$ via the closed-open map
\begin{equation}
\mathcal{CO}: SH^{*}(X) \to \r{HH}^{*}(\mathcal{W}(X))
\end{equation}
which is a ring homomorphism.
If $X$ is non-degenerate, then we can use the non-compact Calabi-Yau structure to define a product on Hochschild homology $\r{HH}_{*-n}(\mathcal{W}(X))$ 
by identifying it with the Hochschild cohomology $\r{HH}^{*}(\mathcal{W}(X))$,
so that $\mathcal{OC}$ becomes an isomorphism of rings.

The purpose of this subsection is to prove Theorem \ref{aspirationalOC},
which will follow from the main result of the paper Theorem \ref{thm:main} under two additional folklore results.
The first result is a purely algebraic fact, stated in \cite{efimov}. 
To state it, let us recall some notions in homological algebra.
Let $\cc$ be an $\ainf$-category, and $M \in \ob \perf \cc$.
There is a natural functor
\[
\mathbf{1}_{M}: \perf \K \to \perf \cc
\]
sending $\K$ to $M$.
The {\it Chern character} of $M$ is defined to be
\begin{equation}
\r{ch}_{M} = (\mathbf{1}_{M})_{*}(1) \in \r{HH}_{0}(\cc) \cong \r{HH}_{0}(\perf \cc).
\end{equation}
For a perfect $(\cc, \dd)$-bimodule $P$, the Chern character of $P$ is defined similarly to be an element
\begin{equation}\label{chernofbimodule}
\r{ch}_{P} \in \r{HH}_{*}(\cc) \otimes \r{HH}_{*}(\dd).
\end{equation}
of total degree zero defined as follows:
Let $\perf \cc\!-\!\dd$ denote the category of perfect $(\cc, \dd)$-bimodules.
The K\"{unneth} isomorphism for Hochschild homology (along with Morita invariance) says that the Eilenberg-Zilber type map
\begin{equation}
    \r{HH}_{*}(\cc^{op}) \otimes \r{HH}_{*}(\dd) \to \r{HH}_{*}(\perf \cc-\dd),
\end{equation}
induced by the natural functor that takes the $\K$-linear tensor product of a left Yoneda $\cc$-module with a right Yoneda $\dd$-module to get a $(\cc, \dd)$-bimodule, is an isomorphism. The op-invariance property of Hochschild homology further identifies $\r{HH}_*(\cc^{op})$  with $\r{HH}_*(\cc)$. All together these isomorphisms imply that the Chern character of $P$ thought of as an object of $\perf \cc-\dd$ can be viewed as an element of $\r{HH}_{*}(\cc) \otimes \r{HH}_{*}(\dd)$ as in \eqref{chernofbimodule}.
In particular, if $\cc$ is a smooth $\ainf$-category so that its diagonal bimodule $\cc_{\D}$ is perfect,
we have the Chern character of the diagonal bimodule
\[
\r{ch}_{\cc_{\D}} \in \r{HH}_{*}(\cc) \otimes \r{HH}_{*}(\cc),
\]
which is also known as the {\em Shklyarov copairing} \cite{shklyarov} on Hochschild homology
\begin{equation}\label{shkcopairing}
\hat{c}_{Shk}:=\r{ch}_{\cc_{\D}}: \K \to \r{HH}_{*}(\cc) \otimes \r{HH}_{*}(\cc),
\end{equation}
This copairing naturally induces a map
\begin{equation}\label{cherndiagmap}
c_{Shk} = (\r{ch}_{\cc_{\D}})_{*}: \r{HH}_{*}(\cc)^{*} \to \r{HH}_{*}(\cc).
\end{equation}
By choosing an appropriate chain representative of $\r{ch}_{\cc_{\D}}$, we see that the above map \eqref{cherndiagmap} is induced by a chain map, 
denoted by abuse of notation by
\[
c_{Shk} = (\r{ch}_{\cc_{\D}})_{*}: \r{CC}_{*}(\cc)^{*} \to \r{CC}_{*}(\cc).
\]
Now Efimov's result can be stated as the following lemma.
\begin{lem}[Folklore, compare \cite{efimov} p. 37]\label{hhcinf=coneofchern}
For any smooth $\ainf$-category $\cc$, there is a quasi-isomorphism of chain complexes
\begin{equation}
\r{CC}_{*}(\cc, \cinf) \stackrel{\sim}\to \cone(\r{CC}_{*}(\cc)^{*} \stackrel{(c_{Shk})_{*}}\to \r{CC}_{*}(\cc))
\end{equation}
\end{lem}

The second result is a statement concerning the relation between the copairing on symplectic cohomology and the {\it Shklyrarov copairing} on Hochschild homology, which is due to work-in-progress \cite{rezchikov}.
In the closed-string sector, there is a canonical copairing on symplectic cochain complex similar to the copairing \eqref{copairing} on wrapped Floer cochain complex,
\begin{equation}
\hat{c}_{SH}: \K \to SC^{*}(X) \otimes SC^{2n-*}(X),
\end{equation}
inducing a copairing on symplectic cohomology
\begin{equation}\label{SHcopairing}
\hat{c}_{SH}: \K \to SH^{*}(X) \otimes SH^{2n-*}(X).
\end{equation}
By analogy with the construction in \S\ref{section:rc}, this is precisely the copairing inducing the continuation map from symplectic chains to cochains
\begin{equation}
c_{SH}: SC_{*}(X) \to SC^{*}(X),
\end{equation}
where by Poincar\'{e} duality $SC_{*}(X)$ is quasi-isomorphic to $SC^{2n-*}(X)^{*}$,
which can be seen by a closed-string analogue of the isomorphism $\bar{I}$ \eqref{PD} for wrapped Floer chain and cochain complexes.
In particular, we have an quasi-isomorphism of chain complexes
\begin{equation}\label{conetoRFH}
\cone(SC_{*}(X) \stackrel{c_{SH}}\to SC^{*}(X)) \stackrel{\sim}\to RFC^{*}(X)
\end{equation}
whose homology is the Rabinowitz Floer homology $RFH^{*}(X)$; this quasi-isomorphism is the main result of \cite{CFO} as mentioned in the introduction.

Rezschikov's result asserts the following relation between the two copairings $\hat{c}_{SH}$ \eqref{SHcopairing} and $\hat{c}_{Shk}$ \eqref{shkcopairing}.

\begin{lem}[In progress, \cite{rezchikov}]\label{OCtransfercopairing}
    Suppose the open-closed map $\mathcal{OC}$ \eqref{ocwrapped} is an isomorphism. Then it transfers $\hat{c}_{SH}$ to $\hat{c}_{Shk}$.
\end{lem}

\begin{proof}[Proof of Theorem \ref{aspirationalOC}]
By Theorem \ref{thm:main}, we have an isomorphism
\begin{equation}
\r{HH}_{*-n}(\w, \rw) \cong \r{HH}_{*-n}(\w, \winf),
\end{equation}
which by Lemma \ref{hhcinf=coneofchern} is further isomorphic to 
\begin{equation}\label{coneofshklyarov}
H^{*}(\cone(\r{CC}_{*}(\w)^{*} \stackrel{(c_{Shk})_{*}}\to \r{CC}_{*}(\w))).
\end{equation}
Note that $\r{HH}_{*-n}(\w, \w)^{*}$ is isomorphic to $SH^{*}(X)^{*}[-2n] = SH^{2n-*}(X)^{*}$ by applying the linear dual of the open-closed isomorphism $\mathcal{OC}$ \eqref{ocwrapped}.
Under the open-closed isomorphism $\mathcal{OC}$, and by Lemma \ref{OCtransfercopairing} where we identify $SH^{2n-*}(X)^{*}$ with $SH_{*}(X)$ by Poincar\'{e} duality,
the above group \eqref{coneofshklyarov} is isomorphic to
\[
H^{*}(\cone(SC_{*}(X) \stackrel{c_{SH}}\to SC^{*}(X)),
\]
which is isomorphic to Rabinowitz Floer homology $RFH^{*}(X)$ by the quasi-isomorphism \eqref{conetoRFH}, as desired.
\end{proof}

This theorem allows us to give a superficially different proof of the following interesting ``non-proper or zero'' property of wrapped Fukaya categories recently established in \cite{Gnonproper}, which is related to the conjectural `zero or infinite' property of symplectic cohomology for Liouville manifolds:
\begin{cor}[compare \cite{Gnonproper} Main Theorem part (1)]
For any non-degenerate Liouville manifold $(X, \lambda)$, its wrapped Fukaya category $\w:=\w(X)$ is either zero or non-proper. 
In particular if $\w(X) \neq 0$ then given any finite collection $\{L_{i}\}$ of Lagrangians split-generating the wrapped Fukaya category, 
there exist $i, j$ such that the wrapped Floer cohomology $HW^{*}(L_{i}, L_{j})$ is infinite-dimensional.
\end{cor}
\begin{proof}
If $\w$ is proper, then $\widehat{\w}_{\infty}$ is zero so Theorem \ref{thm:main} implies $\rw = 0$, hence $\r{HH}_{*}(\w, \rw) = 0$.
By Theorem \ref{aspirationalOC}, this implies $RFH^{*}(X) = 0$.
By \cite{ritterTQFT}*{Thm 13.1}, $SH^{*}(X) = 0$, which implies (by the usual argument that $HW^*(L,L)$ is a unital module over $SH^*(X)$ for each $L$) that $\w$ must be zero.
\end{proof}

\bibliography{rabinowitz}

\end{document}